\documentclass[platex]{amsart}
\usepackage{amssymb}
\usepackage{braket}
\usepackage{mathrsfs}
\usepackage{amsrefs}
\usepackage{ifthen}
\usepackage{graphicx}
\usepackage{tikz}
\usetikzlibrary{patterns}
\usepackage{comment} 
\usepackage{mleftright}
\usepackage{cleveref}
 \usepackage[all]{xy}

\theoremstyle{plain}
\newtheorem{theo}{Theorem}[section]
\crefname{theo}{Theorem}{Theorems}
\Crefname{theo}{Theorem}{Theorems}
\newtheorem{prop}[theo]{Proposition}
\crefname{prop}{Proposition}{Propositions}
\Crefname{prop}{Proposition}{Propositions}
\newtheorem{lem}[theo]{Lemma}
\crefname{lem}{Lemma}{Lemmas}
\Crefname{lem}{Lemma}{Lemmas}
\newtheorem{cor}[theo]{Corollary}
\crefname{cor}{Corollary}{Corollaries}
\Crefname{cor}{Corollary}{Corollaries}

\crefname{claim}{Claim}{Claims}
\Crefname{claim}{Claim}{Claims}

\crefname{property}{Property}{Properties}
\Crefname{property}{Property}{Properties}

\crefname{problem}{Problem}{Problems}
\Crefname{problem}{Problem}{Problems}

\theoremstyle{definition}
\newtheorem{defi}[theo]{Definition}
\crefname{defi}{Definition}{Definitions}
\Crefname{defi}{Definition}{Definitions}

\crefname{notation}{Notation}{Notations}
\Crefname{notation}{Notation}{Notations}

\crefname{convention}{Convention}{Conventions}
\Crefname{convention}{Convention}{Conventions}

\crefname{cond}{Condition}{Conditions}
\Crefname{cond}{Condition}{Conditions}
\newtheorem{assum}[theo]{Assumption}
\crefname{assum}{Assumption}{Assumptions}
\Crefname{assum}{Assumption}{Assumptions}

\theoremstyle{remark}
\newtheorem{rem}[theo]{Remark}
\crefname{rem}{Remark}{Remarks}
\Crefname{rem}{Remark}{Remarks}
\newtheorem{ex}[theo]{Example}
\crefname{ex}{Example}{Examples}
\Crefname{ex}{Example}{Examples}

\crefname{section}{Section}{Sections}
\Crefname{section}{Section}{Sections}
\crefname{subsection}{Subsection}{Subsections}
\Crefname{subsection}{Subsection}{Subsections}
\crefname{figure}{Figure}{Figures}
\Crefname{figure}{Figure}{Figures}

\newtheorem*{acknowledgement}{Acknowledgement}

\newcommand{\Z}{\mathbb{Z}}
\newcommand{\N}{\mathbb{N}}
\newcommand{\R}{\mathbb{R}}

\newcommand{\Q}{\mathbb{Q}}
\newcommand{\CP}{\mathbb{CP}}

\newcommand{\fiber}{\mathrm{fiber}}
\newcommand{\base}{\mathrm{base}}
\newcommand{\pt}{\mathrm{pt}}
\newcommand{\cpt}{\mathrm{cpt}}
\newcommand{\op}{\mathrm{op}}
\newcommand{\even}{\mathrm{even}}
\newcommand{\odd}{\mathrm{odd}}
\newcommand{\spc}{\mathrm{spin}^c}
\newcommand{\spcGL}{\mathrm{spin}^c_{GL}}
\newcommand{\tGL}{\widetilde{GL}}
\newcommand{\FrGL}{\mathrm{Fr}_{GL}}
\newcommand{\Area}{\mathrm{Area}}

\newcommand{\Shrink}{\mathrm{Shrink}}
\newcommand{\sign}{\mathrm{sign}}

\newcommand{\conn}{{\mathscr A}}
\newcommand{\gauge}{{\mathscr G}}

\newcommand{\M}{{\mathcal M}}

\newcommand{\vnU}{{\mathcal U}}

\newcommand{\calM}{{\mathcal M}}
\newcommand{\calE}{{\mathcal E}}
\newcommand{\calO}{{\mathcal O}}
\newcommand{\calU}{{\mathcal U}}
\newcommand{\calS}{{\mathcal S}}

\newcommand{\calV}{{\mathcal V}}

\newcommand{\calL}{{\mathcal L}}
\newcommand{\calH}{{\mathcal H}}
\newcommand{\calW}{{\mathcal W}}
\newcommand{\scrB}{{\mathscr B}}
\newcommand{\scrC}{{\mathscr C}}
\newcommand{\scrD}{{\mathscr D}}
\newcommand{\scrE}{{\mathscr E}}
\newcommand{\scrH}{{\mathscr H}}
\newcommand{\scrX}{{\mathscr X}}
\newcommand{\frakm}{\mathfrak{m}}
\newcommand{\frakM}{\mathfrak{M}}

\newcommand{\fraks}{\mathfrak{s}}
\newcommand{\frakt}{\mathfrak{t}}

\newcommand{\frakP}{\mathfrak{P}}

\newcommand{\Met}{\mathrm{Met}}
\newcommand{\circPi}{\mathring{\Pi}}
\newcommand{\PSC}{\mathrm{PSC}}
\newcommand{\ad}{\mathrm{ad}}

\newcommand{\Diff}{\mathrm{Diff}}
\newcommand{\Homeo}{\mathrm{Homeo}}
\newcommand{\Aut}{\mathrm{Aut}}

\newcommand{\Map}{\mathrm{Map}}
\newcommand{\Si}{\Sigma}
\newcommand{\si}{\sigma}
\newcommand{\vp}{\varphi}
\newcommand{\inc}{\hookrightarrow}

\newcommand{\lb}{\left<}
\newcommand{\rb}{\right>}
\newcommand{\bt}{{\bf t}}
\newcommand{\bs}{{\bf s}}
\newcommand{\del}{\partial}
\newcommand{\codel}{\delta}
\newcommand{\Ker}{\mathop{\mathrm{Ker}}\nolimits}
\newcommand{\Coker}{\mathop{\mathrm{Coker}}\nolimits}
\newcommand{\im}{\mathop{\mathrm{Im}}\nolimits}

\newcommand{\Hom}{\mathop{\mathrm{Hom}}\nolimits}

\newcommand{\PD}{\mathop{\mathrm{P.D.}}\nolimits}
\newcommand{\Rest}{\mathop{\mathrm{Rest}}\nolimits}

\newcommand{\Repa}{\mathop{\mathrm{Repa}}\nolimits}
\newcommand{\SWinv}{\mathop{\mathrm{SW}}\nolimits}
\newcommand{\Dcoh}{\mathop{{\mathbb D}}\nolimits}
\newcommand{\SWcoh}{\mathop{{\mathbb S}{\mathbb W}}\nolimits}
\newcommand{\SWcoch}[1][{{}}]{\mathop{{\mathcal S}{\mathcal W}_#1}\nolimits}
\newcommand{\Acoh}{\mathbb{A}}
\newcommand{\Acoch}[1][{{}}]{\mathop{{\mathcal A}_#1}\nolimits}
\newcommand{\tilA}{\mathop{\tilde{\mathcal A}}\nolimits}

\newcommand{\tilg}{\tilde{g}}
\newcommand{\tilalpha}{\tilde{\alpha}}
\newcommand{\SW}{Seiberg--Witten }
\newcommand{\vn}{virtual neighborhood }
\newcommand{\id}{\mathrm{id}}
\newcommand{\ind}{\mathop{\mathrm{ind}}\nolimits}

\newcommand{\reflect}[2]{{\overline{#1}}^{#2}}

\newcommand{\Euler}{\mathop{\chi}\nolimits}

\title[Characteristic classes via gauge theory]{Characteristic classes via 4-dimensional gauge theory}

\author{Hokuto Konno}

\address{Graduate School of Mathematical Sciences, the University of Tokyo, 3-8-1 Komaba, Meguro, Tokyo 153-8914, Japan}

\email{hkonno@ms.u-tokyo.ac.jp}

\date{}

\begin{document}

\maketitle

\begin{abstract}
We construct characteristic classes of 4-manifold bundles using $SO(3)$-Yang--Mills theory and Seiberg--Witten theory for families.
\end{abstract}

\tableofcontents

\section{Introduction}

The aim of this paper is to construct characteristic classes for bundles whose fiber is a $4$-manifold using gauge theory.
Our construction can be regarded as an infinite dimensional analogue of that of familiar characteristic classes typified by the Mumford--Morita--Miller classes~\cite{MR857372,MR914849,MR717614} for surface bundles.
To explain it, let us recall the definition of the Mumford--Morita--Miller classes.
From a topological point of view, the definition of these classes includes the following two steps:

\begin{description}
\item[(1)] Consider a ``linearization" of a given surface bundle, more precisely, the vector bundle consisting of tangent vectors along the fibers.
\item[(2)] Consider the Euler class of the linearization.
\end{description}
The Mumford--Morita--Miller classes are defined as the fiber integration of cup products of the Euler class in {\bf (2)}.
On the other hand, there are the following infinite dimensional analogues of {\bf (1)} and {\bf (2)}: 
\begin{description}
\item[(I)] Consider a functional space on a given a non-linear object, say a manifold.
\item[(II)] Consider the zero set of a given Fredholm section of a Hilbert bundle.
\end{description}
The reason why {\bf (I)} and {\bf (II)} correspond to {\bf (1)} and {\bf (2)} can be explained as follows.
On {\bf (I)}, a functional space on a given manifold is also regarded as a ``linearization" of the manifold.
For example, in representation theory, given a manifold acted by a Lie group, the induced infinite dimensional representation on a functional space on the manifold can be regarded as a linearization of the original non-linear action.
We also note that, if the non-linear object is given as a fiber bundle, one can consider a family of functional spaces as its linearization. 
On {\bf (II)}, suppose that we are given an infinite dimensional Hilbert bundle on a Hilbert manifold defined as some functional spaces on a finite dimensional manifold and that we are also given a Fredholm section of it.
The zero set of the section morally corresponds to the ``Euler class" of the infinite dimensional bundle.
We recall that {\bf (II)} is one of the main ideas to define the invariants of $4$-manifolds emerging from gauge theory, say the Donaldson invariant~\cite{MR1066174} and the \SW invariant~\cite{MR1306021}.
In this paper, using gauge theory, we shall construct characteristic classes for bundles of $4$-manifolds via the above analogy between  {\bf (1)}, {\bf (2)} and {\bf (I)}, {\bf (II)}.

The initial roots of this work have been given by S.~K.~Donaldson~\cite{MR1171888,MR1339810}.
He has suggested the idea of characteristic classes based on family gauge theory, however there are several problems to justify the idea.
The first one is the natural structure groups which are compatible with gauge theoretic equations are not subgroups of the diffeomorphism group of the $4$-manifold given as a fiber.
(For example, see M.~Szymik~\cite{MR2652709}.)
The second problem arises if we want to establish the universal theory:
the base space of the universal bundle of the diffeomorphism group, namely the classifying space of the diffeomorphism group, is not a finite dimensional smooth manifold.
Since the usual family gauge theory is considered on smooth finite dimensional parameter spaces, it cannot be applied to this situation.
In this paper we use an idea given by N.~Nakamura~\cite{MR2644908} and a family version of Y.~Ruan's virtual neighborhood technique~\cite{MR1635698} to overcome these difficulties of the construction.
In general, even if we could succeed to construct a characteristic class of bundles of smooth manifolds, it is a difficult problem to prove that the characteristic class is non-trivial.
(For example, in the case of the Mumford--Morita--Miller classes, the proof of the non-triviality is one of difficult parts of S.~Morita~\cite{MR914849}.)
However, for our characteristic classes, we can use D.~Ruberman's results~\cite{MR1671187} and author's~\cite{Konno3} on invariants of diffeomorphisms, and they ensure that our characteristic classes are non-trivial.
We also give other calculations in \cref{section Non triviality}.

We now explain our main results.
We can construct our theory of characteristic classes using both the $SO(3)$-Yang--Mills anti-self-dual  (ASD) equation and the \SW equations.
Let $X$ be an oriented closed smooth $4$-manifold, $\Diff^{+}(X)$ be the group of orientation preserving diffeomorphisms on $X$, $\frakP$ be the isomorphism class of an $SO(3)$-bundle on $X$ with $w_{2}(\frakP) \neq 0$,
and $\fraks$ be the isomorphism class of a $\spc$ structure on $X$.
(Although an $SO(3)$-bundle (or a $\spc$ structure) and the isomorphism class of it are often denoted by a same notation, we explicitly distinguish them in this paper.)
We define subgroups $\Diff(X, \frakP)$ and $\Diff(X, \fraks)$ of $\Diff^{+}(X)$ by
\begin{align}
&\Diff(X, \frakP) := \Set{f \in \Diff^{+}(X) | f^{\ast}\frakP = \frakP}, \text{ and}
\label{str grp ASD intro}\\
&\Diff(X, \fraks) := \Set{ f \in \Diff^{+}(X) | f^{\ast} \fraks = \fraks}.
\label{str grp SW intro}
\end{align}
Let $\calO$ be a homology orientation of $X$ (i.e. an orientation of the vector space $H^{1}(X; \R) \oplus H^{+}(X;\R)$, where $H^{+}(X;\R)$ is a maximal positive definite subspace of $H^{2}(X;\R)$ with respect to the intersection form of $X$).
We define
\begin{align}
&\Diff(X, \frakP, \calO) := \Set{ f \in \Diff(X,\frakP) | f^{\ast} \calO = \calO}, \text{ and}
\label{str grp ASD O intro}\\
&\Diff(X, \fraks, \calO) := \Set{ f \in \Diff(X,\fraks) | f^{\ast} \calO = \calO}.
\label{str grp SW O intro}
\end{align}
We denote by $d(\frakP)$ and $d(\fraks)$ the formal dimension of the moduli space of the $SO(3)$-ASD equations for $\frakP$ and the \SW equations for $\fraks$ respectively: given by the formula $d(\frakP) = -2p_{1}(\frakP) - 3(1 - b_{1}(X) + b^{+}(X))$ and $d(\fraks)= (c_1(\fraks)^2 - 2\Euler(X) -3\sign(X))/4$.
We shall construct characteristic classes of bundles of a $4$-manifold $X$ with structure group $\Diff(X, \frakP)$, $\Diff(X, \frakP, \calO)$ and $\Diff(X,\fraks)$, $\Diff(X,\fraks,\calO)$ using families of ASD equations and \SW equations respectively, summarized as follows:

\begin{theo}
\label{the theorem of final aim}
Let $n$ be a non-negative integer satisfying that $b^{+}(X) \geq n+2$.
\begin{enumerate}
\item
\label{the theorem of final aim Donaldson}
Assume that $d(\frakP) = -n$.
Then we can give cohomology classes
\[
\Dcoh(X,\frakP) \in H^{n}(B\Diff(X,\frakP);\Z/2)
\]
and
\[
\Dcoh(X,\frakP,\calO) \in H^{n}(B\Diff(X,\frakP,\calO);\Z)
\]
via $n$-parameter families of $SO(3)$-Yang--Mills ASD equations.
These cohomology classes depend only on the pair $(X,\frakP)$ and the triple $(X,\frakP,\calO)$ respectively.
\item
\label{the theorem of final aim SW}
Assume that $d(\fraks) = -n$.
Then we can give cohomology classes
\[
\SWcoh(X,\fraks) \in H^{n}(B\Diff(X,\fraks);\Z/2)
\]
and
\[
\SWcoh(X,\fraks,\calO) \in H^{n}(B\Diff(X,\fraks,\calO);\Z)
\]
via $n$-parameter families of \SW equations.
These cohomology classes depend only on the pair $(X,\fraks)$ and the triple $(X,\fraks,\calO)$ respectively.
\end{enumerate}
\end{theo}

See \cref{defi value for universal bundles} for the corresponding statement.
The cases that $n=0$ in \eqref{the theorem of final aim Donaldson} and \eqref{the theorem of final aim SW} of \cref{the theorem of final aim} are nothing other than the mod $2$ and the usual $SO(3)$-Donaldson invariants of $(X, \frakP)$ and \SW invariants of $(X, \fraks)$ respectively, via $H^{0}(B\Diff(X,\frakP);\Z/2) \cong \Z/2$ and $H^{0}(B\Diff(X,\frakP,\calO);\Z) \cong \Z$, and similar isomorphisms for $\fraks$.

As usual, \cref{the theorem of final aim} can be rephrased as follows.
For $n \geq 0$ and a CW complex $B$, if we are given a bundle $X \to E \to B$ whose structure group $G$ is given as \eqref{str grp ASD intro}, \eqref{str grp SW intro}, \eqref{str grp ASD O intro}, or \eqref{str grp SW O intro}, we can functorially define a cohomology class $\Dcoh(E) \in H^{n}(B)$ or $\SWcoh(E) \in H^{n}(B)$ under suitable assumptions on $b^{+}$ and the formal dimension.
This cohomology class is given as the pull-back of the universal characteristic class given in \cref{the theorem of final aim}.
(In fact our characteristic classes are defined even if $B$ is a general topological space.
See \cref{Dcoh SWcoh in the base case} for the summarized statements.)
If $n>0$ and the characteristic class $\Dcoh$ or $\SWcoh$ of $E$ is non-trivial, then we can see that $E$ is a non-trivial $G$-bundle, explained in \cref{rem on relation between vanishing and triviality}.
We mention that non-vanishing of our characteristic classes implies indecomposability as fiberwise connected sum under certain conditions, and $\SWcoh$ also obstructs the existence of a family of positive scalar curvature metrics, described in \cref{section Characteristic classes as obstruction}.

Since it is hard to compute gauge theoretic invariants in general, one may wonder computability of our characteristic classes.
In particular, it is natural to ask whether there are non-trivial examples for a positive degree $n>0$.
In \cref{section Non triviality} we shall give several non-trivial calculations of our characteristic classes for $n>0$.
As well as the usual Donaldson invariants and \SW invariants, $\SWcoh$ is relatively easy to compute rather than $\Dcoh$.
We here note that, if $X$ is simply connected and $\fraks$ is the isomorphism class of a $\spc$ structure coming from a spin structure on $X$, then we have $\Diff(X, \fraks) = \Diff^{+}(X)$, and thus we can get a cohomology class on $B\Diff^{+}(X)$ by considering $\SWcoh$.
For example, denoting by $\fraks$ the isomorphism class of a $\spc$ structure coming from a spin structure on $K3 \# n(S^{2} \times S^{2})$ for a positive integer $n$,
we see that
\[
\SWcoh(K3 \# n(S^{2} \times S^{2}),\fraks) \neq 0 \text{ in } H^{n}(B\Diff^{+}(K3 \# n(S^{2} \times S^{2}));\Z/2),
\]
which is explained in \cref{non vanishing for K3}.
We also mention that $\SWcoh$ can detect a difference between $\Homeo(X, \fraks)$-bundles and $\Diff(X, \fraks)$-bundles.
Here $\Homeo(X, \fraks)$ is the group of homeomorphisms which preserve the orientation and  $\fraks$, as in the definition of $\Diff(X, \fraks)$.
For example, for some natural numbers $k,l>0$, we can find the isomorphism class $\fraks$ of a $\spc$ structure on $X = k\CP^{2}\#l(-\CP^{2})$ and a bundle $X \to E \to S^{1}$ with structure group $\Diff(X, \fraks)$ such that $E$ is trivial as a $\Homeo(X, \fraks)$-bundle, but non-trivial as a $\Diff(X, \fraks)$-bundle, deduced by showing $\SWcoh(E) \neq 0$.
See~\cref{ex comp between fiberwise connected sum and connected sum concrete}.
This argument is a \SW analogue of one in D.~Ruberman~\cite{MR1671187}.

We note that family gauge theory has been studied by several people:
\cite{MR1671187,MR1734421,MR1874146,MR1868921,MR2015245,MR2644908,MR2652709,Kronheimer1,Konno1,Konno2,Konno3}.
The relation between some of these works and this paper is as follows.
(See also \cref{section Concluding remarks}.)
D.~Ruberman~\cite{MR1671187, MR1734421, MR1874146} has given invariants of diffeomorphisms, and the author generalized a part of \cite{MR1671187} to a tuple of commutative diffeomorphisms in \cite{Konno3}.
Our characteristic classes can be regarded as a generalization of the invariants of (tuples of) diffeomorphisms given in \cite{MR1671187,Konno3}.
T.-J.~Li and A.-K.~Liu~\cite{MR1868921} and M.~Szymik~\cite{MR2652709} have considered families of $\spc$ $4$-manifolds and \SW equations.
In addition, Szymik has considered characteristic cohomotopy classes for bundles of $\spc$ $4$-manifolds using a finite dimensional approximation of \SW equations, in other words, family version of Bauer--Furuta invariant~\cite{MR2025298}.
(See \cref{rem on cohomotopy version}.)
One of the big differences between theories of \cite{MR1868921,MR2652709} and our theory is the structure group: 
the structure group of their family is the automorphism group of a given $\spc$ $4$-manifold.
This group does {\it not} coincide with our structure group $\Diff(X, \fraks)$ (or $\Diff(X, \fraks, \calO)$).
This is because $\Diff(X, \fraks)$ respects only the {\it isomorphism class} of a $\spc$ structure.
As we mentioned, this is the first difficulty to construct a theory of characteristic classes whose structure group is (a subgroup of) the diffeomorphism group, and we use an idea of Nakamura~\cite{MR2644908} to deal with it.
In \cref{section Non triviality}, to calculate our characteristic classes we consider a higher-dimensional generalization of an argument given in D.~Ruberman~\cite{MR1671187, MR1734421, MR1874146} and the idea due to N.~Nakamura~\cite{MR2644908} on double mapping tori as in \cite{Konno3}.
Kronheimer~\cite{Kronheimer1} has considered a homological formulation of family gauge theory to study the space of symplectic forms of a $4$-manifold, and the author~\cite{Konno2} has considered a cohomological formulation of family gauge theory to study the adjunction inequality.
One of the important differences between \cite{Kronheimer1,Konno2} and the current paper is a type of families.
Families of perturbations or Riemannian metrics of a fixed $4$-manifold are considered in \cite{Kronheimer1,Konno2}, and they are therefore trivial families as a family of $4$-manifolds.
On the other hand, in this paper we consider families of $4$-manifolds themselves, and our characteristic classes can detect the non-triviality as families of $4$-manifolds, as we explained. 

We also note that, in symplectic geometry, several people have studied characteristic classes of symplectic fibrations based on the study of a family version of Gromov--Witten invariant.
(See H.-V.~L\^e and K.~Ono~\cite{MR2509705}, O.~Bu\c se~\cite{MR2218350}, and D.~McDuff~\cite{MR2396906}.)
Our theory of characteristic classes may be regarded as the gauge theoretic counterpart of their work.

The author hopes that this paper will be the first of a series of a cohomological study of family gauge theory, which was also considered in \cite{Kronheimer1} and \cite{Konno2} from different points of view as we explained.
Further potential developments hoped to be given in subsequent papers include a Bauer--Furuta version of our characteristic classes, calculations relating to algebraic geometry, and Floer homology analogue of this paper.
(See \cref{section Concluding remarks} for details of some of them.)

\begin{acknowledgement}
The author would like to express his deep gratitude to Mikio Furuta for the helpful suggestions and for continuous encouragement during this work.
The author would also like to express his appreciation to Yosuke Morita for asking him about fiberwise connected sum formula.
It induces the author to consider \cref{theo vanishing fiber wiseconnected sum}.
The author also wishes to thank David Baraglia for pointing out some mistakes and giving comments for a preprint version of this paper.
The author was supported by JSPS KAKENHI Grant Number 16J05569 and
the Program for Leading Graduate Schools, MEXT, Japan.
\end{acknowledgement}

\section{Structure groups}
\label{section Structure groups}

In this section we define the structure groups for our characteristic classes and some related groups.
We first consider Yang--Mills theory.
To construct a theory of characteristic classes fruitfully, we need some compactness, corresponding to \cref{compactness assumption for unparameterized case,assumption on compactness for parameterized moduli}.
We therefore consider $SO(3)$-Yang--Mills ASD equations rather than $SU(2)$-Yang--Mills ASD equations.
Let $X$ be an oriented smooth $4$-manifold, $P \to X$ be an $SO(3)$-bundle, and $\frakP$ be the isomorphism class of $P$.
We have already defined the groups $\Diff(X, \frakP)$ and $\Diff(X, \frakP, \calO)$ in the introduction.
We here define
\[
\Aut(X, P)
:= \Set{
	(f, \tilde{f}) |
		\begin{matrix}
			f \in \Diff(X, \frakP),\\
			\tilde{f} : P \to P \text{ is an isomorphism}\\
			\text{between $SO(3)$-bundles covering $f$.}
		\end{matrix}
		}.
\]
This group $\Aut(X, P)$ is the automorphism group of the pair $(X, P)$ in the category of pairs of a $4$-manifold and an $SO(3)$-bundle on it.
Let $\calO$ be a homology orientation of $X$, and define
\[
\Aut(X, P, \calO) := \Set{ (f, \tilde{f}) \in \Aut(X,P) | f \in \Diff(X, \frakP, \calO)}.
\]
Let $\gauge_{P}$ be the gauge group of $P \to X$.
Then we have the following exact sequences
\[
1 \to \gauge_{P} \to \Aut(X, P) \to \Diff(X, \frakP) \to 1
\]
and 
\[
1 \to \gauge_{P} \to \Aut(X, P, \calO) \to \Diff(X, \frakP, \calO) \to 1.
\]

We next consider \SW theory.
To avoid taking care with Riemannian metrics, we use {\it spin$^{c}_{GL}$ structure} as in \cite{Konno3} rather than $\spc$ structure.
Let us review the definition of spin$^{c}_{GL}$ structure.
Fix a double covering $GL^{+}_{4}(\R)$ of $GL^{+}_{4}(\R)$, where $GL^{+}_{4}(\R)$ is the group of invertible real $4\times4$-matrices with $\det>0$.
Set 
\[
Spin^{c}_{GL}(4) := (\tGL^{+}_{4}(\R) \times U(1))/\pm1.
\]
We have the natural map $Spin^{c}_{GL}(4) \to GL_{4}^{+}(\R)$ as the map $Spin^{c}(4) \to SO(4)$.
We denote by $\FrGL(X) \to X$ the frame bundle whose fiber at $x \in X$ is the set of oriented frames of $T_{x}X$.
We define a {\it $spin^{c}_{GL}$ structure} on $X$ as a pair $s = (P_{GL}, \psi)$ consisting of a $Spin^{c}_{GL}(4)$-bundle $P_{GL} \to X$ and an isomorpshism
\[
\psi : P_{GL} \times_{Spin^{c}_{GL}(4)} GL^{+}_{4}(\R) \to \FrGL(X)
\]
between $GL^{+}_{4}(\R)$-bundles.
Given a $\spcGL$ structure $s$, a $\spc$ structure $s_{g}$ is induced corresponding to each Riemannian metric $g$ on $X$.
For a fixed metric, an isomorphism class of $\spcGL$ structures corresponds one-to-one with an isomorphism class of $\spc$ structures.
We do not therefore distinguish $\spcGL$ structure from $\spc$ structure when we consider them at the level of isomorphism classes.

Let $s = (P_{GL}, \psi)$ be a $\spcGL$ structure on $X$, $\fraks$ be the isomorpshim class of $s$, and $\calO$ be a homology orientation of $X$.
We have already defined the groups $\Diff(X, \fraks)$ and $\Diff(X, \fraks, \calO)$ in the introduction.
The automorphism group of the pair $(X, s)$ in the category of $\spcGL$ $4$-manifolds is given as
\[
\Aut(X, s)
:= \Set{
	(f, \tilde{f}) |
		\begin{matrix}
			f \in \Diff(X, \fraks),\\
			\tilde{f} : P_{GL} \to P_{GL} \text{ is an isomorphism between}\\
			\text{$Spin^{c}_{GL}(4)$-bundles such that the diagram } \eqref{diagram for def of Aut} \text{ commutes}.
		\end{matrix}
		},
\]
where \eqref{diagram for def of Aut} is given by
\begin{align}
\xymatrix{
 		P_{GL} \ar[d] \ar[r]^{\tilde{f}} & P_{GL} \ar[d] \\
		P_{GL} \times_{Spin^{c}_{GL}(4)} GL^{+}_{4}(\R) \ar[d]_{\psi} & P_{GL} \times_{Spin^{c}_{GL}(4)} GL^{+}_{4}(\R)\ar[d]^{\psi} \\
		\FrGL(X) \ar[r]_{df} & \FrGL(X).
		 }
\label{diagram for def of Aut}
\end{align}
We also define 
\[
\Aut(X, s, \calO) := \Set{ (f, \tilde{f}) \in \Aut(X,s) | f \in \Diff(X, \fraks, \calO)}.
\]
For the gauge group $\gauge_{s}$ of the $\spcGL$ structure $s$, we have exact sequences
\[
1 \to \gauge_{s} \to \Aut(X, s) \to \Diff(X, \fraks) \to 1
\]
and 
\[
1 \to \gauge_{s} \to \Aut(X, s, \calO) \to \Diff(X, \fraks, \calO) \to 1.
\]
For a given metric $g$ on $X$, it is easy to see that the gauge group of $s$  is isomorphic to that of the induced $\spc$ structure $s_{g}$; both gauge groups are isomorphic to $\Map(X,S^{1})$.

\section{ASD setting and SW setting}
\label{subsection ASD setting and SW setting}

Most of the story of this paper is common to ASD and \SW equations.
We present here a summary of settings in which we work for both equations.
Let $X$ be an oriented closed smooth $4$-manifold and $B$ be a topological space.
The first setting is on ASD equations, to which we refer as the {\it ASD setting}:
\begin{itemize}
\item  An $SO(3)$-bundle $P \to X$ with $w_{2}(P) \neq 0$ is given.
Let $\frakP$ be the isomorphism class of $P$.
\item We write $\gauge$ for the gauge group $\gauge_{P}$ of $P$ which is $L^{2}_{k+1}$-completed for $k > 1$.
\item The groups $G$ and $\tilde{G}$ are either one of the following \eqref{candidate of tG and G in ASD no homology ori} and \eqref{candidate of tG and G in ASD with homology ori}:
\begin{enumerate}
\item $G = \Diff(X, \frakP)$ and $\tilde{G} = \Aut(X, P)$. \label{candidate of tG and G in ASD no homology ori}
\item $G = \Diff(X, \frakP, \calO)$ and $\tilde{G} = \Aut(X, P, \calO)$ for a given homology orientation $\calO$. \label{candidate of tG and G in ASD with homology ori}
\end{enumerate}
\item A continuous fiber bundle $X \to E \to B$ with structure group $G$ is given.
\item For the single $4$-manifold $X$, define $\Pi(X) := \Met(X)$, where $\Met(X)$ is the space of Riemannian metrics on $X$.
We also define $\Pi(E) := \bigsqcup_{b \in B} \Pi(E_{b})$.
\item For each $g \in \Pi(X)$, define
\[
\scrC_{g} := \conn_{L^{2}_{k}}(P),\quad \scrD_{g}:= L^{2}_{k-1}(\Lambda^{+}_{g}(X) \otimes \ad{P}),
\]
where $\conn_{L^{2}_{k}}(P)$ is the space of $SO(3)$-connections on $P$ which is $L^{2}_{k}$-completed,
$\Lambda^{+}_{g}(X)$ is the $g$-self-dual part of $\Lambda^{2}T^{\ast}X$, $\ad{P} \to X$ is the adjoint bundle associated with $P$ with fiber $\mathfrak{so}(3)$, and $L^{2}_{k-1}(\cdot)$ means the completion.
In fact $\scrC_{g}$ is independent of $g$ in this case, we use this notation.
Let $\scrC_{g}^{\ast}$ be the subset of $\scrC_{g}$ consisting of reducible connections.
\item Let  $\tilde{s}_{g} : \scrC_{g} \to \scrD_{g}$ be the
map defined by the ASD equation
\[
A \mapsto F^{+_{g}}_{A}.
\]
\item We call the integer $d$ defined by the formula
\[
d= -2p_{1}(\frakP) - 3(1 - b_{1}(X) + b^{+}(X))
\]
the {\it formal dimension} of the moduli space of the solutions to the $SO(3)$-ASD equation.
\end{itemize}

The second setting is on \SW equations, to which we refer as the {\it SW setting}:
\begin{itemize}
\item  A $\spcGL$ structure $s=(P_{GL}, \psi)$ is given.
Let $\fraks$ be the isomorphism class of $s$.
\item We write $\gauge$ for the gauge group $\gauge_{s}$ of $s$ which is $L^{2}_{k+1}$-completed for $k > 1$.
\item The groups $G$ and $\tilde{G}$ are either one of the following \eqref{candidate of tG and G in SW no homology ori} and \eqref{candidate of tG and G in SW with homology ori}:
\begin{enumerate}
\item $G = \Diff(X, \fraks)$ and $\tilde{G} = \Aut(X, s)$. \label{candidate of tG and G in SW no homology ori}
\item $G = \Diff(X, \fraks, \calO)$ and $\tilde{G} = \Aut(X, s, \calO)$ for a given homology orientation $\calO$. \label{candidate of tG and G in SW with homology ori}
\end{enumerate}
\item A continuous fiber bundle $X \to E \to B$ with structure group $G$ is given.
\item For the single $4$-manifold $X$, define $\Pi(X) := \bigsqcup_{g \in \Met(X)}L^{2}_{k-1}(\Lambda^{+}_{g}(X))$.
We also define $\Pi(E) := \bigsqcup_{b \in B} \Pi(E_{b})$.
\item Let $\pi : \Pi(X) \to \Met(X)$ be the projection.
For each $\mu \in \Pi(X)$, define 
\[
\scrC_{\mu} := \conn_{L^{2}_{k}}(s) \times L^{2}_{k}(S^{+}_{\pi(\mu)}), \ \scrD_{\mu} := L^{2}_{k-1}(i\Lambda^{+}_{\pi(\mu)}(X)) \times L^{2}_{k-1}(S^{-}_{\pi(\mu)})
\]
where $\conn_{L^{2}_{k}}(s)$ is the $L^{2}_{k}$-completion of the space of $U(1)$-connections of the determinant line bundle of $s$, and $S^{\pm}_{\pi(\mu)}$ is the positive and negative spinor bundles of the $\spc$ structure $s_{\pi(\mu)}$ induced from $s$ corresponding to the metric $\pi(\mu)$.
Let $\scrC_{\mu}^{\ast}$ be the subset of $\scrC_{\mu}$ consisting of reducible configurations.
\item Let $\tilde{s}_{\mu} : \scrC_{\mu} \to \scrD_{\mu}$ be the
map defined by the \SW equations
\[
(A, \Phi) \mapsto (F^{+_{\pi(\mu)}}_A + i \mu - \rho^{-1}(\si(\Phi, \Phi)), D_{A}\Phi),
\]
where $\rho : \Lambda^+_{\pi(\mu)}(X) \to \mathfrak{su}(S^+_{\pi(\mu)})$ is the map obtained from the Clifford multiplication, $\sigma(\cdot, \cdot)$ is the quadratic form given by $\sigma(\Phi, \Phi) = \Phi \otimes \Phi^\ast - |\Phi|^2 \id/2$, and $D_{A} : L^{2}_{k}(S^{+}_{\pi(\mu)}) \to L^{2}_{k-1}(S^{-}_{\pi(\mu)})$ is the Dirac operator.
\item We call the integer $d$ defined by the formula
\[
d= (c_1(\fraks)^2 - 2\Euler(X) -3\sign(X))/4
\]
the {\it formal dimension} of the moduli space of the solutions to the \SW equation.
\end{itemize}

In both settings, we use the following notations
\[
\gauge,\quad G,\quad \tilde{G},\quad X \to E \to B,\quad \Pi(X),\quad \Pi(E),\quad \tilde{s}_{\bullet} : \scrC_{\bullet} \to \scrD_{\bullet}, \quad  \scrC_{\bullet}^{\ast}\ (\bullet \in \Pi(X))
\]
as we defined above.
We refer to the case that we consider $G = \Diff(X, \frakP, \calO)$ and $\tilde{G} = \Aut(X, P, \calO)$, or $G = \Diff(X, \fraks, \calO)$ and $\tilde{G} = \Aut(X, s, \calO)$ as the {\it homology oriented case}.
For each $\bullet \in \Pi(X)$, we can define the moduli space of the ASD/\SW equations corresponding to $\bullet$ by
\[
\M_{\bullet} := \tilde{s}_{\bullet}^{-1}(0)/\gauge.
\]
We also define
\[
\scrB_{\bullet} := \scrC_{\bullet}/\gauge,\quad 
\scrB_{\bullet}^{\ast} := \scrC_{\bullet}^{\ast}/\gauge,\quad
\scrE_{\bullet} := \scrC_{\bullet}^{\ast} \times_{\gauge} \scrD_{\bullet}.
\]
Recall that $\scrB_{\bullet}^{\ast}$ is a paracompact Hausdorff Hilbert manifold and $\scrE_{\bullet}$ is a Hilbert bundle on $\scrB^{\ast}_{\bullet}$ with fiber $\scrD_{\bullet}$.
Since the map $\tilde{s}_{\bullet} : \scrC^{\ast}_{\bullet} \to \scrD_{\bullet}$ is $\gauge$-equivariant, this defines a section
\[
s_{\bullet} : \scrB^{\ast}_{\bullet} \to \scrE_{\bullet}.
\]

\section{Cocycle condition modulo gauge}
\label{subsection Nakamura's idea and families of Fredholm sections}

In \cref{section Construction of the characteristic classes} we shall construct characteristic classes with structure groups $\Diff(X, \frakP)$, $\Diff(X, \fraks)$, and the versions of them obtained by adding $\calO$ given in \cref{section Structure groups}.
In the usual theory of characteristic classes, the natural candidate of the structure group is the automorphism group of the fiber in the category in which the fiber belongs.
However, our groups are not the automorphism groups in the categories arising from gauge theory, i.e. in the category of the pairs of a $4$-manifold and an $SO(3)$-bundle on it and that of $\spc$ $4$-manifolds.
(This has been noticed in M.~Szymik~\cite{MR2652709}.)
To resolve it we use N.~Nakamura's idea~\cite{MR2644908} for families of ASD/\SW equations.
The idea can be regarded as a ``global version'' of D.~Ruberman's idea given in~\cite{MR1671187, MR1734421, MR1874146} used to define invariants of diffeomorphisms of a $4$-manifold.
We also note that this kind of idea is initially given by M.~Furuta~\cite{MR991097} for consideration of group action rather than that of family.
Although Nakamura has considered the idea for parameterized moduli spaces, we apply the idea also to families of Fredholm sections themselves arising from ASD/\SW equations.
This is because we shall use a \vn technique in \cref{section Construction of the characteristic classes}.
Let us choose one of the ASD setting or the SW setting, and henceforth work on it in this \lcnamecref{subsection Nakamura's idea and families of Fredholm sections}.

We first recall the observation due to Ruberman~\cite{MR1671187, MR1734421, MR1874146}:
each $f \in G$ defines a well-defined ``isomorphism''
\begin{align}
f^{\ast} : \M_{\bullet} \to \M_{f^{\ast}\bullet}
\label{isomorphism by pull back between moduli spaces}
\end{align}
between the moduli spaces for any $\bullet \in \Pi(X)$.
Here the term ``isomorphism'' means, for example, a diffeomorphism if $\bullet$ is generic (and the moduli $\M_{\bullet}$ is a manifold). 
This map is defied using a lift $\tilde{f} \in \tilde{G}$ of $f$ as follows.
Recall we have the exact sequence
\begin{align}
1 \to \gauge \to \tilde{G} \to G \to 1.
\label{exact sequence abstract}
\end{align}
The pull-back by $\tilde{f}$ defines an invertible map
\begin{align}
\tilde{f}^{\ast} : \scrC_{\bullet} \to \scrC_{f^{\ast}\bullet}.
\label{isomorphism by pull back between configuration spaces}
\end{align}
The gauge group $\gauge$ is a normal subgroup of $\tilde{G}$ because of \eqref{exact sequence abstract}, and therefore the map \eqref{isomorphism by pull back between configuration spaces} induces a map between quotients
\begin{align}
\tilde{f}^{\ast} :  \scrB_{\bullet} \to \scrB_{f^{\ast}\bullet}.
\label{isomorphism by pull back between quotients of configuration spaces}
\end{align}
This map \eqref{isomorphism by pull back between quotients of configuration spaces} is independent of the choice of lift $\tilde{f}$ again because of the exact sequence \eqref{exact sequence abstract}.
We therefore write $f^{\ast} :  \scrB_{\bullet} \to \scrB_{f^{\ast}\bullet}$ for the map \eqref{isomorphism by pull back between quotients of configuration spaces}, and define the map \eqref{isomorphism by pull back between moduli spaces} as the restriction of this map.

We next use Nakamura's idea~\cite{MR2644908} of considering the ``cocycle condition modulo gauge'' to construct the parameterized moduli space which  is globally defined on the whole of $B$.
Fix a section $\si : B \to \Pi(E)$.
Since $G$ is not the automorphism group in the suitable categories in both settings, we cannot write down families of ASD/\SW equations corresponding to $\si$.
Nevertheless, we can consider the moduli space parameterized by $\si$ on the whole of $B$ as follows.
Let $\{U_{\alpha}\}_{\alpha}$ be an open covering of $B$ satisfying that $U_{\alpha} \cap U_{\beta}$ is contractible for any $\alpha, \beta$.
(We can take such a covering by mimicking the argument, for example given in Hatcher's book~\cite{MR1867354}, to prove that a CW complex is locally contractible.)
We write $U_{\alpha\beta}$ and $U_{\alpha\beta\gamma}$ for $U_{\alpha} \cap U_{\beta}$ and $U_{\alpha} \cap U_{\beta} \cap U_{\gamma}$ respectively.
Take local trivializations of $E \to B$ on this covering
and let $\{ g_{\alpha\beta} : U_{\alpha\beta} \to G \}_{\alpha, \beta}$ be the transition functions corresponding to the local trivializations.
Since $U_{\alpha\beta}$ is contractible for each $\alpha, \beta$, there exists a lift $\tilde{g}_{\alpha\beta} : U_{\alpha\beta} \to \tilde{G}$ of $g_{\alpha\beta}$.
Note that $\{\tilde{g}_{\alpha\beta}\}$ satisfies the ``cocycle condition modulo gauge'', namely,
\begin{align}
\tilde{g}_{\alpha\beta}\tilde{g}_{\beta\gamma}\tilde{g}_{\gamma\alpha}(b) \in \gauge
\label{cocycle condition modulo gauge}
\end{align}
holds for any $b \in U_{\alpha\beta\gamma}$ since $\{g_{\alpha\beta}\}$ satisfies the cocycle condition and we have the exact sequence \eqref{exact sequence abstract}.
The given section $\si : B \to \Pi(E)$ corresponds to a system of maps $\{\si_{\alpha} : U_{\alpha} \to \Pi(X)\}_{\alpha}$ satisfying that $\si_{\alpha} = g_{\alpha\beta} \cdot \si_{\beta}$ on $U_{\alpha\beta}$.
Here the action of $g_{\alpha\beta}$ to $\si_{\beta}$ is given by the action of $G$ on $\Pi(X)$ via $G \inc \Diff^{+}(X)$, namely, $g_{\alpha\beta} \cdot \si_{\beta} = g_{\alpha\beta}^{\ast} \si_{\beta}$.
For each $\alpha$, let us define the ``locally defined'' parameterized moduli space 
\begin{align*}
\calM_{\si_{\alpha}} := \bigsqcup_{b \in U_{\alpha}} \calM_{\si_{\alpha}(b)}.
\end{align*}
For each point $b \in U_{\alpha\beta}$,
we obtain an invertible map
\[
\tilde{g}_{\alpha\beta}(b)^{\ast} : \calM_{\si_{\beta}(b)} \to \calM_{g_{\alpha\beta}(b)^{\ast} \si_{\beta}(b)}
\]
like \eqref{isomorphism by pull back between moduli spaces}.
(In fact this map is independent of lift $\tilde{g}_{\alpha\beta}$ as in the definition of \eqref{isomorphism by pull back between moduli spaces}, we keep using the notation $\tilde{g}_{\alpha\beta}^{\ast}$.)
Because of the relation $\si_{\alpha} = g_{\alpha\beta} \cdot \si_{\beta} = g_{\alpha\beta}^{\ast} \si_{\beta}$ on $U_{\alpha\beta}$, we eventually have
\[
\tilde{g}_{\alpha\beta}^{\ast} : \calM_{\si_{\beta}}|_{U_{\alpha\beta}} \to \calM_{\si_{\alpha}}|_{U_{\alpha\beta}}.
\]
The composition
\[
\tilde{g}_{\alpha\beta}^{\ast} \circ \tilde{g}_{\beta\gamma}^{\ast} \circ \tilde{g}_{\gamma\alpha}^{\ast} :\calM_{\si_{\alpha}}|_{U_{\alpha\beta\gamma}} \to \calM_{\si_{\alpha}}|_{U_{\alpha\beta\gamma}}
\]
coincides with the identity because of \eqref{cocycle condition modulo gauge}.
This is again a consequence of the definition of the moduli space: it is the quotient space by the gauge group.
We can therefore obtain the well-defined quotient space:

\begin{defi}
\label{defi: globally defined parameterized moduli space}
We define the space $\calM_{\si}$ equipped with a map $\calM_{\si} \to B$ by
\begin{align*}
\calM_{\si} := \bigsqcup_{\alpha} \calM_{\si_{\alpha}} / \sim,
\end{align*}
where the equivalence relation $\sim$ is given by the invertible maps $\{\tilde{g}^{\ast}_{\alpha\beta}\}$.
We refer to the space $\calM_{\si}$ as the {\it (globally defined) parameterized moduli space}.
\end{defi}

We now consider this idea at the level of Fredholm sections.
Namely, we construct the ``globally defined'' family of Fredholm sections as follows.
For a fixed section $\si : B \to \Pi(E)$,
set
\begin{align*}
\scrB_{\si_{\alpha}}^{\ast} := \bigsqcup_{b \in U_{\alpha}} \scrB_{\si_{\alpha}(b)}^{\ast},\quad
\scrE_{\si_{\alpha}} := \bigsqcup_{b \in U_{\alpha}} \scrE_{\si_{\alpha}(b)},\quad
s_{\si_{\alpha}} := \bigsqcup_{b \in U_{\alpha}} s_{\si_{\alpha}(b)}: \scrB_{\si_{\alpha}}^{\ast} \to \scrE_{\si_{\alpha}}.
\end{align*}
Let us take $f \in G$ and $\bullet \in \Pi(X)$.
Then we can consider the invertible map $f^{\ast} :  \scrB_{\bullet} \to \scrB_{{f^{\ast}\bullet}}$ as in the definition of \eqref{isomorphism by pull back between moduli spaces}.
The restriction of this map gives an invertible map $f^{\ast} : \scrB_{\bullet}^{\ast} \to \scrB_{{f^{\ast}\bullet}}^{\ast}$ between irreducible configurations.
We can therefore define a parameterized Hilbert manifold and a parameterized Hilbert bundle using $\scrB^{\ast}_{\bullet}$ and $\scrE_{\bullet}$:

\begin{defi}
\label{defi family of scrB and scrE in the most general setting}
We define
\[
\scrB_{\si}^{\ast} := \bigsqcup_{\alpha} \scrB^{\ast}_{\si_{\alpha}} / \sim
\]
by the same argument to define $\M_{\si}$ as in \cref{defi: globally defined parameterized moduli space}.
This is a Hilbert manifold bundle parameterized on $B$.
Similarly we can define a parameterized Hilbert bundle $\scrE_{\si} \to \scrB_{\si}^{\ast} \to B$ by 
\[
\scrE_{\si} := \bigsqcup_{\alpha} \scrE_{\si_{\alpha}} / \sim.
\]
\end{defi}

In addition, the Fredholm sections themselves are also obtained as the quotient of the gauge group, and we can define:

\begin{defi}
\label{defi: family of Fredholm sections in the most general setting}
We define a family of Fredholm sections $s_{\si} : \scrB_{\si}^{\ast} \to \scrE_{\si}$ parameterized on $B$ by
\begin{align*}
s_{\si} := \bigsqcup_{\alpha} s_{\si_{\alpha}} / \sim.
\end{align*}
\end{defi}


\section{Virtual neighborhood for families}
\label{section Virtual neighborhood for families}

One of the precepts of Atiyah--Singer index theory for families~\cite{MR0279833} is that the natural class of base spaces for a theory on families of Fredholm operators is {\it not} smooth manifolds: the smoothness of base spaces is an excessive assumption.
Therefore, in our theory of characteristic classes, we do not assume that the base space of a given family of ASD/\SW equations is a smooth manifold.
(In fact, we shall eventually extend our theory to families on any topological space.)
The problem here is, of course, how to deal with the transversality to count the moduli spaces parameterized on a non-smooth space. 
To justify this counting argument, we describe a family version of Y.~Ruan's virtual neighborhood technique~\cite{MR1635698}.
This provides a foundation for our construction of characteristic classes in \cref{section Construction of the characteristic classes}.

\subsection{Virtual neighborhood}
\label{subsection unparameterized vn}

In this \lcnamecref{subsection unparameterized vn} we consider the usual (i.e. unparameterized) virtual neighborhood.
The purposes of this \lcnamecref{subsection unparameterized vn} is to rewrite the counting argument in the \vn technique in terms of the relative Euler class:
although Ruan has defined the invariant associated with a given non-linear Fredholm section of a given Hilbert bundle using the finite dimensional Sard's theorem, we define the invariant using the relative Euler class given in M.~Kervaire~\cite{MR0090051} instead of Sard's theorem.
We note that P.~Feehan and T.~Leness~\cite{MR1936209,FL} have used Ruan's technique in a similar form.
This description of the invariant allows us easily to extend Ruan's technique to families on non-smooth spaces.

We work in the following general setting.

\begin{itemize}
\item Let $\scrX$ be a Hilbert manifold.
We assume that $\scrX$ is paracompact Hausdorff and the model of $\scrX$ is a separable Hilbert space.
These topological conditions are used to assure that there exists a smooth partition of unity on $\scrX$.
(See, for example,  Lang~\cite{MR1335233}, Chapter II, Corollary~3.8.)
\item Let $\scrH$ be a Hilbert space and $\scrE \to \scrX$ be a Hilbert bundle with fiber $\scrH$.
\item Let $s : \scrX \to \scrE$ be a Fredholm section.
\end{itemize}

Here the definition of a Fredholm section in this paper is as follows.
For $x\in \scrX$, we have the canonical decomposition $T_{(x,0)}\scrE = T_x\scrX \oplus \scrH$.
We also write $ds_x : T_x\scrX \to \scrH$ for the composition of $ds_x : T_x\scrX \to T_{(x,0)}\scrE$ and the projection $T_{(x,0)}\scrE = T_x\scrX \oplus \scrH \to \scrH$.
We call a smooth section $s : \scrX \to \scrE$ a {\it Fredholm section} if the differential
$ds_x : T_x\scrX \to \scrH$ is a Fredholm map for each $x \in s^{-1}(0)$.
We assume that the Fredholm index of $ds_{x}$ is common to all $x \in s^{-1}(0)$ and write $\ind{s}$ for the index.
For the zero set 
\[
\M := s^{-1}(0),
\] 
we assume the compactness:

\begin{assum}
\label{compactness assumption for unparameterized case}
The zero set $\M$ is compact.
\end{assum}

A typical example of $s$ satisfying this assumption is given as the \SW equations modulo the gauge group.
Under this assumption, we shall define the ``counted number $\#\M$'' $\in \Z \text{ or } \Z/2$, corresponding to the \SW invariant in the case of this example, via a virtual neighborhood.

We first review the construction of a virtual neighborhood.
Let $x$ be a point in $\M$.
Then there exists $N_x \in \N$ and a linear map $f_x : \R^{N_x} \to \scrH$ such that
\[
ds_{x} + f_x : T_{x} \scrX \oplus \R^{N_{x}}\to \scrH
\]
is surjective.
Since surjectivity are open conditions, there exists a small open neighborhood $U_x$ of $x$ in $\scrX$ such that $ds_y + f_x : T_{y}\scrX \oplus \R^{N_x}\to \scrH$ is surjective for any $y \in U_x \cap \M$.
Because $\M$ is compact, there exists finite points $x_{1}, \ldots, x_{m} \in \calM$ such that $\M \subset \bigcup_{i=1}^m U_{x_{i}}$.
Set $N_{i} := N_{x_{i}}$, $f_{i} := f_{x_{i}}$, $U_{i}:=U_{x_{i}}$ and $N := N_1+\cdots+N_m$.
As we mentioned, we can take a smooth partition of unity $\{\rho_i\}_{i=1}^{m}$ subordinate to $\{U_i\}_{i=1}^m$.
We also use the notation $\scrE$ for its pull-back to $\scrX \times \R^N$.
Fix a local trivialization of $\scrE$ on each $U_i$, and define a map
\begin{align*}
\vp : \scrX \times \R^N \to \scrE
\end{align*}
by
\begin{align}
(x,(a_1,\ldots,a_m)) \mapsto \sum_{i=1}^m \rho_{i}(x) f_i(a_i),
\label{eq : definition of finite dim perturbation}
\end{align}
where $x \in \scrX$ and $(a_1,\ldots,a_m) \in \R^{N_1} \times \cdots\times \R^{N_m} = \R^{N}$, 
and the map $(x, a_{i}) \mapsto \rho_{i}(x) f_i(a_i)$ is regarded as a local section of $\scrE$ via the fixed local trivialization.
Note that $\scrX \times \{0\} \subset \vp^{-1}(0)$ holds.
We call $\vp$ a {\it finite dimensional perturbation} of $s$.
For the new section 
\[
\tilde{s} := s+ \vp : \scrX \times \R^N \to \scrE,
\]
the following lemma is straightforward:

\begin{lem}[Ruan~\cite{MR1635698},~Lemma~2.3]
\label{lem : surjective}
For any $x \in \M$, the differential
\[
d\tilde{s}_{(x,0)} : T_x\scrX \oplus \R^N \to\scrH
\]
is surjective.
\end{lem}

Since surjectivity is an open condition, there exists a neighborhood $\mathscr{N}$ of $\M \times \{0\}$ in $\scrX \times \R^N$ such that the differential of $\tilde{s}$ is surjective on any point of $(\tilde{s}|_{\mathscr{N}})^{-1}(0) = \tilde{s}^{-1}(0) \cap \mathscr{N}$.
Set
\[
\calU = \calU(s,\vp) := \tilde{s}^{-1}(0) \cap \mathscr{N}.
\]
Because of the implicit function theorem, $\vnU$ is a smooth manifold.
Since $\dim \Ker (d \tilde{s}_{(x,0)}) = \ind{s} + N$ holds for any $x \in \scrX$,
we have
\[
\dim\vnU = \ind{s} + N.
\]
Henceforth we regard the zero set $\M$ as a subspace of $\calU$ by 
\[
\M \cong \M \times \{0\} = {\tilde s}^{-1}(0) \cap \left(\scrX \times \{0\} \right) \cap \mathscr{N} \subset \vnU.
\]

\begin{defi}[Ruan~\cite{MR1635698}]
The finite dimensional manifold $\vnU$ constructed above is called a {\it virtual neighborhood} for the Fredholm section $s : \scrX \to \scrE$.
\end{defi}

The restriction of the projection $\scrX \times \R^N \to \R^N$ is equipped with for a \vn $\vnU$.
We denote by $h_{\calU} : \vnU \to \R^N$ this map.
The space $\M$ is the level set of $h_{\calU}$ for the height zero.
We are interested only in the ``germ'' of $h_{\calU}$ near $\calM$, and we do not distinguish $\calU$ and another \vn $\calU'$ obtained from the common finite dimensional perturbation $\vp$ (even if another $\mathscr{N}'$ is used instead of $\mathscr{N}$).

Here we remark the orientation of $\calU$.
The determinant line bundle $\det{s} \to s^{-1}(0)$ is associated with the section $s : \scrX \to \scrH$.
Since $\tilde{s} : \scrX \times \R^N \to \scrH$ is also a Fredholm section, we can also define $\det\tilde{s} \to \tilde{s}^{-1}(0)$.

\begin{lem}[Ruan~\cite{MR1635698},~Lemma~2.4]
If $\det{s} \to s^{-1}(0)$ has a nowhere-vanishing section, it induces a a nowhere-vanishing section on $\det\tilde{s} \to \tilde{s}^{-1}(0)$.
(Thus a nowhere-vanishing section on $\det{s} \to s^{-1}(0)$ gives an orientation of $\vnU$.)
\end{lem}

We now define the ``counted number $\#\M$'' in terms of the relative Euler class.
Assume that $k := \ind s$ is non-negative and fix a cohomology class $\alpha \in H^{k}(\scrX;\Z)$.
(If we work with $\Z/2$-coefficient, we can take a general element of $H^{k}(\scrX;\Z/2)$ as $\alpha$.)
Then we obtain the cohomology class $p^\ast \alpha \in H^{k}(\vnU;\Z)$, where $p : \vnU \to \scrX$ is the restriction of the projection $\scrX \times \R^N \to \scrX$.
The map $h_{\calU} : \vnU \to \R^N$ can be regarded as a section
\begin{align}
h_{\calU} : \vnU \to \vnU \times \R^N
\label{finite dim approximated section}
\end{align}
of the trivial bundle $\vnU \times \R^N \to \vnU$.
(The trivial bundle $\vnU \times \R^N \to \vnU$ and the section $h_{\calU} : \vnU \to \vnU \times \R^N$ can be regarded as a ``finite dimensional approximation" of the Hilbert bundle $\scrE \to \scrX$ and one of the Fredholm section $s : \scrX \to \scrE$ respectively.)
Since $h_{\calU}^{-1}(0) \cong \M$ is compact, one can take a relatively compact neighborhood $K$ of $h_{\calU}^{-1}(0)$ in $\vnU$.
Let 
\[
\tau(\vnU \times \R^N) \in H^{N}(D(\vnU \times \R^N),S(\vnU \times \R^N);\Z)
\]
be the Thom class of the trivial bundle $\vnU \times \R^N \to \vnU$, where $D(\cdot)$ and $S(\cdot)$ mean the disk bundle and sphere bundle respectively.
Then one can consider the relative Euler class with respect to $h_{\calU}$:
\begin{align}
e_{\calU} := h_{\calU}^\ast \tau(\vnU \times \R^N) \in H^N(\vnU, \vnU \setminus K;\Z).
\label{definition of the relative Euler class}
\end{align}
We use the same notation $e_{\calU}$ for the image of $e_{\calU}$ by the map 
\[
H^N( \vnU, \vnU \setminus K; \Z) \to H_{\cpt}^N( \vnU; \Z),
\]
where $H_{\cpt}^\ast( \cdot; \Z)$ means the cohomology with compact supports.
Since $\vnU$ is an (open) manifold, we can consider the fundamental class $[\vnU]_{BM}$ in the sense of Borel-Moore homology.
Here $[\vnU]_{BM}$ is regarded as a homology class with $\Z$-coefficient if a nowhere-vanishing section of $\det{s}$ is given, and otherwise with $\Z/2$-coefficient.

\begin{defi}
We define the number $\frakm(s, \alpha) \in \Z$ or $\Z/2$ by
\begin{align*}
\frakm(s, \alpha) := \left< e_{\calU} \cup p^{\ast}\alpha, [\vnU]_{BM} \right>.
\end{align*}
\end{defi}

The number $\frakm(s, \alpha)$ corresponds to the``counted number $\#\M$''.
(The cohomology class $\alpha$ corresponds to a cutting of the moduli space of higher dimension.)
The proof of the following lemma can be regarded as a non-linearization of the proof of the well-definedness of the index for families.

\begin{rem}
\label{remark on vanishing of the numerical invariant}
Assume that the Poincar\'e dual (in the sense of Borel--Moore) of $p^{\ast}\alpha$ can be represented by a smooth submanifold $V \subset \calU$ of $\calU$, and also that $\M \cap V = \emptyset$.
Then we have $\frakm(s, \alpha) = 0$.
\end{rem}

\begin{lem}[corresponding to Ruan~\cite{MR1635698},~Proposition~2.6]
\label{lem : well-def1}
The number $\frakm(s, \alpha) $ depends only on the section $s : \scrX \to \scrH$ and the cohomology class $\alpha \in H^{k}(\scrX;\Z)$.
\end{lem}

\begin{proof}
Let $\vnU_i$ $(i=1,2)$ be two virtual neighborhoods obtained from
$\tilde{s}_{i} := s + \vp_i : \scrX \times \R^{N_i} \to \scrE$.
Let us consider two maps
\begin{align}
&\tilde{s}_{1}' := s + \vp_1 + 0  : \scrX \times \R^{N_1} \times \R^{N_2} \to \scrE, \text{ and} \label{tilde section1p}\\
&\tilde{s}_{12} := s + \vp_1 + \vp_2  : \scrX \times \R^{N_1} \times \R^{N_2} \to \scrE.\label{tilde section12}
\end{align}
The differential of the map
\begin{align}
s + \vp_1 + \bullet \vp_2 : \scrX \times \R^{N_1} \times \R^{N_2} \times [0,1] \to \scrE
\label{eq : cobordism}
\end{align}
defined as
\[
(x, a_{1}, a_{2}, \tau)  \mapsto s(x) + \vp_1(x,a_{1}) + \tau \vp_2(x,a_{2})
\]
is surjective at each point on $\M \times \{0\} \times \{0\} \times [0,1]$, and so on a neighborhood of $\M \times \{0\} \times \{0\} \times [0,1]$ in $\scrX \times \R^{N_1} \times \R^{N_2} \times [0,1]$.
Thus we obtain a manifold ${\vnU} \subset \scrX  \times  \R^{N_1} \times \R^{N_2} \times  [0,1]$ with boundary from the Fredholm section \eqref{eq : cobordism} via the same procedure used to define a \vn as above.
By replacing $\calU$ and virtual neighborhoods to some smaller neighborhoods of $\calM$ if we need, ${\vnU}$ gives a cobordism between a virtual neighborhood $\vnU_1'$ obtained from $\tilde{s}_{1}'$ and a virtual neighborhood $\vnU_{12}$ obtained from $\tilde{s}_{12}$.
(Such a manifold $\calU$ is called a {\it virtual neighborhood cobordism} in Ruan~\cite{MR1635698}.)
Let us consider the section 
\[
h_{\calU} : {\vnU} \times [0,1] \to {\vnU} \times \R^{N_1} \times \R^{N_2} \times [0,1]
\]
and the relative Euler class $e_{\calU} \in H_{\cpt}^{N_{1}+N_{2}}(\calU \times [0,1];\Z)$ obtained in a similar way to define \eqref{finite dim approximated section} and \eqref{definition of the relative Euler class}.
Since the restrictions of $h_{\calU}$ to $\tau = 0,1$ are the sections $h_{\calU_{1}'}$ and $h_{{\vnU}_{12}}$ respectively, the relative Euler classes $e_{\calU_{1}'}$ and $e_{{\vnU}_{12}}$ are the restrictions of $e_{\calU}$ to $\vnU_1'$ and to $\vnU_{12}$ respectively.
Thus we have
\begin{align}
\left<e_{{\vnU}_{1}'} \cup p^\ast \alpha, [\vnU_1']_{BM} \right> - \left<e_{{\vnU}_{12}} \cup p^\ast \alpha, [\vnU_{12}]_{BM} \right> 
=\left<e_{\calU} \cup p^\ast \alpha, \partial[\vnU]_{BM} \right> = 0,
\label{eq : cobordism argument at point}
\end{align}
where $p$ in the left-hand side is the map given by the projection $\scrX \times \R^{N_1} \times \R^{N_2} \to \scrX$ and $p$ in the right-hand side is that given by $\scrX \times \R^{N_1} \times \R^{N_2} \times [0,1] \to \scrX$.
Let $e_{{\vnU}_{1} \times \R^{N_{2}}} \in H_{\cpt}^{N_{1}+N_{2}}(\calU_{1} \times \R^{N_{2}};\Z)$ be the relative Euler class defined by the section
\[
h_{\calU_{1}} \times \id_{\R^{N_{2}}}: \calU_{1} \times \R^{N_{2}} \to \calU_{1} \times \R^{N_{1}} \times \R^{N_{2}}.
\]
Since the \vn $\vnU_1'$ is same to $\vnU_1 \times \R^{N_2}$ if we focus only on neighborhoods of $\calM$, we have
\begin{align}
\left<e_{{\vnU}_{1}'} \cup p^\ast \alpha, [\vnU_1']_{BM} \right>
= \left<e_{{\vnU}_{1} \times \R^{N_{2}}} \cup p^\ast \alpha, [\vnU_1 \times \R^{N_{2}}]_{BM} \right>.
\label{eq : relation to the product vn}
\end{align}
In addition, since $e_{{\vnU}_{1} \times \R^{N_{2}}}$ corresponds to $e_{\calU_{1}}$
via the suspension isomorphism
\[
H_{\cpt}^{\ast}(\vnU_1 \times \R^{N_2};\Z) \cong H_{\cpt}^{\ast-N_2}(\vnU_1;\Z),
\]
we obtain
\begin{align}
\left<e_{{\vnU}_{1} \times \R^{N_{2}}} \cup p^\ast \alpha, [\vnU_1 \times \R^{N_{2}}]_{BM} \right>
= \left<e_{{\vnU}_{1}} \cup p^\ast \alpha, [\vnU_1]_{BM} \right>.
\label{eq : relation to the product vn via suspension iso}
\end{align}
The equalities \eqref{eq : cobordism argument at point}, \eqref{eq : relation to the product vn}, and \eqref{eq : relation to the product vn via suspension iso} imply that 
\begin{align}
\left<e_{{\vnU}_{1}} \cup p^\ast \alpha, [\vnU_1]_{BM} \right>
= \left<e_{{\vnU}_{12}} \cup p^\ast \alpha, [\vnU_{12}]_{BM} \right>.
\label{eq : relation between vn1 and vn12}
\end{align}
Via $\left<e_{{\vnU}_{12}} \cup p^\ast \alpha, [\vnU_{12}]_{BM} \right>$, it follows that 
\[
\left<e_{{\vnU}_{1}} \cup p^\ast \alpha, [\vnU_1]_{BM} \right>
= \left<e_{{\vnU}_{2}} \cup p^\ast \alpha, [\vnU_2]_{BM} \right>
\]
from the equality \eqref{eq : relation between vn1 and vn12} and the similar equality relating $\calU_{2}$ and $\calU_{12}$ shown by the same argument.
\end{proof}

\subsection{Family virtual neighborhood}
\label{subsection Family virtual neighborhood}

We give a family version of Ruan's \vn technique in this \lcnamecref{subsection Family virtual neighborhood}.
We work in the following setting in this subsection.

\begin{itemize}
\item Let $B$ be a normal space. (Then we can take continuous cut-off functions.)
\item Let $\pi : \scrX = \bigsqcup_{b \in B} \scrX_b \to B$ be a parametrized Hilbert manifold, namely, be a continuous fiber bundle and suppose that each fiber $\scrX_b$ be a paracompact Hausdorff Hilbert manifold whose model Hilbert space is separable.
\item Let $\scrE = \bigsqcup_{b \in B} \scrE_b \to \bigsqcup_{b \in B} \scrX_{b} \to B$ be a parametrized Hilbert bundle.
Namely, let $\scrE\to B$ be a continuous fiber bundle and suppose that $\scrE_b \to \scrX_b$ is a smooth Hilbert bundle whose fiber is a Hilbert space $\scrH_b$ for each $b \in B$.
\item Let $s = \bigsqcup_{b \in B} s_{b} : \scrX \to \scrE$ be a parametrized Fredholm section, namely, be a continuous section and suppose that the restriction to each fiber $s_b : \scrX_b \to \scrE_b$ is a smooth Fredholm section.
\item We assume that the index of $d(s_{b})_{x} : T_{x} \scrX_{b} \to \scrH_{b}$ is common to all $b \in B$ and $x \in \scrX_b$ and write $\ind{s}$ for the index.
We also assume that $d(s_{b})$ continuously depends on $b$.
Namely, for the bundles $T_{\fiber}\scrX = \bigsqcup_{b \in B} T\scrX_{b}$ and $T_{\fiber}\scrE = \bigsqcup_{b \in B} T\scrE_{b}$, the section of $\Hom(T_{\fiber}\scrX, s^{\ast}T_{\fiber}\scrE) \to B$ induced from $ds$ is continuous.

\end{itemize}

Set 
\[
\M = \bigsqcup_{b \in B} \M_{b}: = s^{-1}(0) = \bigsqcup_{b \in B} s_b^{-1}(0),
\]
corresponding to the parametrized moduli space. 
We assume that $s$ satisfies the following assumption on compactness:

\begin{assum}
\label{assumption on compactness for parameterized moduli}
The space $\M$ is compact.
\end{assum}

We now construct a ``family virtual neighborhood".
Let $x$ be a point in $\M$.
Then there exists $N_{x} \in \N$ and a linear map $f_{x}$ such that
\[
d(s_{\pi(x)})_x + f_{x} : T_x \scrX_{\pi(x)} \oplus \R^{N_{x}} \to \scrH_{\pi(x)}
\]
is surjective.
Take a small neighborhood $U_x$ of $x$ in $\scrX$ such that $d(s_{\pi(y)})_{y} + f_{x}$ is surjective for any $y \in U_x \cap \M$.
We can choose $U_x$ such as it is the product of open sets $U_x^{\base}$ of $B$ and $U_x^{\fiber}$ of the fiber $\scrX_{b}$ via a fixed local trivialization of $\scrX \to B$, i.e. $U_x \cong U_x^{\base} \times U_x^{\fiber}$.
Let $p_{x} : U_x^{\base} \times U_x^{\fiber} \to U_{x}^{\fiber}$ be the projection.
Take a (continuous) cut-off function $\rho_{x}^{\base} : U_{x}^{\base} \to [0,1]$ supported in $U_{x}^{\base}$ satisfying $\rho_{x}^{\base}(x)>0$, and do a smooth partition of unity $\{\rho_{x}^{\fiber}\}_{x \in \M_{b}}$ subordinate to $\{U_{x}^{\fiber}\}_{x \in \M_{b}}$ in $\scrX_{b}$.
Define 
\[
\rho_{x} : U_{x} \to [0,1]
\]
as 
\[
\rho_{x}^{\base} \cdot p_{x}^{\ast} \rho_{x}^{\fiber} : U_x^{\base} \times U_x^{\fiber} \to [0,1]
\]
via the fixed local trivialization.
Take an open set $U_{x}'$ in $\scrX$ with $x \in U_{x}' \subset U_{x}$ such that $\rho_{x}>0$ on $U_{x}'$.
Then $\{U_{x}'\}_{x \in \M}$ is an open covering of $\M$, and hence there exists finite points $x_{1}, \ldots, x_{m} \in \calM$ such that $\M \subset \bigcup_{i=1}^m U_{x_{i}}$.
Set $N_{i} := N_{x_{i}}$, $f_{i} := f_{x_{i}}$, $U_{i} := U_{x_{i}}$, and $N := N_1+\cdots+N_m$.
Fix a local trivialization of $\scrE$ on each $U_i$, and define a map
\[
\vp : \scrX \times \R^N \to \scrE
\]
by the formula \eqref{eq : definition of finite dim perturbation} via the fixed local trivializations.
Set 
\begin{align*}
&\tilde{s} := s + \vp : \scrX \times \R^{N} \to \scrE, \text{ and} \\
&\tilde{s}_{b} := s_{b} + \vp : \scrX_{b} \times \R^{N} \to \scrE_{b}.
\end{align*}
We now have the following lemma by the completely same argument to prove \cref{lem : surjective} since $\rho_{i}$ is smooth along the fiber direction of $\scrX \to B$ and for any $x \in \M$ there exists $i$ such that $\rho_{i}(x)>0$:

\begin{lem}
\label{lem : surjective for family}
For any $b \in B$ and any $x \in \M_{b}$, the differential
\[
d(\tilde{s}_{b})_{(x,0)} : T_x\scrX_{b} \oplus \R^N \to\scrH_{b}
\]
is surjective.
\end{lem}

From this lemma, there exists a neighborhood $\mathscr{N}$ of $\M \times \{0\}$ in $\scrX \times \R^N$ such that such that the differential of $\tilde{s}_{b}$ is surjective on any point of $(\tilde{s}_{b}|_{\mathscr{N}})^{-1}(0) = \tilde{s}_{b}^{-1}(0) \cap \mathscr{N}$ for any $b$.
Set
\begin{align*}
&\calU = \calU(s,\vp) := \tilde{s}^{-1}(0) \cap \mathscr{N}, \text{ and}\\
&\vnU_b := \tilde{s}_{b}^{-1}(0)\cap\mathscr{N}.
\end{align*}
Then each $\vnU_b$ is a smooth manifold with $\dim\vnU_{b} = \ind{s} + N$.

\begin{defi}
The family of manifolds
\[
\vnU = \vnU(s, \vp) = \bigsqcup_{b \in B} \vnU_{b}
\]
constructed above is called a {\it family virtual neighborhood} for the parametrized Fredholm section $s : \scrX \to \scrH$.
\end{defi}

As in the non-parameterized case, the restriction of the projection $\scrX \times \R^N \to \R^N$ is equipped with for a family \vn $\vnU$.
We denote by $h_{\calU} : \vnU \to \R^N$ this map.
The space $\M$ is regarded as a subspace of $\calU$ and is the level set of $h_{\calU}$ for the height zero.

\begin{rem}
We note that a family virtual neighborhood $\vnU$ is not a fiber bundle in general,
and also note that, although $\vnU$ is parameterized on the whole of $B$, it is ``supported'' on the subspace 
\[
B^{s} := \Set{b \in B | s_{b}^{-1}(0) \neq \emptyset},
\]
i.e.  if $s_{b}^{-1}(0) = \emptyset$ holds, then a small neighborhood of $\M_{b}$ in $\calU_{b}$ is also empty.
\end{rem}

We now construct a (compactly supported) cohomology class on $B$ from a family virtual neighborhood.
For $k \geq 0$, fix a cohomology class $\alpha$ in $H^{k}(\scrX;\Z)$ or $H^{k}(\scrX;\Z/2)$.
Let $p : \vnU \to \scrX$ be the restriction of the projection $\scrX \times \R^N \to \scrX$.
The map $h_{\calU} : \vnU \to \R^N$ can be regarded as a section
\begin{align*}
h_{\calU} : \vnU \to \vnU \times \R^N
\end{align*}
of the trivial bundle $\vnU \times \R^N \to \vnU$.
Since $h_{\calU}^{-1}(0) \cong \M$ is compact, one can take a relatively compact neighborhood $K$ of $h_{\calU}^{-1}(0)$ in $\vnU$.
Let us consider the relative Euler class with respect to $h_{\calU}$:
\begin{align*}
e_{\calU} := h_{\calU}^\ast \tau(\vnU \times \R^N) \in H^N(\vnU, \vnU \setminus K;\Z)
\end{align*}
using the Thom class $\tau(\vnU \times \R^N)$ of the trivial bundle $\vnU \times \R^N \to \vnU$.
We use the same notation $e_{\calU}$ for the image of $e_{\calU}$ by the map 
\[
H^N( \vnU, \vnU \setminus K; \Z) \to H_{\cpt}^N( \vnU; \Z).
\]
Since $\M$ is compact and $\calU$ is supported on $B^{s}$, one can find a fiberwise embedding $\calU \to B \times \R^{n}$ for large $n$.
Although $\calU \to B$ is not a fiber bundle in general, one can therefore define ``integration along the fiber''
\begin{align}
\pi_{!} : H_{\cpt}^{\ast}(\calU) \to H_{\cpt}^{\ast-(\ind{s}+N)}(B)
\label{eq: fiber integration}
\end{align}
as follows.
For each $b \in B$, let $\nu_{\calU_b} \to \calU_b$ be the normal bundle for the embedding $\calU_b \inc \{b\} \times \R^{n}$, and let $\nu_{\calU} \to \calU$ be the ``parametrized normal bundle", i.e. $\nu_{\calU} := \bigsqcup_{b \in B} \nu_{\calU_b} \to \calU \to B$.
Then $\nu_{\calU} \to \calU$ is a continuous vector bundle on $\calU$ of rank $n-(\ind{s}+N)$.
We regard $\nu_{\calU}$ as a small neighborhood of $\calU$ in $B \times \R^{n}$.
Using the Thom isomorphism and the excision isomorphism, we get the map
\begin{align*}
H^\ast(\calU) \cong& H^{\ast + n-(\ind{s}+N)}(D(\nu_{\calU}), S(\nu_{\calU})) \\
\cong & H^{\ast + n-(\ind{s}+N)}(D(B \times \R^n), \overline{D(B \times \R^n) \setminus \nu_{\calU}}) \\
\to & H^{\ast + n-(\ind{s}+N)}(D(B \times \R^n), S(B \times \R^n) ) \cong H^{\ast-(\ind{s}+N)}(B).
\end{align*}
This map induces a map between compact supported cohmology grourps, and we define \eqref{eq: fiber integration} as the map.
Here the coefficient of the cohomology groups in \eqref{eq: fiber integration} is $\Z$ if a nowhere-vanishing section of $\det{s} \to s^{-1}(0)$ is given, and otherwise $\Z/2$.
The map \eqref{eq: fiber integration} is independent of the choice of fiberwise embedding as usual.

\begin{defi}
\label{general def of cohomological invariant from family vn}
We define a cohomology class $\frakM(s, \alpha)$ by the formula
\begin{align}
\frakM(s, \alpha) := \pi_{!}(e_{\calU} \cup p^{\ast}\alpha) \in H_{\cpt}^{k-\ind{s}}(B).
\label{eq : family invariant0}
\end{align}
Let $B' \subset B$ be a subset satisfying $B' \cap B^{s} = \emptyset$, i.e. $s$ is nowhere-vanishing on $B'$.
Then the cohomology class \eqref{eq : family invariant0} can be regarded as an element of the relative cohomology:
\begin{align}
\frakM(s, \alpha;B' ):= \pi_{!} \left( e_{\vnU}\cup p^\ast \alpha \right) \in H_{\cpt}^{k-\ind{s}}(B,B').
\label{eq : family invariant1}
\end{align}
Here the coefficient of the cohomology groups in \eqref{eq : family invariant0} and \eqref{eq : family invariant1} is $\Z$ if a nowhere-vanishing section of $\det{s} \to s^{-1}(0)$ is given, and otherwise is $\Z/2$.
We call $\frakM(s, \alpha)$ and $\frakM(s, \alpha;B')$ the {\it cohomological invariants emerging from} $s$, $\alpha$ (and $B'$ for the latter case).
\end{defi}

\begin{rem}
\label{remark on relative family invariant using PD}
Assume that the fiberwise Poincar\'{e} dual of $p^{\ast}\alpha$ can be represented as a family of smooth submanifolds $V \subset \calU$, $V = \bigsqcup_{b \in B} V_{b}$.
For a subspace $B' \subset B$, suppose that $\M_{b} \cap V_{b} = \emptyset$ for any $b \in B'$.
Then, as in \cref{remark on vanishing of the numerical invariant}, the cohomology class $\frakM(s, \alpha)$ can be regarded as an element of the relative cohomology group; we can define $\frakM(s, \alpha;B' ) \in H_{\cpt}^{k-\ind{s}}(B,B')$ as in \eqref{eq : family invariant1}.
\end{rem}

\begin{rem}
For a given family of $4$-manifolds, assume that the family of ASD or \SW equations arising from the family of $4$-manifolds satisfies the condition corresponding to
\cref{assumption on compactness for parameterized moduli}.
Then we can obtain a cohomology class on the base space by combing \cref{general def of cohomological invariant from family vn} with the argument given in \cref{section Structure groups}.
However, for a general family on a non-compact base (for example the universal bundle of the diffeomorphism group on the classifying space of it), there is no standard way to assure the condition corresponding to
\cref{assumption on compactness for parameterized moduli}.
To establish the universal theory, we need an argument obtained by mimicking the obstruction theory, given in \cref{section Construction of the characteristic classes}.

\end{rem}

\begin{lem}
\label{lemma on well definedness of the cohomological inv in family vn context}
The cohomology class $\frakM(s, \alpha)$ depends only on the section $s : \scrX \to \scrE$ and the cohomology class $\alpha \in H^{k}(\scrX;\Z)$.
 Similarly, if $s: \scrX \to\scrE$ is nowhere-vanishing on a subspace $B' \subset B$,
 the cohomology class $\frakM(s, \alpha ; B' )$  depends only on $s : \scrX \to \scrE$, $\alpha \in H^{k}(\scrX;\Z)$, and $B'$.
\label{lem : well-def2}
\end{lem}

\begin{proof}
Let $\vnU_i$ $(i=1,2)$ be two virtual neighborhoods obtained from
$\tilde{s}_{i} := s + \vp_i : \scrX \times \R^{N_i} \to \scrE$ and consider two maps $\tilde{s}_{1}'$ and $\tilde{s}_{12}$ defined by the same formulae \eqref{tilde section1p} and \eqref{tilde section12}.
For each $b \in B$, the differential of 
\begin{align*}
s_b + \vp_1 + \bullet \vp_2 :  \scrX_b \times \R^{N_1} \times \R^{N_2} \times [0,1] & \to \scrE_{b} \times [0,1]
\end{align*}
is surjective at each point on $\M_b \times \{0\} \times \{0\} \times [0,1]$, and thus on a neighborhood of $\M_b \times \{0\} \times \{0\} \times [0,1]$ in $\scrX_b \times \R^{N_1} \times \R^{N_2} \times [0,1]$.
Thus we obtain a family of manifolds with boundaries ${\vnU} \subset \scrX \times \R^{N_1} \times \R^{N_2} \times [0,1]$ and we can assume that ${\vnU}$ gives a fiberwise cobordism between a family \vn $\vnU_1'$ obtained from $\tilde{s}_{1}'$ and a family \vn $\vnU_{12}$ obtained from $\tilde{s}_{12}$.
The relative Euler classes $e_{\calU_{1}'}$ and $e_{\calU_{12}}$ are the restrictions of the relative Euler class obtained from $h_{\calU}$ to $\vnU_1'$ and $\vnU_{12}$ respectively.
Therefore, using the following elementary Lemma~\ref{lem : fiberwise cobordism}, we have $\pi_{!}(e_{{\vnU}_{1}'} \cup p^\ast \alpha)
= \pi_{!}(e_{{\vnU}_{12}} \cup p^\ast \alpha)$.
The rest of the proof is an argument based on the suspension isomorphism, which is same as the proof of Lemma~\ref{lem :  well-def1}.
\end{proof}

\begin{lem}
Let $\calV_0 \to B,\ \calV_1 \to B$ and $\calV \to B$ be continuous maps (not necessary fiber bundles).
Assume that the inverse images by these maps of each point of $B$ are smooth manifolds, and also assume that they admit fiberwise embeddings into a trivial vector bundle.
(Then one can define integration along the fiber for these families.)
Suppose that $\calV$ gives a fiberwise cobordism between $\calV_0$ and $\calV_1$.
Then, for any cohomology class $\beta \in H_{\cpt}^\ast(\calV)$,
\[
\pi_{!}(\beta|_{\calE_0})
= \pi_{!}(\beta|_{\calE_1})
\]
holds.
\label{lem : fiberwise cobordism}
\end{lem}

\begin{proof}
Since integration along the fiber commutes with restriction, it is enough to show our statement in the case that $\calV_{i}$ and $\calV$ are trivial bundles.
In this case, one can easily to see it using the K\"{u}nneth formula and the cobordism argument as \eqref{eq : cobordism argument at point}.
\end{proof}

\begin{defi}
\label{defi of the evaluated invariant}
\begin{enumerate}
\item
\label{case of absolute evaluation of cohomological invariant}
If $B$ is a closed manifold of dimension of $k-\ind{s}$, we define a number $\frakm(s, \alpha)$ by
\[
\frakm(s, \alpha) := \left< \frakM(s, \alpha), [B]\right> \in \Z \text{ or } \Z/2.
\]
Here, if a nowhere-vanishing section of $\det{s} \to s^{-1}(0)$ is given and if $B$ is oriented, $\frakm(s, \alpha) \in \Z$, and otherwise $\in \Z/2$.
\item
If $B$ is a compact manifold with boundary and $s$ is nowhere-vanishing on $\del B$,  we define a number $\frakm(s, \alpha; \del B)$ by
\begin{align*}
\frakm(s, \alpha; \del B) := \left< \frakM(s, \alpha; \del B), [B, \del B]\right> \in \Z \text{ or } \Z/2.
\end{align*}
Here whether $\frakm(s, \alpha; \del B)$ is in $\Z$ or not is similar to the case \eqref{case of absolute evaluation of cohomological invariant}.
\end{enumerate}
\end{defi}

These numbers $\frakm(s, \alpha)$ and $\frakm(s, \alpha; \del B)$ correspond to the``counted number $\#\M$'' for the parameterized moduli space on $B$.

We remark a lemma relating the ``naturality'' at the end of this section:

\begin{lem}
\label{naturality for vn general}
Let $\scrE \to \scrX \to B$ and $s : \scrX \to \scrE$ be as above, $A$ be a normal space, and $f : A \to B$ be a continuous map.
Assume that, for the pull-backed section $f^{\ast}s : f^{\ast}\scrX \to f^{\ast}\scrE$, the parameterized zero set $(f^{\ast}s)^{-1}(0)$ is also compact.
Let $\calU$ be a family \vn for the parameterized Fredholm section $s$.
Then, there exist a family \vn $f^{\ast}\calU$ for the pull-backed section $f^{\ast}s$ and a continuous map $\tilde{f} : f^{\ast}\calU \to \calU$ such that $\tilde{f}$ covers $f$ and
\begin{align}
\tilde{f}^{\ast}e_{\calU} = e_{f^{\ast}\calU}
\label{eq: equality between cohomology classes from vn}
\end{align}
holds.
\end{lem}

\begin{proof}
Let $\vp : \scrX \times \R^{N} \to \scrE$ be the finite dimensional perturbation used to define $\calU$.
Then $f^{\ast}\calU$ is constructed as the \vn obtained from $f^{\ast}\vp : f^{\ast}\scrX \times \R^{N} \to f^{\ast}\scrE$.
Let $\bar{f} : f^{\ast}\scrX \to \scrX$ be the natural map covering $f$.
Then the map $\tilde{f}$ is given as the restriction of $\bar{f} \times \id_{\R^{N}} : f^{\ast}\scrX \times \R^{N} \to \scrX \times \R^{N}$.
If we take the open manifold $\mathscr{N}$ used in the definition of $\calU$ and that for $f^{\ast}\calU$ to be sufficiently small, one can check that the equality 
\[
(\tilde{f} \times \id_{\R^{N}})^{\ast} \tau(\calU \times \R^{N})
= \tau(f^{\ast}\calU \times \R^{N})
\]
between the Thom classes holds in
\[
H^{N}(f^{\ast}\calU \times D(\R^{N}), f^{\ast}\calU \times S(\R^{N})) \cong H^{0}(f^{\ast}\calU) \otimes R,
\]
where $R = \Z$ or $\Z/2$ is the coefficient ring.
This implies the equality \eqref{eq: equality between cohomology classes from vn} between the Euler classes.
\end{proof}

\begin{cor}
\label{naturality for vn}
Let $\scrE \to \scrX \to B$, $s : \scrX \to \scrE$ , and $f : A \to B$ be that given in \cref{naturality for vn general} and assume that $(f^{\ast}s)^{-1}(0)$ is also compact.
Then
\[
f^{\ast} \frakM(s, \alpha)
= \frakM(s, \bar{f}^{\ast}\alpha)
\]
holds for any $\alpha \in H^{\ast}(\scrX)$, where $\bar{f} : f^{\ast}\scrX \to \scrX$ is the natural map covering $f$.
Similarly, for subsets $A' \subset A$ and $B' \subset B$ satisfying that $f(A') \subset B'$ and that $s$ is nowhere-vanishing on $A'$, we have
\[
f^{\ast} \frakM(s, \alpha; B')
= \frakM(s, \bar{f}^{\ast}\alpha; A').
\]
\end{cor}

\begin{proof}
This is because pull-back commutes with integration along the fiber.
\end{proof}

\section{Construction of the characteristic classes}
\label{section Construction of the characteristic classes}

The aim of this section is to construct characteristic classes bundles of $4$-manifolds via ASD/\SW equations.
The procedure of the construction is some analogy of that of obstruction theory.
The basic reason why the analogy works is that we have an interpretation of the Donaldson/\SW invariants as the ``Euler classes" of some Hilbert bundles, explained as the introduction.
To consider the Euler class in the rigorous sense, we use some finite dimensional approximations of the Hilbert bundles, given in \cref{section Virtual neighborhood for families}.
We note that the story of this \lcnamecref{section Construction of the characteristic classes} is similar to that of Section~3 in \cite{Konno2}, in which the author has given the construction of the cohomological \SW invariant associated with the adjunction complex of surfaces.

\subsection{Main construction}
\label{subsection Definition of the characteristic classes in the basic case}

Let $X$ be an oriented closed smooth $4$-manifold.
Let us choose one of the ASD setting or the SW setting, and work on it in this \lcnamecref{subsection Definition of the characteristic classes in the basic case}.
We shall eventually define characteristic classes for bundles on a general topological space, but let $B$ be a CW complex until \cref{well-definedness of the invariant on CW str}.
Let $n$ be a non-negative integer and suppose that $b^{+}(X) \geq n+2$.
Assume that the formal dimension of the moduli space is $-n$.
In this \lcnamecref{subsection Definition of the characteristic classes in the basic case} we define our characteristic classes under these assumptions.
Let $B^{(n)}$ denote the $n$-skeleton of $B$.
We first construct a section
\begin{align}
\si = \si^{(n)} : B^{(n)} \to \Pi(E)|_{B^{(n)}}
\label{eq: section from Bn to perturbations}
\end{align}
inductively as follows.
For each $b \in B^{(0)}$, take a generic point in $\Pi(E_{b})$.
Then we have $\si^{(0)} : B^{(0)} \to \Pi(E)|_{B^{(0)}}$.
Assume that we have constructed $\si^{(k-1)} : B^{(k-1)} \to \Pi(E)|_{B^{(k-1)}}$ for $k \leq n$ such that the parameterized moduli space for $\si^{(k-1)}$ is empty:
$\M_{\si^{(k-1)}} = \emptyset$.
Here $\M_{\si^{(k-1)}}$ is the space obtained by substituting $\si^{(k-1)}$ for $\si$ in \cref{defi: globally defined parameterized moduli space}.
Note that this condition on $\si^{(k-1)}$ implies that $\M_{\si^{(k-1)}}$ contains no reducible solution.
Let $e \subset B$ be a $k$-cell and $\vp_{e} : D_{e}^{k} \to \bar{e} \subset B$ be the characteristic map of $e$.
Here $D^{k}_{e}$ is the standard $k$-dimensional disk indexed by $e$.
Since the bundle $\vp_{e}^{\ast} \Pi(E) \to D_{e}^{k}$ is trivial, we can take a trivialization $\psi_{e} : \vp_{e}^{\ast} \Pi(E) \to D_{e}^{k} \times \Pi(X)$.
Let us consider the continuous map obtained as the composition
\[
p_{2} \circ \psi_{e} \circ (\vp_{e}|_{\del D_{e}^{k}})^{\ast}\si^{(k-1)}
: \del D_{e}^{k} \to (\vp_{e}|_{\del D_{e}^{k}})^{\ast} \Pi(E) \to \del D_{e}^{k} \times \Pi(X) \to \Pi(X),
\]
where $p_{2} : \del D_{e}^{k} \times \Pi(X) \to \Pi(X)$ is the projection.
Since $b^{+}(X) \geq n+1$, we can smoothly and generically extend this map to a map from $D_{e}^{k}$ into $\Pi(X)$ avoiding the wall.
(Here the term ``smoothly'' means smoothness in the interior of $D_{e}^{k}$.)
This extended map gives a section $\bar{e} \to \Pi(E)|_{\bar{e}}$.
We therefore obtain $\si^{(k)} : B^{(k)} \to \Pi(E)|_{B^{(k)}}$.
In fact this procedure can be continued until $k=n+1$ since we assume $b^{+}(X) \geq n+2$, but we stop it until $k=n$ in this subsection.
(The case that $k=n+1$ is used in \cref{subsection Well-definedness}.)
We call such $\si$ an {\it inductive section}.
Note that whether $\M_{\si^{(k)}}$ is empty or not is independent of the choice of local trivialization.
This is because another choice of trivialization induces a bijection as in \eqref{isomorphism by pull back between moduli spaces}.
To detect $\M_{\si^{(k)}}$ is empty or not, we can therefore use the smooth structure of the pull-backed moduli space on $D^{n-1}_{e}$ via the trivializations $\psi_{e}$.
Thus we have $\M_{\si^{(k)}} = \emptyset$ for $k<n$ because of formal dimension.

For each $n$-cell $e$, substituting  $\vp_{e}^{\ast}\si$ for $\si_{\si}$ in \cref{defi: family of Fredholm sections in the most general setting}, we can consider the family of Fredholm sections $s_{\vp_{e}^{\ast}\si}$ corresponding to $\vp_{e}^{\ast}\si$ and parameterized on $D^{n}_{e}$.
As we mentioned, the section \eqref{eq: section from Bn to perturbations} satisfies that $\M_{\si^{(n-1)}} = \emptyset$.
Namely, for each $n$-cell $e$, the family of Fredholm sections $s_{\vp_{e}^{\ast}\si}$ is nowhere-vanishing on $\del D_{e}^{n}$.
In addition, the parameterized moduli space $\M_{\vp_{e}^{\ast}s}$ is compact because of the compactness of $D_{e}^{n}$ and that of the usual (i.e. unparameterized) moduli space of the solutions to the $SO(3)$-ASD equation with non-trivial $w_{2}$ or to the \SW equations.
In other words, the family of Fredholm sections $s_{\vp_{e}^{\ast}\si}$ satisfies the compactness assumption corresponding to \cref{assumption on compactness for parameterized moduli}.
We can therefore have the number
\[
\frakm(s_{\vp_{e}^{\ast}\si}, 1; D_{e}^{n}) \in \Z \text{ or } \Z/2
\]
by substituting $s_{\vp_{e}^{\ast}\si}$ and $1 \in H^{0}(\scrB^{\ast}_{\vp_{e}^{\ast}\si})$ for $s$ and $\alpha$ in \cref{defi of the evaluated invariant}, where $\scrB^{\ast}_{\vp_{e}^{\ast}\si}$ is the Hilbert manifold obtained by substituting $\vp_{e}^{\ast}\si$ for $\si$ in \cref{defi family of scrB and scrE in the most general setting}.
Here if we consider the homology oriented case, the number is in $\Z$, and otherwise in $\Z/2$.
We write
\[
C_\ast(B),\quad
C^\ast(B),\quad
\del : C_\ast(B) \to C_{\ast-1}(B),\quad
\codel : C^\ast(B) \to C^{\ast+1}(B)
\]
for the (cellular) chain complex, the cochain complex, the boundary operator, and the coboundary operator respectively, where the coefficient is $\Z$ if we consider the homology oriented case, and otherwise is $\Z/2$.
In what follows, we keep this convention on coefficient.

\begin{defi}
\label{defi of the basic cochain}
For the section $\si$ constructed above, we define a cochain 
\[
\Acoch(E, \si) \in C^{n}(B)
\]
by
\[
e \mapsto \frakm(s_{\vp_{e}^{\ast}\si}, 1; D_{e}^{n}).
\]
\end{defi}

We shall prove the following \lcnamecref{prop: coclosedness} in \cref{subsection Well-definedness}:

\begin{prop}
\label{prop: coclosedness}
The cochain $\Acoch(E, \si)$ constructed above is a cocycle.
\end{prop}

We now may write down the definition of our characteristic classes:

\begin{defi}
\label{def of characteristic classes in the base case}
We define
\[
\mathbb{A}(E) := [\Acoch(E,\si)] \in H^{n}(B),
\]
where $\si : B^{(n)} \to \Pi(E)|_{B^{(n)}}$ is an inductive section.
\end{defi}

This cohomology class is invariants of $E$:

\begin{theo}
\label{theo: well-definedness of the invariant}
The cohomology class $\Acoh(E)$ given in \cref{def of characteristic classes in the base case} is independent of the choice of $\si$.
\end{theo}

We shall prove \cref{theo: well-definedness of the invariant} in \cref{subsection Well-definedness}.

\begin{rem}
\label{rem on relation between vanishing and triviality}
If we assume \cref{theo: well-definedness of the invariant}, we can immediately see that, if $E$ is a trivial $G$-bundle and $n>0$, then $\Acoh(E) = 0$.
Indeed, in this case, we can take a ``constant" inductive section $\si$, then we have $\Acoch(E,\si) = 0$ because of formal dimension.
\end{rem}

In addition, $\Acoh(\cdot)$ satisfies functoriality.
Namely, $\Acoh(\cdot)$ is a characteristic class:

\begin{theo}
\label{part of theo of well-definedness of the invariant functoriality}
The correspondence $E \mapsto \Acoh(E)$ is functorial.
Namely, for a CW complex $B'$ and a continuous map $f : B' \to B$, we have
\[
f^{\ast} \Acoh(E) = \Acoh(f^{\ast}E).
\]
\end{theo}

We shall prove \cref{part of theo of well-definedness of the invariant functoriality}  in \cref{subsection Well-definedness}.
By applying \cref{part of theo of well-definedness of the invariant functoriality} to the identity map, we also have the following independence of $\Acoh(E)$ on CW structure of the base space:

\begin{cor}
\label{well-definedness of the invariant on CW str}
The cohomology class $\Acoh(E)$ given in \cref{def of characteristic classes in the base case} is independent of the choice of CW structure of $B$.
\end{cor}

In fact, we can functorially extend the definition of $\Acoh(\cdot)$ to any bundle on any topological space
using a purely topological \lcnamecref{lem for a general topological base space}:

\begin{lem}
\label{lem for a general topological base space}
For any topological space $B$, we can associate a cohomology class $\Acoh(E) \in H^{n}(B)$ to any continuous fiber bundle $X \to E \to B$ with structure group $G$, and this correspondence $E \mapsto \Acoh(E)$ satisfies that:
\begin{itemize}
\item if $B$ is a CW complex, this $\Acoh(E)$ coincides with that given in \cref{def of characteristic classes in the base case}, and
\item for a topological space $B'$ and a continuous map $f : B' \to B$, we have $f^{\ast} \Acoh(E) = \Acoh(f^{\ast}E)$.
\end{itemize}
\end{lem}

\begin{proof}
Recall that, for any topological space $B$, there exists a CW complex $|\Delta(B)|$ equipped with a weak homotopy equivalence map $\rho_{B} : |\Delta(B)| \to B$.
The space $|\Delta(B)|$ is obtained by considering the geometric realization of the simplicial set arising from singular simplices, and this construction $B \mapsto |\Delta(B)|$ is functorial with respect to the map $\rho_{B} : |\Delta(B)| \to B$.
(For example, see May's book~\cite{MR1702278}.)
For a bundle $X \to E \to B$, let us define
\[
\Acoh(E) := (\rho_{B}^{\ast})^{-1} \Acoh(\rho_{B}^{\ast}E) \in H^{n}(B),
\]
where $\Acoh(\cdot)$ in the right-hand side is the one defined in \cref{def of characteristic classes in the base case}.
Then it is straightforward to check the required conditions. 
\end{proof}

Write $\Dcoh(\cdot) := \Acoh(\cdot)$ and $\SWcoh(\cdot) := \Acoh(\cdot)$ for the ASD setting and the SW setting respectively.
We summarize the results in this \lcnamecref{subsection Definition of the characteristic classes in the basic case} as follows.

\begin{theo}
\label{Dcoh SWcoh in the base case}
Let $n$ be a non-negative integer, $X$ be an oriented closed smooth $4$-manifold with $b^{+}(X) \geq n+2$, and $B$ be a topological space.
\begin{enumerate}
\item
\label{Dcoh in the base case}
Let $\frakP$ be the isomorphism class of an $SO(3)$-bundle satisfying that $w_{2}(\frakP) \neq 0$ and $d(\frakP) = -n$.
\begin{enumerate}
\item
\label{association of characteristic class ASD wo homology ori}
To a continuous fiber bundle $X \to E \to B$ with structure group $\Diff(X, \frakP)$, we can associate
\[
\Dcoh(E) \in H^{n}(B;\Z/2).
\]
\item
\label{association of characteristic class ASD with homology ori}
To a continuous fiber bundle $X \to E \to B$ with structure group $\Diff(X, \frakP, \calO)$, we can associate
\[
\Dcoh(E) \in H^{n}(B;\Z).
\]
\end{enumerate}
In both case of \eqref{association of characteristic class ASD wo homology ori} and \eqref{association of characteristic class ASD with homology ori}, the correspondence $E \mapsto \Dcoh(E)$ is functorial.
\item
\label{SWcoh in the base case}
Let $\fraks$  be the isomorphism class of a $\spc$ structure on $X$ with $d(\fraks) = -n$.
\begin{enumerate}
\item
\label{association of characteristic class SW wo homology ori}
To a continuous fiber bundle $X \to E \to B$ with structure group $\Diff(X, \fraks)$, we can associate
\[
\SWcoh(E) \in H^{n}(B;\Z/2).
\]
\item
\label{association of characteristic class SW with homology ori}
To a continuous fiber bundle $X \to E \to B$ with structure group $\Diff(X, \fraks, \calO)$, we can associate
\[
\SWcoh(E) \in H^{n}(B;\Z).
\]
\end{enumerate}
In both case of \eqref{association of characteristic class SW wo homology ori} and \eqref{association of characteristic class SW with homology ori}, the correspondence $E \mapsto \SWcoh(E)$ is functorial.
\end{enumerate}
\end{theo}

\begin{defi}
\label{defi value for universal bundles}
In the setting of \cref{Dcoh SWcoh in the base case}, we define
\begin{align*}
&\Dcoh(X, \frakP) := \Dcoh(E\Diff(X, \frakP)) \in H^{n}(B\Diff(X, \frakP);\Z/2),\\
&\Dcoh(X, \frakP, \calO) := \Dcoh(E\Diff(X, \frakP, \calO)) \in H^{n}(B\Diff(X, \frakP, \calO);\Z)\\
&\SWcoh(X, \fraks) := \SWcoh(E\Diff(X, \fraks)) \in H^{n}(B\Diff(X, \fraks);\Z/2), \text{ and }\\
&\SWcoh(X, \fraks, \calO) := \SWcoh(E\Diff(X, \fraks, \calO)) \in H^{n}(B\Diff(X, \fraks, \calO);\Z).
\end{align*}
\end{defi}

Of course the statement of \cref{Dcoh SWcoh in the base case} has no meaning if we cannot show the non-triviality of them.
In \cref{section Non triviality} we shall explicitly calculate some of these characteristic classes.
Note that, in the case that $n=0$, the cohomology classes given in \cref{defi value for universal bundles} are nothing other than the usual $SO(3)$-Donaldson invariants and \SW invariants (valued in $\Z/2$ and in $\Z$) defined by counting the moduli space of formal dimension zero.
Therefore for $n=0$ theses classes are obviously non-trivial, and so we are interested in the non-triviality in the case that $n>0$, which is the subject of \cref{section Non triviality}.

\begin{rem}
\label{rem stack}
We have assumed that the formal dimension of the moduli space is $-n$ in this \lcnamecref{subsection Definition of the characteristic classes in the basic case}.
We can relax the assumptions using non-trivial cohomology class $\alpha$ appeared in \cref{section Virtual neighborhood for families}, which corresponds to cutting of higher-dimensional moduli spaces.
However, the following technical issue arises to consider such a generalization of characteristic classes.
Let us focus on the SW case for simplicity.
In the unparameterized situation, to obtain a cohomology class of positive degree on the irreducible configuration space divided by the gauge group, denoted by $\scrB^{\ast}$, one needs to consider an $S^{1}$-fibration $S^{1} \to \calL \to \scrB^{\ast}$.
The total space $\calL$ is given as the quotient of the configuration space divided by a subgroup $\gauge_{0}$ of $\gauge$ inducing an exact sequence
\[
1 \to \gauge_{0} \to \gauge \to S^{1} \to 1.
\]
Because of the rest symmetry of $S^{1}$, we cannot apply the argument of \cref{subsection Nakamura's idea and families of Fredholm sections} to $\calL$: for a family with structure group $\Diff(X, \fraks)$ or $\Diff(X, \fraks, \calO)$,
we cannot obtain a line bundle which is globally defined on the whole base space by gluing together $\calL$'s.
However, we hope one can avoid this problem by considering some stacks rather than line bundles, which will be discussed in a subsequent paper.
\end{rem}

\subsection{Well-definedness and naturality}
\label{subsection Well-definedness}

The purpose of this \lcnamecref{subsection Well-definedness} is to prove \cref{prop: coclosedness,theo: well-definedness of the invariant,part of theo of well-definedness of the invariant functoriality}.
Before starting it, we note the following \lcnamecref{lem expression of chchain as global cohomology class} on an expression of the cochain $\Acoch(\cdot)$.
Henceforth we identify $C^{n}(B)$ with $H^{n}(B^{(n)}, B^{(n-1)})$.

\begin{lem}
\label{lem expression of chchain as global cohomology class}
In the setting of \cref{defi of the basic cochain}, assume that $B$ is compact.
Then we have
\begin{align}
\frakM(s_{\si}, 1;B^{(n-1)}) = \Acoch(E,\si)
\label{global cohomology class coincides with Acoch}
\end{align}
in $H^{n}(B^{(n)}, B^{(n-1)}) = C^{n}(B)$.
\end{lem}

\begin{proof}
We first note that the left-hand side of \eqref{global cohomology class coincides with Acoch} can be defined in $H^{n}(B^{(n)}, B^{(n-1)})$ since $B$ is compact and $\M_{\si^{(n-1)}}=\emptyset$ holds.
Let $e$ be an $n$-cell of $B$ and $\vp_{e} : D^{n}_{e} \to B$ be its characteristic map.
Because of the compactness of $B$ and $D^{n}_{e}$, we can apply \cref{naturality for vn} to $\vp_{e} : D^{n}_{e} \to B$, and thus we obtain 
\begin{align*}
\vp_{e}^{\ast}\frakM(s_{\si}, 1 ; B^{(n-1)}) = \frakM(\vp_{e}^{\ast}s_{\si}, 1; \del D^{n}_{e}).
\end{align*}
The equality and the isomorphism
\[
\prod_{e \subset B} \vp_{e}^{\ast} : C^{n}(B) = H^{n}(B^{(n)}, B^{(n-1)}) \to \prod_{e \subset B} H^{n}(D^{n}_{e}, \del D^{n}_{e})
\]
imply \eqref{global cohomology class coincides with Acoch}.
\end{proof}

We now prove the naturality at the level of cochains.

\begin{prop}
\label{lem on naturality at the level of cochain}
Let us follow the setting of \cref{defi of the basic cochain}.
Then, for a CW complex $B'$ and a cellular map $f : B' \to B$, 
\begin{align}
f^{\ast} \Acoch(E, \si)
= \Acoch(f^{\ast} E, f^{\ast}\si)
\label{eq cochain level naturality}
\end{align}
holds in $C^{n}(B')$.
\end{prop}

\begin{proof}
We first note that the notation $\Acoch(f^{\ast} E, f^{\ast}\si)$ makes sense:
$f^{\ast}\si$ is also an inductive section on $B'^{(n)}$, which is a straightforward verification.
Let $e \subset B'$ be an $n$-cell and $\vp_{e} : D^{n}_{e} \to B'$ be the characteristic map of $e$.
Take a finite subcomplex $K' \subset B'$ containing the image of $\vp_{e}$, and also take a finite subcomplex $K \subset B$ containing the image $f(K')$.
By applying \cref{lem expression of chchain as global cohomology class} to $\si|_{K^{(n)}}$, we have
\begin{align}
\frakM(s_{\si|_{K^{(n)}}}, 1;K^{(n-1)}) = i_{K}^{\ast} \Acoch(E,\si)
\label{relating big cohomological inv to small 0}
\end{align}
in $C^{n}(K^{(n)}) = C^{n}(K)$, where $i_{K} : K \inc B$ is the inclusion.
On the other hand, because of the compactness of $K^{(n)}$, $K'^{(n)}$ and $D^{n}_{e}$, we can apply \cref{naturality for vn} to the restriction $f : K'^{(n)} \to K^{(n)}$ of $f$ and $\vp_{e}$, and thus have 
\begin{align}
\vp_{e}^{\ast}f^{\ast}\frakM(s_{\si|_{K^{(n)}}}, 1; K^{(n-1)}) = \frakM(\vp_{e}^{\ast}f^{\ast}s_{\si|_{K^{(n)}}}, 1; \del D^{n}_{e}).
\label{relating big cohomological inv to small 1}
\end{align}
Since $\vp_{e}^{\ast}f^{\ast}s_{\si} = s_{\vp_{e}^{\ast}f^{\ast}\si}$ holds, the equalities \eqref{relating big cohomological inv to small 0} and \eqref{relating big cohomological inv to small 1} imply that
\[
\vp_{e}^{\ast}f^{\ast}\Acoch(E,\si)
= \frakM(s_{\vp_{e}^{\ast}f^{\ast}\si}, 1; \del D^{n}_{e}).
\]
It therefore follows that
\begin{align*}
&f^{\ast} \Acoch(E, \si)(e)
= \left<f^{\ast} \Acoch(E, \si), e\right>
= \left<\vp_{e}^{\ast}f^{\ast} \Acoch(E, \si), [D^{n}_{e}, \del D^{n}_{e}]\right>\\
=& \left<\frakM(s_{\vp_{e}^{\ast}f^{\ast}\si}, 1; \del D^{n}_{e}), [D^{n}_{e}, \del D^{n}_{e}]\right>
= \left<\Acoch(f^{\ast}E, f^{\ast}\si),e \right>.
\end{align*}
Thus we obtain \eqref{eq cochain level naturality}.
\end{proof}

\begin{proof}[Proof of \cref{part of theo of well-definedness of the invariant functoriality}]
This is a direct consequence of \cref{lem on naturality at the level of cochain} and the cellular approximation theorem.
(Here we interpret $\Acoh(f^{\ast}E)$ as $[\Acoch(f^{\ast}E, f^{\ast}\si)]$, which is in fact independent of $f^{\ast}\si$ by \cref{prop: coclosedness}.)
\end{proof}

We now give the proof of \cref{prop: coclosedness,theo: well-definedness of the invariant}.
Both of them are shown using $(n+1)$-parameter families, corresponding to so-called arguments by cobordisms.
However, in fact, any cobordism between manifolds does not explicitly appear in this \lcnamecref{subsection Well-definedness}.
This is because it has been absorbed into the well-definedness of the cohomological invariants (\cref{lemma on well definedness of the cohomological inv in family vn context}) in family \vn context.

\begin{proof}[Proof of \cref{prop: coclosedness}]
As we noted in \cref{subsection Definition of the characteristic classes in the basic case}, an inductive section $\si^{(\bullet)}$ can be constructed also for $\bullet=n+1$ since $b^{+}(X) \geq n+2$.
Let us fix $\si^{(n+1)} : B^{(n+1)} \to \Pi(E)|_{B^{(n+1)}}$ through the inductive procedure in \cref{subsection Definition of the characteristic classes in the basic case}.
By abuse of notation, we write $\si$ for both $\si^{(n)}$ and $\si^{(n+1)}$.
Take an $(n+1)$-cell $e \subset B$ and its characteristic map $\vp_{e} : \Delta^{n+1} \to B$.
Here we equip $D^{n}_{e} \cong \Delta^{n+1}$ with the CW structure as the standard simplex, and by cellular approximation, we can assume that $\vp_{e}$ is cellular with respect to this CW structure on $\Delta^{n+1}$.
From \cref{lem on naturality at the level of cochain}, it follows that
\begin{align}
\label{eq between codelta and evaluation of deltadelta}
&\delta \Acoch(E, \si)(e)
= \delta \Acoch(E, \si)({\vp_{e}}_{\ast}\Delta^{n+1})
= \vp_{e}^{\ast}\delta \Acoch(E, \si)(\Delta^{n+1})\\
=& \delta \vp_{e}^{\ast} \Acoch(E, \si)(\Delta^{n+1})
= \delta \Acoch(\vp_{e}^{\ast}E, \vp_{e}^{\ast}\si)(\Delta^{n+1})
= \Acoch(\vp_{e}^{\ast}E, \vp_{e}^{\ast}\si)(\del\Delta^{n+1}).
\nonumber
\end{align}
On the other hand, by applying \cref{lem expression of chchain as global cohomology class} to $\vp_{e}^{\ast}\si|_{(\Delta^{n+1})^{(n)}}$, we have
\begin{align}
\frakM(s_{\vp_{e}^{\ast}\si|_{(\Delta^{n+1})^{(n)}}}, 1;(\Delta^{n+1})^{(n-1)}) = \Acoch(\vp_{e}^{\ast}E, \vp_{e}^{\ast}\si).
\label{eq between delta n1 and coch}
\end{align}
In addition, since $\M_{\vp_{e}^{\ast}\si^{(n-1)}} = \emptyset$ holds, we can define
\[
\frakM(s_{\vp_{e}^{\ast}\si}, 1;(\Delta^{n+1})^{(n-1)})
\in H^{n}(\Delta^{n+1}, (\Delta^{n+1})^{(n-1)})
\]
and 
\begin{align}
i^{\ast}\frakM(s_{\vp_{e}^{\ast}\si}, 1;(\Delta^{n+1})^{(n-1)})
= \frakM(s_{\vp_{e}^{\ast}\si|_{(\Delta^{n+1})^{(n)}}}, 1;(\Delta^{n+1})^{(n-1)})
\label{eq between pull back coch delta n1}
\end{align}
holds, where
\[
i : ((\Delta^{n+1})^{(n)}, (\Delta^{n+1})^{(n-1)}) \inc (\Delta^{n+1}, (\Delta^{n+1})^{(n-1)})
\]
is the inclusion.
From the equalities \eqref{eq between codelta and evaluation of deltadelta}, \eqref{eq between delta n1 and coch}, \eqref{eq between pull back coch delta n1}, and $i_{\ast}\del\Delta^{n+1}=0$, it follows that
\begin{align*}
\delta \Acoch(E, \si)(e)
= \left<\Acoch(\vp_{e}^{\ast}E, \vp_{e}^{\ast}\si), \del\Delta^{n+1} \right>
&= \left<i^{\ast}\frakM(s_{\vp_{e}^{\ast}\si}, 1;(\Delta^{n+1})^{(n-1)}), \del\Delta^{n+1} \right>\\
&= \left<\frakM(s_{\vp_{e}^{\ast}\si}, 1;(\Delta^{n+1})^{(n-1)}), i_{\ast}\del\Delta^{n+1} \right> = 0.
\end{align*}
Thus we have $\delta \Acoch(E, \si)=0$.
\end{proof}

\begin{proof}[Proof of \cref{theo: well-definedness of the invariant}]
Let us take two inductive sections $\si_{i} : B^{(n)} \to \Pi(E)|_{B^{(n)}}$ for $i = 0,1$, and set
$\Acoch_{i} := \Acoch(E, \si_{i})$.
Let $p : B \times [0,1] \to B$ the the projection.
For each cell $e$ of $B$, let $\vp_{e} : D^{\dim{e}} \to B$ be its characteristic map and fix a trivialization
\[
\tilde{\psi}_{e} : (\vp_{e}\times \id)^{\ast} p^{\ast} \Pi(E) \to D^{\dim{e}} \times [0,1] \times \Pi(X).
\]
We inductively construct a section $\tilde{\si} = \tilde{\si}^{(n)}: B^{(n)} \times [0,1] \to p^{\ast}\Pi(E)|_{B^{(n)}}$ as follows.
First let $e \subset B$ be a $0$-cell.
We take a section $\tilde{\si}^{(0)} : e \times [0,1] \to p^{\ast}\Pi(E)|_{e}$ such that $\tilde{\si}^{(0)}|_{e \times \{i\}} = \si_{i}^{(0)}$ and the composition
\[
p_{2} \circ \tilde{\psi}_{e} \circ (\vp_{e}\times \id)^{\ast} \tilde{\si}^{(0)} : D^{0}_{e} \times (0,1) \to (\vp_{e}\times \id)^{\ast} p^{\ast} \Pi(E)
\to D^{0} \times [0,1] \times \Pi(X)
\to \Pi(X)
\]
is smooth (in the interior of the domain) and generic, and avoiding the wall.
Then we obtain $\tilde{\si}^{(0)} : B^{(0)} \to p^{\ast} \Pi(E)|_{B^{(0)}}$.
We next assume that we have constructed $\tilde{\si}^{(k-1)} : B^{(k-1)} \to p^{\ast}\Pi(E)|_{B^{(k-1)}}$ for $k \leq n$ such that $\tilde{\si}^{(k-1)}|_{e \times \{i\}} = \si_{i}^{(k-1)}$ and $\M_{\tilde{\si}^{(k-1)}}=\emptyset$.
For a $k$-cell $e \subset B$, we take a section $\tilde{\si}^{(k)} : e \times [0,1] \to p^{\ast}\Pi(E)|_{e}$ such that $\tilde{\si}^{(k)}|_{e \times \{i\}} = \si_{i}^{(k)}$, $\tilde{\si}^{(k)}|_{\bar{e}\setminus {e}} = \tilde{\si}^{(k-1)}|_{\bar{e}\setminus {e}}$, and the composition
\[
p_{2} \circ \tilde{\psi}_{e} \circ (\vp_{e}\times \id)^{\ast} \tilde{\si}^{(k)} : D^{k}_{e} \times (0,1) \to \Pi(X)
\]
is smooth and generic, and avoiding the wall.
We can assume that this composition avoids the wall since we assume that $b^{+}(X) \geq n+2$.
We now obtain $\tilde{\si}^{(k)} : B^{(k)} \to p^{\ast} \Pi(E)|_{B^{(k)}}$, and thus $\tilde{\si} : B^{(n)} \times [0,1] \to p^{\ast}\Pi(E)|_{B^{(n)}}$.
Let us define
\[
\tilA \in C^{n-1}(B)
\]
by
\[
\tilA(e) := (-1)^{n-1} \Acoch(p^{\ast}E, \tilde{\si})(e \times I)
\]
for each $(n-1)$-cell $e$ of $B$,
where $I$ is the $1$-cell of $[0,1]$ equipped with the standard cell structure.
We shall show that $\codel \tilA = \Acoch_{1} - \Acoch_{0}$.
Let $R = \Z$ or $\Z/2$ be the our coefficient, and let us write the basis of $C_{\ast}(I)$ as $C_{0}(I) = R \cdot 0 \oplus R \cdot 1$, $C_{1}(I) = R \cdot I$, and write the dual basis as $C^{0}(I) = R \cdot 0^{\ast} \oplus R \cdot 1^{\ast}$, $C^{1}(I) = R \cdot I^{\ast}$.
Let $\Phi : C^{\ast}(B) \to C^{\ast+1}(B \times [0,1])$ be the isomorphism given by $\Phi(c) := c \otimes I^{\ast}$, which commutes with $\codel$.
Let us consider an $n$-cell of $B \times [0,1]$ written as
\begin{align*}
\tilde{e} &= e^{n-1} \otimes I + e^{n}_{0} \otimes 0 + e^{n}_{1} \otimes 1\\
&\in C_{n-1}(B) \otimes C_{1}([0,1]) \oplus C_{n}(B) \otimes C_{0}([0,1])
\cong C_{n}(B \times [0,1]).
\end{align*}
Then, since $\Acoch(p^{\ast}E, \tilde{\si})(e \times i) = \Acoch_{i}(e)$ holds for each $n$-cell $e$ of $B$, we have
\begin{align*}
\Phi \tilA(\tilde{e})
= \tilA(e^{n-1})
&= (-1)^{n-1} \Acoch(p^{\ast}E, \tilde{\si})(e^{n-1} \times I)\\
&= (-1)^{n-1} \{\Acoch(p^{\ast}E, \tilde{\si})(\tilde{e})
- \Acoch_{0}(e^{n}_{0}) - \Acoch_{1}(e^{n}_{1})\}\\
&= (-1)^{n-1} \{\Acoch(p^{\ast}E, \tilde{\si})(\tilde{e})
- \Acoch_{0} \otimes 0^{\ast}(\tilde{e}) - \Acoch_{1} \otimes 1^{\ast}(\tilde{e})\}.
\end{align*}
It therefore follows that
\[
\Phi \tilA
= (-1)^{n-1} (\Acoch(p^{\ast}E, \tilde{\si})
- \Acoch_{0} \otimes 0^{\ast} - \Acoch_{1} \otimes 1^{\ast}).
\]
Using this equality and \cref{prop: coclosedness}, we have
\begin{align*}
\Phi \codel \tilA
= \codel \Phi \tilA
&= (-1)^{n-1} \codel (\Acoch(p^{\ast}E, \tilde{\si})
- \Acoch_{0} \otimes 0^{\ast} - \Acoch_{1} \otimes 1^{\ast})\\
&= (-1)^{n} \{(-1)^{n}\Acoch_{0} \otimes \codel 0^{\ast} + (-1)^{n}\Acoch_{1} \otimes \codel 1^{\ast}\}\\
&= -\Acoch_{0} \otimes I^{\ast} + \Acoch_{1} \otimes I^{\ast}
= \Phi(-\Acoch_{0} + \Acoch_{1}).
\end{align*}
Since $\Phi$ is an isomorpshism, we obtain the required equality $\codel \tilA = \Acoch_{1} - \Acoch_{0}$.
\end{proof}

\begin{rem}
In the argument of this \lcnamecref{subsection Well-definedness}, the cohomological invariants given in \cref{general def of cohomological invariant from family vn} is fruitfully used.
This is one of the clear merits of our family \vn technique, and here we also mention another good point of it.
In the proof of \cref{theo: well-definedness of the invariant}, for $i = 0,1$, let $\psi_{e, i} : \vp_{e}^{\ast}\Pi(E) \to D^{n}_{e} \times \Pi(X)$ be the trivialization of $\vp_{e}^{\ast}\Pi(E) \to D^{n}_{e}$ used in the construction of $\si_{i}$.
We note that we do {\it not} need to take $\tilde{\psi}_{e}$ to be an extension of $\psi_{e, i}$:
although the composition
\[
p_{2} \circ \tilde{\psi}_{e} \circ \vp_{e}^{\ast} \si_{i}^{(n)} : D^{n}_{e} \times \{i\} \to \Pi(X)
\]
is not smooth in general, it does not matter since we here do not consider any cobordism between smooth moduli spaces.
This is also one of the advantages of the family \vn technique:
we do not need to connect by a path between two local trivializations.
If one tries to connect the trivializations $\psi_{e, 0}$ and $\psi_{e, 1}$ via a family of trivializations, we have to study $\pi_{0}(\Homeo(\Pi(X) \setminus \text{wall}))$,
where $\Homeo(\cdot)$ denotes the homeomorphism group.
Especially in the ASD setting, the structure of the wall is complicated, and hence $\Homeo(\Pi(X) \setminus \text{wall})$ is so.
The use of the family \vn technique allows us to avoid such a problem.
\end{rem}

\begin{rem}
We mention other possibilities of ways to construct gauge theoretic characteristic classes on the classifying space of the diffeomorphism group avoiding family \vn technique.
A simple way is to abandon $\Z$ or $\Z/2$-coefficient cohomology and to work over $\Q$.
Since any homology class over $\Q$ on any space, say the classifying space, can be represented as a smooth finite dimensional manifold, one may apply usual family gauge theory (and Nakamura's idea described in \cref{subsection Nakamura's idea and families of Fredholm sections}) to the representative.
For example, D.~McDuff~\cite{MR2396906} describes characteristic classes obtained based on the family Gromov--Witten invariant via this way.
Another possibility is to use a model of the classifying space which is a smooth infinite dimensional manifold.
(For example, see Kriegl--Michor~\cite{MR1471480}.)
However this approach, of course, involves a lot of subtle problems on the counting of the moduli space, and the author does not certain whether one can completely avoid family \vn technique via this way.
\end{rem}

\section{Characteristic classes as obstruction}
\label{section Characteristic classes as obstruction}

In this \lcnamecref{section Characteristic classes as obstruction} we interpret our characteristic classes $\Dcoh$ and $\SWcoh$ as obstructions to some structure on $4$-manifold bundles.
We first note that these characteristic classes are obstructions to fiberwise connected sum under suitable assumption on $b^{+}$.
Let $B$ be a topological space and $X_{i} \to E_{i} \to B$ $(i=1,2)$ be fiber bundles of oriented closed $4$-manifolds $X_{i}$.
Given sections $\calS_{i} : B \to E_{i}$, if the normal bundle of $\calS_{1}(B)$ in $E_{1}$ is isomorphic to that of $\calS_{2}(B)$ in $E_{2}$  fiber-preservingly and fiber-orientation-reversingly,
we can define the fiberwise connected sum $X_{1} \# X_{2} \to E_{1} \#_{f} E_{2} \to B$ by considering the connected sum $(E_{1})_{b} \# (E_{2})_{b}$ along small disklike neighborhoods of $\calS_{i}(b)$ for each $b \in B$.
We call such sections $\calS_{i}$ {\it compatible sections}.
Let $\frakP_{i}$, $\fraks_{i}$, and $\calO_{i}$ be the isomorphism class of an $SO(3)$-bundle with $w_{2}(\frakP_{i}) \neq 0$, the isomorphism class of a $\spc$ structure, and a homology orientation on $X_{i}$ respectively.
They define data $\frakP = \frakP_{1} \# \frakP_{2}$, $\fraks = \fraks_{1} \# \fraks_{2}$, and $\calO = \calO_{1} \# \calO_{2}$ on $X = X_{1} \# X_{2}$, and we call the relations between $\frakP$, $\fraks$, $\calO$ and $\frakP_{i}$, $\fraks_{i}$, $\calO_{i}$ the {\it connected sum relations}.
The characteristic classes $\Dcoh$ and $\SWcoh$ obstruct the fiberwise connected sum as follows.
We note that D.~Ruberman has pointed this kind of phenomena out as Theorem~3.3 in \cite{MR1734421}.

\begin{theo}
\label{theo vanishing fiber wiseconnected sum}
Let $X$ be an oriented closed smooth $4$-manifold, $B$ be a CW complex, $E \to B$ be a continuous $X$-bundle, and $n$ be a non-negative integer.
Assume that $b^{+}(X) \geq n+2$ and $B$ has finitely many $n$-cells, and also suppose that either
\begin{enumerate}
\item $d(\frakP) = -n$ for the isomorphism class $\frakP$ of an $SO(3)$-bundle on $X$ with $w_{2}(\frakP) \neq 0$, the structure group of $E$ reduces to $G = \Diff(X, \frakP)$ or $\Diff(X, \frakP, \calO)$ for some homology orientation $\calO$, and $\Dcoh(E) \neq 0$ or 
\item
\label{vanishing connected sum SW}
$d(\fraks) = -n$ for the isomorphism class $\fraks$ of a spin$^{c}$ structure on $X$, the structure group of $E$ reduces to $G = \Diff(X, \fraks)$ or $\Diff(X, \fraks, \calO)$ for some homology orientation $\calO$, and $\SWcoh(E) \neq 0$.
\end{enumerate}
Then, there are no bundles of oriented closed $4$-manifolds $X_{i} \to E_{i} \to B$ $(i=1,2)$ with structure group $G_{i}$ satisfying that $E = E_{1} \#_{f} E_{2}$ along some compatible sections $\calS_{i} : B \to E_{i}$ and that $b^{+}(X_{i}) > n$.
Here $G_{i}$ is given as $\Diff(X_{i}, \frakP_{i})$, $\Diff(X_{i}, \frakP_{i}, \calO_{i})$, $\Diff(X_{i}, \fraks_{i})$, or $\Diff(X_{i}, \fraks_{i}, \calO_{i})$ corresponding to $G$ for some $\frakP_{i}$, $\fraks_{i}$, or $\calO_{i}$ satisfying the connected sum relations.
\end{theo}

\begin{proof}
Let us write $\Acoh$ for $\Dcoh$ and $\SWcoh$ as \cref{section Construction of the characteristic classes}.
Assume that there exist $X_{i} \to E_{i} \to B$ $(i=1,2)$ with structure group $G_{i}$ satisfying that $E = E \#_{f} E_{2}$ along some sections $\calS_{i} : B \to E_{i}$ and that $b^{+}(X_{i}) > n$.
As in the usual vanishing theorem of the Donaldson invariant or \SW invariant for connected sum, consider a sequence of fiberwise metrics $g_{i,b,m} \in \Met((E_{i})_{b})$ $(b\in B, m \in \N)$ and $g_{b,m} = g_{1,b,m} \# g_{2,b,m} \in \Met(E_{b})$ such that $g_{b,m}$ has a neck whose radius converges to $0$ as $m \to \infty$ and $g_{i,b,m}$ is constant with respect to $m$ on the complement of a neighborhood of the disk used to define the neck.
If we work on the ASD setting, we can define inductive sections $\si_{i,m} : B^{(n)} \to \Pi(E_{i})|_{B^{(n)}}$ as $\si_{i,m}(b) = g_{i,b,m}$ by taking suitable $g_{i,b,m}$ in advance.
If we work on the SW setting, we take inductive sections $\si_{i,m} : B^{(n)} \to \Pi(E_{i})|_{B^{(n)}}$ so that $\si_{i,m}(b)$ vanishes near the disk used to define the neck, $\si_{i,m}(b)$ is constant with respect to $m$ on the complement of the disk, and $\pi(\si_{i,m}(b)) = g_{i,b,m}$ holds, where $\pi : \Pi((E_{i})_{b}) \to \Met((E_{i})_{b})$ is the projection.
The inductive sections $\si_{i,m}$ define the limiting inductive sections $\si_{i,\infty} : B^{(n)} \to \Pi(E_{i})|_{B^{(n)}}$ by an obvious way.
Let us define $\si_{m}(b) := \si_{1,m}(b) \# \si_{2,m}(b)$.
This new section $\si_{m} : B^{(n)} \to \Pi(E)|_{B^{(n)}}$ is also an inductive section.
Assume that, for any $m$, there exists an $n$-cell of $B$ such that the parameterized moduli space with respect to $\si_{m}$ on it is not empty.
Because of our finiteness assumption on $n$-cells of $B$, without loss of generality, we may assume that there exists an $n$-cell $e$ of $B$ such that, for any $m>>0$,
the parameterized moduli space with respect to $\si_{m}$ on $e$ is not empty.
Let $\vp_{e} : D^{n}_{e} \to B$ be the characteristic map of $e$.
By considering limit and removing the singularity, we get an element of the parameterized moduli pace with respect to $\si_{i, \infty}$ on $e$ for both $i=1,2$.
On the other hand, by the definition of inductive section, there exists trivializations $\psi_{e,i} : \vp_{e}^{\ast} \Pi(E_{i}) \to D^{n}_{e} \times \Pi(X_{i})$ such that $\si_{i, \infty}$ is smooth, generic and avoiding reducibles via $\psi_{e,i}$.
The parameterized moduli spaces with respect to $\si_{i,\infty}$ on $e$ therefore admit smooth structure via $\psi_{e,i}$, and hence one of them is empty because of formal dimension.
This contradicts the existence of the limiting elements discussed above.
This argument implies that there exists $m$ such that the parameterized moduli space with respect to $\si_{m}$ is empty.
Thus we have $\Acoch(E, \si_{m}) = 0$, and hence $\Acoh(E) = 0$.
\end{proof}

We next mention that $\SWcoh$ relates to families of positive scalar curvature metrics, as in D.~Ruberman~\cite{MR1874146} and \cite{Konno3} by the author.
For a given $4$-manifold $X$, we denote by $\PSC(X)$ the space of positive scalar curvature metrics on $X$.
For a fiber bundle $X \to E \to B$ with structure group $\Diff(X)$, the diffeomorphism group of $X$, on a topological space $B$, since $\Diff(X)$ acts on $\PSC(X)$, we get a fiber bundle $\PSC(E) \to B$ whose fiber is $\PSC(X)$.
The following \lcnamecref{theo obstruction to PSC} generalizes Proposition~2.4 in \cite{Konno3}.

\begin{theo}
\label{theo obstruction to PSC}
Let $X$ be an oriented closed smooth $4$-manifold, $\fraks$ be the isomorphism class of a spin$^{c}$ structure on $X$, $B$ be a CW complex, and $n$ be a non-negative number.
Assume that $\PSC(X) \neq \emptyset$, $b^{+}(X) \geq n+2$, $d(\fraks) = -n$, and $B^{(n)}$ is compact.
Let $E \to B$ be a bundle of $X$ with structure group $\Diff(X, \fraks)$ or $\Diff(X, \fraks, \calO)$ for some homology orientation $\calO$.
Then, if $\SWcoh(E) \neq 0$, there is no section of $\PSC(E)|_{B^{(n)}} \to B^{(n)}$.
In particular, $\pi_{i}(\PSC(X)) \neq 0$ holds for at least one $i \in \{0, \ldots, n-1\}$.
\end{theo}

\begin{proof}
We first consider the case that either $c_1(\fraks)^2 > 0$, or $c_1(\fraks)^2 = 0$ and $c_1(\fraks)$ is not torsion.
In these cases, $\PSC(X)$ does not intersects with the wall.
We can therefore construct an inductive section $\si : B^{(n)} \to \Pi(E)|_{B^{(n)}}$
 so that $\si$ factors through the inclusion $\PSC(E) \inc \Met(X) \inc \Pi(E)$.
Since (unperturbed) \SW equations has no solution with respect to a positive scalar curvature metrics, we have $\SWcoch(E, \si) = 0$, and hence $\SWcoh(E)=0$.
Here $\SWcoch(\cdot)$ is the cochain corresponding to $\Acoch(\cdot)$ in \cref{subsection Definition of the characteristic classes in the basic case} for the SW setting.
For other possibilities on $c_{1}(\fraks)$, one can easily modify the above argument.
(In particular, for the case that $c_{1}(\fraks)$ is torsion, we have to consider perturbed \SW equations.
For the proof of the case, we need the compactness assumption on $B^{(n)}$.)
See the proof of Proposition~2.4 in \cite{Konno3} for the modification.
\end{proof}

\section{Calculations}
\label{section Non triviality}

In this \lcnamecref{section Non triviality} we give some calculations of our characteristic classes.
In \cref{subsection Calculation of SWcoh for connected sum with n(S2S^2)} we develop a version of higher-dimensional wall-crossing and use it to calculations of $\SWcoh$.
To give more subtle examples, 
in \cref{subsection Cohomologically trivial bundles,subsection Combination with Ruberman's argument: spin case}, we combine an argument of D.~Ruberman's~\cite{MR1671187} and calculations given in \cite{Konno3} and in \cref{subsection Calculation of SWcoh for connected sum with n(S2S^2)}.
In \cref{subsection Other calculations} we mention a calculation of $\Dcoh$ obtained from D.~Ruberman's result~\cite{MR1671187}.
In \cref{subsection Calculation of SWcoh for connected sum with n(S2S^2),subsection Cohomologically trivial bundles,subsection Combination with Ruberman's argument: spin case}, we use the idea of combining the higher-dimensional mapping torus and \SW equations due to N.~Nakamura~\cite{MR2644908}.

\subsection{Almost localized harmonic forms}
\label{subsection Almost localized harmonic forms}

The purpose of this \lcnamecref{subsection Almost localized harmonic forms} is to describe a tool used in \cref{subsection Calculation of SWcoh for connected sum with n(S2S^2)}:
 harmonic forms on a manifold given as a connected sum of several manifolds which is almost localized on one of connected sum components in $L^{2}$-sense.
This is obtained by mimicking the argument of the additivity of Fredholm index of linear elliptic operators.
We mainly refer to Donaldson~\cite{MR1883043} in this \lcnamecref{subsection Almost localized harmonic forms}.

For $i=1,2$, let $(X_{i} ,g_{i})$ be oriented Riemannian manifolds with $b_{1}(X_{i}) = 0$ and $Y$ be an oriented closed Riemannian $3$-manifold.
Assume that $X_{i}$ has one end which is isometric to the cylinder $Y \times [0, \infty)$ with the product metric, and that the boundary of the complement of the end in $X_{1}$ coincides with $Y$ with the given orientation and that in $X_{2}$ do $-Y$: the opposite orientation.
For $S>0$, let $X_{i}(S)$ be the compact $4$-manifold obtained by cutting the end $Y \times (S, \infty)$ from $X_{i}$.
For $T > 0$, by identifying $Y \times \{t\}$ in $X_{1}$ and $-Y \times \{2T -t\}$ in $X_{2}$, we obtain a closed $4$-manifold $X_{T}$, often denoted by $X$.
The closed $4$-manifold $X_{T}$ contains a cylindrical part which is isometric to $Y \times [0, 2T]$.
Let $D_{i}$ and $D=D_{T}$ be the elliptic operators given by either
\begin{align*}
&D_{i} := d + d^{\ast} : \Omega^{\even}(X_{i})\to \Omega^{\odd}(X_{i}),\\
&D = D_{T} := d + d^{\ast} : \Omega^{\even}(X_{T}) \to \Omega^{\odd}(X_{T})
\end{align*}
or
\begin{align*}
&D_{i} := d + (d^{+})^{\ast} : \Omega^{0}(X_{i}) \oplus \Omega^{+}(X_{i}) \to \Omega^{1}(X_{i}),\\
&D = D_{T} := d + (d^{+})^{\ast} : \Omega^{0}(X_{T}) \oplus \Omega^{+}(X_{T}) \to \Omega^{1}(X_{T}).
\end{align*}
Recall that the $L^{2}$-orthogonal decomposition $\Omega^{+}(X) = \calH^{+} \oplus \im{d^{+}}$, where $\calH^{+}$ is the space of self-dual harmonic $2$-forms.
Because of this decomposition,  if $\mu \in \Omega^{+}(X)$ satisfies $(d^{+})^{\ast}\mu = 0$, then $\mu \in \calH^{+}$ holds.
In fact, for our purpose of \cref{subsection Calculation of SWcoh for connected sum with n(S2S^2)}, it is sufficient to consider only the latter operators, which we call the {\it Atiyah--Hitchin--Singer operators}, rather than the de Rham operators.
However the whole arguments in this \lcnamecref{subsection Almost localized harmonic forms} are parallel to both operators, and so we describe this almost localization phenomena for both operators.
On the cylidrical part, these operators admit the decomposition
\begin{align}
d+d^{\ast} = \frac{\del}{\del t} + L_{Y},
\label{decomposition of d plus dast}
\end{align}
where $t$ is the coordinate of $[0, \infty)$ (or some finite length interval) and $L_{Y}$ is a linear elliptic formally self-adjoint operator on $Y$.
(For example, see Section~9.2 of Melrose~\cite{MR1348401} for the de Rham operators and Subsection~3.1 of Donaldson~\cite{MR1883043} for the Atiyah--Hitchin--Singer operators.)
In fact, $L_{Y}$ has non-trivial kernel, and we therefore use weighted Sobolev spaces.
Let $\delta > 0$ be a sufficiently small positive number such that $[-\delta, \delta] \setminus \{0\}$ contains no spectrum of both $D_{i}$.
Fix a smooth positive function $W_{i}$ on $X_{i}$ whose restriction on the cylindrical end coincides with $e^{\delta t}$.
We take $W_{i}$ so that $W_{i} \geq 1$ on $X_{i}(1)$, and define $\|f\|_{L^{2}_{k, \delta}} := \|W_{i}f\|_{L^{2}_{k}}$.
We simply write $L^{2}_{k, \delta}(X_{i})$ for our functional space if the bundle which we consider is clear.
We now get the extended continuous operator
\[
D_{i} : L^{2}_{1, \delta}(X_{i}) \to L^{2}_{\delta}(X_{i}),
\]
which is Fredholm.
Henceforth, the notations $\Ker{D_{i}}$ and $\Coker{D_{i}}$ denote the kernel and the cokernel of this extended operator respectively.
The operator $D_{i}$ is positively weighted, and hence the adjoint $D_{i}^{\ast}$ is negatively weighted.
Therefore $\Coker{D_{i}}$ is isomorphic to the space of $L^{2}$-solutions to 
\[
D_{i}^{\ast} = d + d^{\ast} : \Omega^{\odd}(X_{i}) \to \Omega^{\even}(X_{i}).
\]
If we consider the de Rham operators, thanks to Proposition~4.9 in Atiyah--Patodi--Singer~\cite{MR0397797}, the space of $L^{2}$-solutions to $D_{i}^{\ast}$ is isomorphic to $\hat{H}^{1}(X_{i};\R) \oplus \hat{H}^{3}(X_{i};\R)$, where $\hat{H}^{\ast}(X_{i};\R)$ is the image of $H^{\ast}(Y;\R) \to H^{\ast}(X_{i}(0);\R)$.
This space vanishes since we assumed that $b_{1}(X_{i}) = 0$.
Similarly, if we consider the Atiyah--Hitchin--Singer operators,
we can use the adjoint version of Proposition~3.15 in Donaldson~\cite{MR1883043}.
It implies that $\Ker{D_{i}^{\ast}} \cong H^{1}(X_{i};\R) = 0$.
Thus we have $\Coker{D_{i}} = 0$ in both cases, and hence there exists a bounded right inverse of $D_{i}$.
We here use a concrete right inverse $Q^{X_{i}} : L^{2}_{\delta}(X_{i}) \to L^{2}_{1, \delta}(X_{i})$: the projection to the $L^{2}_{1, \delta}$-orthogonal complement of $\Ker{D_{i}}$.
Take a positive number $T_{1}>0$ with $T \geq T_{1}$.
Fix cut-off functions $\rho^{i}_{1,T_{1}} : X_{i} \to [0,1]$ satisfying:
\begin{itemize}
\item $\rho_{1,T_{1}}^{1}$ is supported on $X_{1}(3T_{1}/2)$ and $\rho_{1,T_{1}}^{1} = 1$ on $X_{1}(T_{1}/2)$,
\item $|\nabla \rho_{1,T_{1}}^{1}| \leq 2/T_{1}$, and
\item $\rho_{1,T_{1}}^{1} + \rho_{1,T_{1}}^{2} = 1$.
\end{itemize}
Using a way given in Subsection~3.3 in \cite{MR1883043}, we can construct a right inverse $Q^{X}_{1,T_{1}}$ of $D$ as follows.
We first define $P_{1, T_{1}} :  L^{2}(X) \to L^{2}_{1, \delta}(X)$ by
\[
P_{1, T_{1}}(\eta) := \rho_{1,T_{1}}^{1}Q^{X_{1}}(\eta_{1}) + \rho_{1,T_{1}}^{2}Q^{X_{2}}(\eta_{2}),
\]
where $\eta_{1}$ is the restriction of $\eta$ to $X_{1}(2T_{1}) \subset X$ and $\eta_{2}$ is that of $\eta$ to $X_{2}(2T) \subset X$, extended by zero to the rest of $X$.
Then we obtain an estimate $\|DP_{1, T_{1}} - \id\|_{\op} < C_{1}/T_{1}$, where $\|\cdot\|_{\op}$ denotes the operator norm and $C_{1}$ is a universal constant.
 By taking $T_{1}$ as $T_{1} > C_{1}$, we can define $(DP_{1, T_{1}})^{-1}$ using the Neumann series, and so can define $Q^{X}_{1,T_{1}} := P(DP_{1, T_{1}})^{-1}$.
By construction, $\|Q^{X}_{1,T_{1}}\|_{\op}$ continuously depends on $g_{1}$ and $g_{2}$ since $\|Q^{X_{i}}\|_{\op}$ continuously do on $g_{i}$.
Henceforth in this \lcnamecref{subsection Almost localized harmonic forms} we fix $T \geq  T_{1} > C_{1}$.
Let us define a map
\[
\Theta_{1, T_{1}} : \Ker{D_{1}}\oplus\Ker{D_{2}} \to \Ker{D}
\]
by
\[
(\mu_{1}, \mu_{2}) \mapsto \theta_{1, T_{1}}(\mu_{1}, \mu_{2}) - Q^{X}_{1, T_{1}}D\theta_{1, T_{1}}(\mu_{1}, \mu_{2}),
\]
where $\theta_{1, T_{1}}(\mu_{1}, \mu_{2}) := \rho_{1, T_{1}}^{1}\mu_{1} + \rho_{1, T_{1}}^{2}\mu_{2}$.
Similarly, for $T \geq T_{2} > C_{1}$, by taking $\rho_{2, T_{2}}^{1}$ and $\rho_{2, T_{2}}^{2}$, we can define
$Q^{X}_{2,T_{2}}$, $\theta_{2,T_{2}}$, and
\[
\Theta_{2, T_{2}} : \Ker{D_{1}}\oplus\Ker{D_{2}} \to \Ker{D}.
\]

\begin{defi}
Let $T \geq T_{1}, T_{2} > C_{1}$.
For $\mu \in \Ker{D_{1}}$, we call $\Theta_{1, T_{1}}(\mu,0) \in \Ker{D}_{T}$ the {\it almost localized harmonic form on $X_{1}$ arising from $\mu$} with respect to $T$ and $ T_{1}$ (and $g_{1}, g_{2}$).
We briefly write $\check{\mu}$ for $\Theta_{1, T_{1}}(\mu,0)$ if there is no confusion.
Similarly we define almost localized harmonic forms on $X_{2}$ using $\Theta_{2, T_{2}}$.
In particular, when we consider the Atiyah--Hitchin--Singer operators and $\mu \in \Ker{D}_{1} \cap \Omega^{+}(X_{i})$, we also call $\check{\mu}$ the {\it almost localized self-dual harmonic form arising from $\mu$} to emphasize self duality.
\end{defi}

\begin{lem}
\label{estimate for QDvp}
For $(\mu_{1}, \mu_{2}) \in \Ker{D_{1}}\oplus\Ker{D_{2}}$,
\begin{align*}
\|Q^{X}_{i,T_{i}}D\theta_{i,T_{i}}(\mu_{1}, \mu_{2})\|_{L^{2}_{\delta}}
\leq \frac{C_{2}}{T_{i}}(\|\mu_{1}\|_{L^{2}_{\delta}} + \|\mu_{2}\|_{L^{2}_{\delta}})
\end{align*}
holds, where $C_{2} = C_{2}(g_{1}, g_{2})$ is a constant continuously depending on $g_{1}$ and $g_{2}$.
\end{lem}

\begin{proof}
See Subsection~3.3 in \cite{MR1883043}.
The continuous dependence of $C_{2}(g_{1}, g_{2})$ follows from that of $\|Q^{X}_{i,T_{i}}\|_{\op}$.
\end{proof}

The interaction between almost localized harmonic forms on $X_{1}$ and $X_{2}$ is as follows:

\begin{lem}
\label{lem for interaction between almost localized forms}
Assume that $4T \geq 3(T_{1}+T_{2})$ and that $T_{1}, T_{2} \geq 1$.
Let $\mu_i \in \Ker{D_{i}}$ for $i=1,2$.
Then, for $\check{\mu}_{1} := \Theta_{1, T_{1}}(\mu_{1},0)$ and $\check{\mu}_{2} := \Theta_{2, T_{2}}(0,\mu_{2})$, we have
\begin{align}
|(\check{\mu}_{1}, \check{\mu}_{2})_{L^{2}}| \leq \frac{C_{3}}{\min\{T_{1}, T_{2}\}} (\|\mu_{1}\|_{L^{2}_{\delta}} + \|\mu_{2}\|_{L^{2}_{\delta}})^{2},
\label{estimate for interaction between almost localized forms}
\end{align}
where $C_{3} = C_{3}(g_{1}, g_{2})$ is a constant continuously depending on $g_{1}$ and $g_{2}$.
\end{lem}

\begin{proof}
Note that $\|f\|_{L^{2}} \leq \|f\|_{L^{2}_{\delta}}$ holds for a general $f$, and also note that 
\[
(\rho_{1,T_{1}}^{1}\mu_{1}, \rho_{2,T_{2}}^{2}\mu_{2})_{L^{2}} = 0
\]
follows from $4T \geq 3(T_{1}+T_{2})$.
We therefore have
\begin{align*}
|(\check{\mu}_{1}, \check{\mu}_{2})_{L^{2}}|
\leq& \frac{C_{2}}{T_{1}}(\|\mu_{1}\|_{L^{2}_{\delta}} + \|\mu_{2}\|_{L^{2}_{\delta}})\|\mu_{2}\|_{L^{2}_{\delta}}\\
&+ \frac{C_{2}}{T_{2}}\|\mu_{1}\|_{L^{2}_{\delta}}(\|\mu_{1}\|_{L^{2}_{\delta}}+ \|\mu_{2}\|_{L^{2}_{\delta}}) + \frac{C_{2}^{2}}{T_{1}T_{2}}(\|\mu_{1}\|_{L^{2}_{\delta}} + \|\mu_{2}\|_{L^{2}_{\delta}})^{2}
\end{align*}
using \cref{estimate for QDvp}.
Thus we obtain the inequality \eqref{estimate for interaction between almost localized forms} for suitable $C_{3}$.
\end{proof}

The interaction between almost localized harmonic forms on a common connected sum component is as follows:

\begin{lem}
\label{lem for interaction between almost localized forms on the same}
Fix an $L^{2}$-orthonormal basis $\{\eta_{j}\}_{j}$ of $\Ker{D_{1}}$.
Then, there exists a family of positive numbers $\epsilon_{1}(T_{1})>0$ with $\epsilon_{1}(T_{1}) \to 0$ as $T_{1} \to \infty$ satisfying the following conditions:
\begin{itemize}
\item  $\epsilon_{1}(T_{1}) = \epsilon_{1}(T_{1};g_{1}, g_{2},\{\eta_{j}\}_{j})$ continuously depends on $T_{1}$, $g_{1}$, $g_{2}$, and $\{\eta_{j}\}_{j}$.
\item Let $\mu, \mu' \in \Ker{D_{1}}$.
Then, for $\check{\mu} := \Theta_{1, T_{1}}(\mu,0)$ and $\check{\mu}' := \Theta_{1, T_{1}}(\mu',0)$, we have
\begin{align*}
|(\check{\mu}, \check{\mu}')_{L^{2}} - (\mu, \mu')_{L^{2}}| \leq \epsilon_{1}(T_{1}) \|\mu\|_{L^{2}_{\delta}} \|\mu'\|_{L^{2}_{\delta}}.
\end{align*}
\end{itemize}
\end{lem}

\begin{proof}
For any $\mu \in \Ker{D_{1}}$, we have
\begin{align}
\|\mu-\theta_{1, T_{1}}(\mu,0)\|_{L^{2}}^{2}
\leq \int_{Y \times [T_{1}/2, \infty)}|\mu|^{2}
\leq \left(2\sum_{j} \int_{Y \times [T_{1}/2, \infty)} |\eta_{j}|^{2}\right) \|\mu\|_{L^{2}}^{2}.
\label{ineq between omega vpomega}
\end{align}
Note that an element of $\Ker{D_{i}}$ is smooth, and so is $\eta_{j}$.
Hence
\[
\epsilon_{1}(T_{1}) := 2\max\left\{2\sum_{j} \int_{Y \times [T_{1}/2, \infty)} |\eta_{j}|^{2},\left(\frac{C_{2}}{T_{1}}\right)^{2}\right\}
\]
continuously depends on $T_{1}$, $g_{1}$, $g_{2}$, and $\{\eta_{j}\}_{j}$.
In addition, for $\mu, \mu' \in \Ker{D_{1}}$, we obtain the estimate
\begin{align*}
&|(\check{\mu}, \check{\mu}')_{L^{2}} - (\mu, \mu')_{L^{2}}|\\
\leq& |(\check{\mu} - \theta_{1, T_{1}}(\mu,0), \check{\mu}'-\theta_{1, T_{1}}(\mu',0))_{L^{2}}|
+ |(\theta_{1, T_{1}}(\mu,0)-\mu, \theta_{1, T_{1}}(\mu',0)-\mu')_{L^{2}}|\\
=& |(Q^{X}_{1, T_{1}}D\theta_{1, T_{1}}(\mu,0), Q^{X}_{1, T_{1}}D\theta_{1, T_{1}}(\mu',0))_{L^{2}}|\\
&+ |(\theta_{1, T_{1}}(\mu,0)-\mu, \theta_{1, T_{1}}(\mu',0)-\mu')_{L^{2}}|\\
\leq& \epsilon_{1}(T_{1}) \|\mu\|_{L^{2}_{\delta}} \|\mu'\|_{L^{2}_{\delta}}
\end{align*}
from \cref{estimate for QDvp} and the inequality \eqref{ineq between omega vpomega}.
\end{proof}

We remark a family version of the whole construction.
Let $B$ be a compact topological space and $\{g_{i}^{b}\}_{b \in B}$ be a continuous family of Riemannian metrics on $X_{i}$ whose restriction on the end of $X_{i}$ is constant with respect to $b \in B$ and is isometric to $Y \times [0, \infty)$ like $g_{i}$ above.
Then, for each $T>0$, we get a family of Riemannian metrics $\{g_{T}^{b}\}$ on $X$ by gluing $(X_{1}, g_{1}^{b})$ and $(X_{2}, g_{2}^{b})$.
Since metrics on $Y$ is constant, $L_{Y}$ is also a constant operator.
Thus we can take a weight $\delta > 0$ uniformly with respect to $b \in B$.
For this $\delta$ and $T_{1}, T_{2}$ with $T \geq T_{1}, T_{2}>C_{1}$, we can define $\Theta^{b}_{i,T_{i}} :  \Ker{D_{1}^{b}}\oplus\Ker{D_{2}^{b}} \to \Ker{D^{b}_{T}}$ as $\Theta_{i,T_{i}}$ above, where $D_{i}^{b} = d+d^{\ast_{g_{i}^{b}}}$ and $D^{b}_{T} = d+d^{\ast_{g_{T}^{b}}}$.
Of course, the $4$-dimensional operators $D_{i}^{b}$ depends on $b$.
However $\bigsqcup_{b \in B}\Ker{D_{i}^{b}} \to B$ forms a vector bundle as follows.
We first consider the de Rham operators.
In this case $\Ker{D_{i}^{b}}$ is isomorphic to the space of extended $L^{2}$-solutions to $D_{i}^{b}$ in the sense of Atiyah--Patodi--Singer~\cite{MR0397797} since $D_{i}^{b}$ is equipped with a positive weight.
This space is isomorphic to $H^{0}(X_{i};\R) \oplus H^{2}(X_{i};\R) \oplus H^{4}(X_{i};\R)$.
Similarly, if we consider the Atiyah--Hitchin--Singer operators, we have $\Ker{D_{i}^{b}} \cong H^{+}(X_{i};\R)$, which is the adjoint version of Proposition~3.15 in Donaldson~\cite{MR1883043}.
Hence $\dim\Ker{D_{i}^{b}}$ is constant with respect to $b \in B$ and $\bigsqcup_{b \in B}\Ker{D_{i}^{b}} \to B$ is a vector bundle for both cases.
Suppose that $\bigsqcup_{b \in B}\Ker{D_{i}^{b}} \to B$ is trivial (for example $B$ is contractible), 
and take a global section $\{\eta_{j}^{b}\}_{j, b}$ of the frame bundle of $\bigsqcup_{b \in B}\Ker{D_{1}^{b}} \to B$ respecting the metric for each fiber.
By replacing each of $C_{2}$, $C_{3}$, and $\epsilon_{1}(T_{1})$ in the above argument with the maximum of it with respect to $b \in B$, we have the following lemmas corresponding to \cref{lem for interaction between almost localized forms,lem for interaction between almost localized forms on the same}:

\begin{lem}
\label{lem family: for interaction between almost localized forms}
Assume that $4T \geq 3(T_{1}+T_{2})$ and that $T_{1}, T_{2} \geq 1$.
Then there exists a constant $C_{3} > 0$ such that the following holds.
Let $\{\mu_i^{b}\}_{b \in B}$ be a section of $\bigsqcup_{b \in B}\Ker{D_{i}^{b}} \to B$ for $i=1,2$.
Then, for $\check{\mu}_{1}^{b} := \Theta_{1,T_{1}}^{b}(\mu_{1}^{b},0)$ and $\check{\mu}_{2}^{b} := \Theta_{2,T_{2}}^{b}(0,\mu_{2}^{b})$, we have
\begin{align*}
|(\check{\mu}_{1}^{b}, \check{\mu}_{2}^{b})_{L^{2}}| \leq \frac{C_{3}}{\min\{T_{1}, T_{2}\}} (\|\mu_{1}^{b}\|_{L^{2}_{\delta}} + \|\mu_{2}^{b}\|_{L^{2}_{\delta}})^{2}.
\end{align*}
\end{lem}

\begin{lem}
\label{lem family: for interaction between almost localized forms on the same}
There exists a family of positive numbers $\epsilon_{1}(T_{1})>0$ with $\epsilon_{1}(T_{1}) \to 0$ as $T_{1} \to \infty$  such that the following holds.
Let $\{\mu^{b}\}_{b \in B}, \{\mu'^{b}\}_{b \in B}$ be sections of $\bigsqcup_{b \in B}\Ker{D_{1}^{b}} \to B$.
Then, for $\check{\mu}^{b} := \Theta_{1,T_{1}}^{b}(\mu^{b},0)$ and $\check{\mu}'^{b} := \Theta_{1,T_{1}}^{b}(\mu'^{b},0)$, we have
\begin{align*}
|(\check{\mu}^{b}, \check{\mu}'^{b})_{L^{2}} - (\mu^{b}, \mu'^{b})_{L^{2}}| \leq \epsilon_{1}(T_{1}) \|\mu^{b}\|_{L^{2}_{\delta}} \|\mu'^{b}\|_{L^{2}_{\delta}}.
\end{align*}
\end{lem}

\begin{rem}
In the argument for families, even if metrics on $Y$ varies, in fact we can take uniform $\delta$.
This is because $B$ is compact and $\dim\Ker{L_{Y}}$ is constant.
(See the above argument for $\dim\Ker{D_{i}}$.)
\end{rem}

\begin{rem}
We can work with a negative weight for $D_{i}$.
For example, let us consider the de Rham operators.
In this case $\Coker{D_{i}}$ is isomorphic to the extended $L^{2}$-solutions to this operator, and is also isomorphic to $H^{1}(X_{i};\R) \oplus H^{3}(X_{i};\R)$.
Therefore, if we assume $b_{1}(X_{i}) = 0$, the same argument works.
\end{rem}

\subsection{$\SWcoh$ for connected sum with $n(S^2\times S^2)$}
\label{subsection Calculation of SWcoh for connected sum with n(S2S^2)}

We fix $n>0$ in this \lcnamecref{subsection Calculation of SWcoh for connected sum with n(S2S^2)}.
We shall calculate $\SWcoh$ for a bundle whose fiber is a $4$-manifold given as the connected sum with a given $4$-manifold and $N := nS^{2} \times S^{2}$.
Set $\Si_{1} = S^{2} \times \{\pt\}$ and $\Si_{2} = \{\pt\} \times S^{2}$.
For each $i \in \{1, \ldots, n\}$, let $f_{i} \in \Diff^{+}(S^{2} \times S^{2})$ be an orientation preserving diffeomorphism admitting a fixed $4$-disk, and assume that the induced map $(f_{i})_{\ast}$ on $H_{2}(S^{2} \times S^{2})$ is given by either 
\[
\begin{pmatrix}
 -1 & 0\\
 0 & -1
\end{pmatrix}\quad \text{or}\quad
\begin{pmatrix}
 0 & -1\\
 -1 & 0
\end{pmatrix},
\]
where we take $\{[\Si_{1}], [\Si_{2}]\}$ as the basis of $H_{2}(S^{2} \times S^{2})$.
We write each connected sum component of $N = \#_{i=1}^{n} S^{2} \times S^{2}$ as $N_{i}$, namely $N = \#_{i=n}^{n} N_{i}$.
Regard $f_{i}$ as a diffeomorphism on $N_{i}$, and extend $f_{i}$ along the fixed disk by the identity to the whole of $N$ and to $X = M \# N$ for a given oriented closed $4$-manifold $M$.
These extensions are also denoted by $f_{i}$.
Obviously $f_{1}, \ldots, f_{n} \in \Diff^{+}(X)$ mutually commutes, and thus we get a homomorphism $\pi_{1}(T^{n}) \cong  \Z^{n} \to \Diff^{+}(X)$.
Using this action of $\pi_{1}(T^{n})$ on $X$, we define a family on the torus $E \to T^{n}$ by $E := \R^{n}\times_{\pi_{1}(T^{n})}X$.
In other words, $E$ is the mapping torus of $f_{1}, \ldots, f_{n}$.
We can now write down the following non-vanishing theorem:

\begin{theo}
\label{theo on calculation of SW for connected sum with spin}
Let $M$ be an oriented closed smooth $4$-manifold with $b^+(M) \geq 2$ and $\fraks_0$ be the isomorphism class of a spin$^{c}$ structure on $M$ such that the \SW invariant of $(M, \fraks_0)$ is an odd number.
Let $(X, \fraks)$ be the pair defined by
\[
(X, \fraks) := (M \# n(S^{2} \times S^{2}), \fraks_0\#\frakt),
\]
where $\frakt$ is the isomorphism class of a spin structure on $n(S^{2} \times S^{2})$.
Then, for the bundle $X \to E \to T^{n}$ constructed above, $\SWcoh(E) \in H^{n}(T^{n};\Z/2)$ is the generator of $H^{n}(T^{n};\Z/2)$.
In particular, we have
\[
\SWcoh(X,\fraks) \neq 0 \text{ in } H^{n}(B\Diff(X,\fraks);\Z/2)
\]
by functoriality (\cref{part of theo of well-definedness of the invariant functoriality}).
\end{theo}

Note that the condition $b^{+}(X) \geq n+2$ is obviously satisfied and it is easy to see that $d(\fraks)=-n$ in the situation of \cref{theo on calculation of SW for connected sum with spin}.
We also note that, if $X$ is simply connected and $\fraks$ is the isomorphism class of a $\spc$ structure coming from a spin structure on $X$, then we have $\Diff(X, \fraks) = \Diff^{+}(X)$.
A spin structure on $X$ is unique up to isomorphism in this case.

\begin{ex}
\label{non vanishing for K3}
For $n>0$, let $\fraks$ be the isomorphism class of a spin structure on $K3 \# n(S^{2} \times S^{2})$.
Then we have
\[
\SWcoh(K3 \# n(S^{2} \times S^{2}),\fraks) \neq 0 \text{ in } H^{n}(B\Diff^{+}(K3 \# n(S^{2} \times S^{2}));\Z/2).
\]
\end{ex}

\begin{ex}
\label{ex final is connected sum of SS before}
We here consider the connected sum of some copies of $S^{2}\times S^{2}$ as follows.
(The author thanks David Baraglia for discussing this argument.)
Using techniques given in J.~Park~\cite{MR1900318}, B.~Hanke, D.~Kotschick, and J.~Wehrheim~\cite{MR2021570} have given oriented and closed $4$-manifolds $X(p,q)$ involving two integer parameters $p,q$ (denoted by $k, n$ in \cite{MR2021570} respectively) which are spin, symplectic, signature zero, with $b^{+} \geq 2$ and dissolve after connected summing one copy of $S^{2} \times S^{2}$ for sufficiently large $p,q$.
Fix such $p,q$, and then we have $X(p,q) \# S^{2} \times S^{2} \cong m(S^{2}\times S^{2})$ for some $m>0$.
In \cref{theo on calculation of SW for connected sum with spin}, let us substitute $X(p,q)$ for $M$, and as $\fraks_{0}$ take the (isomorphism class of a) $\spc$ structure coming from the symplectic structure of $X(p,q)$.
Then, for a fixed $n>0$, we get
\[
\SWcoh(k(S^{2}\times S^{2}),\fraks) \neq 0 \text{ in } H^{n}(B\Diff(k(S^{2}\times S^{2}), \fraks);\Z/2),
\]
where $k = m+n-1$.
In addition, thanks to \cite{MR2021570}, one can find infinitely many such $k$ by varying $p,q$.
\end{ex}

\begin{rem}
\label{rem Hanke Kotschick Wehrheim and Wall}
Since Hanke--Kotschick--Wehrheim~\cite{MR2021570} used Wall's theorem~\cite{MR0163324} to prove the above dissolving of $X(p,q) \# S^{2} \times S^{2}$, it seems difficult to determine $k$'s in \cref{ex final is connected sum of SS before}.
\end{rem}

The rest of this \lcnamecref{subsection Calculation of SWcoh for connected sum with n(S2S^2)} is devoted to proving \cref{theo on calculation of SW for connected sum with spin}.
One of the main tools of the proof is a technique due to D.~Ruberman~\cite{MR1671187, MR1734421, MR1874146}:
combination of wall-crossing and gluing.
To use it, we describe wall-crossing on $N$.
This is a spin-analogue of the argument of \cite{Konno1} by the author.
Recall that the definition of the wall for the \SW equations.
Let us fix a spin structure on $N$.
In the space of perturbations 
\[
\Pi(N) = \bigsqcup_{g \in \Met(N)} \Pi_{g}(N) = \bigsqcup_{g \in \Met(N)} L^{2}_{k-1}(\Lambda^{+}_{g}(N)),
\]
we have a codimenion-$n$ subspace
\[
\calW(N) := \bigsqcup_{g \in \Met(N)} \calW_{g}(N),
\]
where
\[
\calW_{g}(N) := F_{A_{0}}^{+_{g}} + \im{d^{+_{g}}}
\]
for a fixed reference connection $A_{0}$ of the determinant line bundle of the spin structure.
We can take $A_{0}$ as a trivial connection, and then $\calW_{g}(N) =  \im{d^{+_{g}}}$.
Set
\[
\circPi(N) := \Pi(N) \setminus \calW(N),\quad \circPi_{g}(N) := \Pi_{g}(N) \setminus \calW_{g}(N).
\]
Then $\circPi(N)$ is homotopy equivalent to $S^{n-1}$.
If a continuous map $\vp : [0,1]^{n} \to \Pi(N)$ satisfying that $\vp(\del[0,1]^{n}) \subset \circPi(N)$ is given, we can therefore define the ``intersection number'' $\vp \cdot \calW(N)$ as the mapping degree of $\vp : ([0,1]^{n}, \del[0,1]^{n})) \to (\Pi(N), \circPi(N))$.
(Since we work on $\Z/2$, we do not mention the convention on sign.)
Denote by $\calH^{+}_{g}(N)$ the space of self-dual harmonic $2$-forms with respect to a metric $g$ on $N$.
Let $g_{\bullet} : [0,1]^{n} \to \Met(N)$ be a family of metrics.
For each element $\bt \in [0,1]^{n}$, we write $g_{\bt}$ for $g_{\bullet}(\bt)$. 
Let $V \subset H^{2}(N;\R)$ be a maximal positive definite subspace with respect to the intersection form and $\gamma_{1}, \ldots, \gamma_{n}$ be a basis of $V$.
Let 
\[
\Rest : \bigsqcup_{\bt \in [0,1]^{n}} (\Pi_{g_{\bt}}(N), \circPi_{g_{\bt}}(N)) \to \bigsqcup_{\bt \in [0,1]^{n}} (\calH^{+}_{g_{\bt}}(N),\calH^{+}_{g_{\bt}}(N) \setminus \{0\})
\]
be the restriction map, and define
\[
\Phi : \bigsqcup_{\bt \in [0,1]^{n}} (\calH^{+}_{g_{\bt}}(N),\calH^{+}_{g_{\bt}}(N) \setminus \{0\}) \to (\R^{n}, \R^{n} \setminus \{0\})
\]
by
\[
\Phi(\eta) := (\left<[\eta] \cup \gamma_{i},[N]\right>)_{i=1}^{n}
\]
for $\eta \in \calH^{+}_{g_{\bt}}(N)$, and set $\Psi := \Phi \circ \Rest$.

\begin{lem}
\label{lem for homotopy equivalence of Phi}
The map $\Psi$ is a homotopy equivalent map between pairs.
\end{lem}

\begin{proof}
Since $[0,1]^{n}$ is contractible, it suffices to see that the restriction of $\Psi$ to a fiber is homotopy equivalent,
and it is clear since $(\Pi_{g_{\bt}}(N), \circPi_{g_{\bt}}(N))$ is homotopy equivalent to $(\calH^{+}_{g_{\bt}}(N),\calH^{+}_{g_{\bt}}(N) \setminus \{0\})$ and $\Phi|_{\calH^{+}_{g_{\bt}}(N)} : \calH^{+}_{g_{\bt}}(N) \to \R^{n}$  is a linear isomorphism.
\end{proof}

\begin{cor}
\label{cor reduction between finite dim mapping degree}
Let $\si$ be a section
\[
\si : ([0,1]^{n}, \del[0,1]^{n}) \to \bigsqcup_{\bt \in [0,1]^{n}} (\Pi_{g_{\bt}}(N), \circPi_{g_{\bt}}(N))
\]
and define $\vp := \iota \circ \si : ([0,1]^{n}, \del[0,1]^{n}) \to (\Pi(N), \circPi(N))$, where
\[
\iota : \bigsqcup_{\bt \in [0,1]^{n}} (\Pi_{g_{\bt}}(N), \circPi_{g_{\bt}}(N)) \to  (\Pi(N), \circPi(N))
\]
is the inclusion.
Then, $\vp \cdot \calW(N)$ is given as the mapping degree of the map
\[
\Psi \circ \si : ([0,1]^{n}, \del[0,1]^{n}) \to (\R^{n}, \R^{n} \setminus \{0\}).
\]
\end{cor}

We shall use \cref{cor reduction between finite dim mapping degree} to prove \cref{theo on calculation of SW for connected sum with spin}.
Before starting the proof of \cref{theo on calculation of SW for connected sum with spin}, we give a map used in the proof.
Let us write $[0,1]^{n} = I_{1} \times \cdots I_{n}$, where $I_{i} = [0,1]$ for each $i$.
Decompose it as $I_{i} = I_{i,1} \cup I_{i,2}$, where $I_{i,1} = [0, 1/2]$ and $I_{i,2} = [1/2,1]$.
Then we obtain the decomposition
\[
[0,1]^{n} = \bigcup_{1 \leq a_{1}, \ldots, a_{n} \leq 2} I_{1,a_{1}} \times\cdots\times I_{n, a_{n}}.
\]
For each $i$,
let $\Gamma_{i,1} : I_{i,1} \to I_{i,1}$ be the identity map and $\Gamma_{i,2} : I_{i,2} \to I_{i,1}$ be the projection onto $\{1/2\}$.
Using these maps, we define a map
\[
\Shrink : [0,1]^{n} \to I_{1,1} \times \cdots \times I_{n,1} = [0,1/2]^{n}
\]
by
\[
\Shrink := \bigcup_{1 \leq a_{1}, \ldots, a_{n} \leq 2} \Gamma_{1, a_{1}} \times \cdots \times \Gamma_{n, a_{n}}.
\]
We next consider the following reparametrization map.
Let us consider the homeomorphism $[0, 1/2] \cong [0,1]$ given by an affine function.
This map induces a homeomorphism $\Repa : [0,1/2]^{n} \cong [0,1]^{n}$.
Set the composition of these maps as
\begin{align}
F_{RS} := \Repa \circ \Shrink : [0,1]^{n} \to [0,1]^{n}.
\label{composition of Repa and Cent}
\end{align}
Intuitively, the pull-back by $F_{RS}$ is the following operation:
for a given object on $[0,1]^{n}$, first give a copy of the object on $[0,1/2]^{n}$, and second extend it on the whole of $[0,1]^{n}$ by a natural pull-back.

\begin{proof}[Proof of \cref{theo on calculation of SW for connected sum with spin}]
We first define $\gamma_{1}, \ldots, \gamma_{n}$ as follows.
Let $\Si_{1,i}$, $\Si_{2,i} \subset N_{i}$ be the copies of $\Si_{1}$ and $\Si_{2}$ respectively.
We define $\gamma_{i} := \PD[\Si_{1,i}] + \PD[\Si_{2,i}]$ for each $i \in \{1, \ldots, n\}$.
We next take $g_{\bullet}$ as follows.
Let $g^{0}_{i}$ be a metric on $N_{i}$ which is a cylindrical near the boundary of an embedded disk in $N_{i}$.
Take a path $h_{i} : [0,1] \to \Met(N_{i})$ from $g^{0}_{i}$ to $f_{i}^{\ast}g^{0}_{i}$ and define $g_{\bullet} : [0,1]^{n} \to \Met(N)$ by 
\[
g_{\bt} = h_{1}(t_{1}) \#\cdots\# h(t_{n})
\]
for $\bt = (t_{1}, \ldots, t_{n})$.
Define faces $F_{i}^{0}$ and $F_{i}^{1}$ of $[0,1]^{n}$ by
\begin{align*}
F_{i}^{0} := \Set{(t_{1}, \ldots, t_{n}) \in [0,1]^{n} | t_{i}  = 0}
\end{align*}
and 
\begin{align*}
F_{i}^{1} := \Set{(t_{1}, \ldots, t_{n}) \in [0,1]^{n} | t_{i} = 1 }.
\end{align*}
We define a map $\reflect{\bullet}{i} : F_{i}^{0} \to F_{i}^{1}$ by
\[
\bt = (t_{1}, \ldots,0, \ldots, t_{n}) \mapsto (t_{1}, \ldots,1, \ldots, t_{n}) = \reflect{\bt}{i},
\]
and similarly define $\reflect{\bullet}{i} : F_{i}^{1} \to F_{i}^{0}$, which satisfies $\reflect{\bullet}{i} \circ \reflect{\bullet}{i} = \id$.
Note that, for $\bt \in F_{i}^{0}$, we have $f_{i}^{\ast}g_{\bt} = g_{\reflect{\bt}{i}}$.
In what follows, we identify $H_{2}(N)$ with $H^{2}(N)$ by the Poincar\'{e} duality.

We next take a section $\si : ([0,1]^{n}, \del[0,1]^{n}) \to \bigsqcup_{\bt \in [0,1]^{n}} (\Pi_{g_{\bt}}(N), \circPi_{g_{\bt}}(N))$ as follows.
For each $i$, let $\hat{N}_{i}$ be the manifold obtained by gluing $N_{i} \setminus D^{4}$ with $S^{3} \times [0, \infty)$, where $D^{4} \subset N_{i}$ is the disk used to consider the connected sum and $S^{3}$ is equipped with the standard metric.
Note that a metric $h$ on $N_{i}$ induces $\hat{h} \in \Met(\hat{N}_{i})$ if $h$ is a form of the product metric near $D^{4}$.
For $t_{i} \in [0,1]$, let us consider
\[
D_{i}^{\hat{h}_{i}(t_i)} = d + (d^{+_{\hat{h}_{i}(t_i)}})^{\ast_{\hat{h}_{i}(t_i)}} : \Omega^{0}(\hat{N}_{i}) \oplus \Omega^{+_{\hat{h}_{i}(t_{i})}}(\hat{N}_{i}) \to \Omega^{1}(\hat{N}_{i}),
\]
which decomposes as $D_{i}^{\hat{h}_{i}(t_i)} = \del/\del t + L_{S^{3}}$ as in \eqref{decomposition of d plus dast} on the cylindrical part.
Let $\delta>0$ be a sufficiently small number such that $[-\delta, \delta] \setminus \{0\}$ contains no spectrum of $L_{S^{3}}$.
Extend $D_{i}^{\hat{h}_{i}(t_i)}$ to a bounded Fredholm operator $D_{i}^{\hat{h}_{i}(t_i)} : L^{2}_{1, \delta}(\hat{N}_{i}) \to L^{2}_{\delta}(\hat{N}_{i})$.
As in \cref{subsection Almost localized harmonic forms}, the notation $\Ker{D_{i}^{\hat{h}_{i}(t_i)}}$ denotes the kernel of this extended operator.
As we mentioned in \cref{subsection Almost localized harmonic forms}, $\Ker{D_{i}^{\hat{h}_{i}(t_i)}}$ is isomorpshic to the space of extended $L^{2}$-solutions to $D_{i}^{\hat{h}_{i}(t_i)}$, which is nothing other than constant functions and extended self-dual $L^{2}$-harmonic forms on $\hat{N}_{i}$.
Thus we have an embedding $\calH^{+}_{L^{2},\hat{h}_{i}(t_i)}(\hat{N}_{i}) \inc \Ker{D_{i}^{\hat{h}_{i}(t_i)}}$, where $\calH^{+}_{L^{2},\hat{h}_{i}(t_i)}(\hat{N}_{i})$ is the space of $\hat{h}_{i}(t_i)$-self-dual $L^{2}$-harmonic $2$-forms on $\hat{N}_{i}$, which is of dimension $b^{+}(N_{i})=1$.
Take a section $\{\mu_{i}^{t_{i}}\}_{t_{i} \in [0,1]}$ of the bundle $\bigsqcup_{t_{i} \in [0,1]}\calH^{+}_{L^{2},\hat{h}_{i}(t_i)}(\hat{N}_{i}) \to [0,1]$ with $\|\mu_{i}^{t_i}\|_{L^{2}_{\delta}} = 1$.
There is no difficulty in extending the whole argument of \cref{subsection Almost localized harmonic forms} for connected sums of finitely many manifolds, not only for connected sums of two manifolds, and so we use it.
Let $T \geq T_{1}, \ldots, T_{n}> 0$ be positive numbers satisfying that $4T \geq 3(T_{i} + T_{j})$ for any $i,j \in \{1, \ldots, n\}$ with $i \neq j$ and that $T_{i} > C_{1}$ for any $i$, where $C_{1}$ is the constant given in \cref{subsection Almost localized harmonic forms}.
We take $T \geq T_{1}, \ldots, T_{n}$ to be sufficiently large, which shall be described below.
In the definition for the connected sum $N = \#_{i=1}^{n} N_{i}$, we use the cylinders of length $2T$ as in \cref{subsection Almost localized harmonic forms}.
For $\bt \in [0,1]^{n}$, let us consider
\[
D = D^{g_{\bt}} = d + (d^{+_{g_{\bt}}})^{\ast_{g_{\bt}}} : \Omega^{0}(N) \oplus \Omega^{+_{g_{\bt}}}(N) \to \Omega^{1}(N).
\]
Let $\check{\mu}_{i}^{\bt} \in \Ker{D^{g_{\bt}}}$ be the almost localized self-dual harmonic form on $N_{i}$ arising from $\mu_{i}^{t_i}$ with respect to $T, T_{1}, \ldots, T_{n}$ and $\hat{h}_{1}(t_{1}), \ldots, \hat{h}_{n}(t_{n})$.
Define a section
\[
\alpha_{i} : [0,1]^{n} \to \bigsqcup_{\bt \in [0,1]^{n}}\Pi_{g_{\bt}}(N)
\]
by $\alpha_{i}(\bt) := \check{\mu}_{i}^{\bt}$ for each $i$.
Using the map \eqref{composition of Repa and Cent}, let us consider
\[
\tilg_{\bullet} := F_{RS}^{\ast}g_{\bullet} : [0,1]^{n} \to \Met(N).
\]
We again write $\tilg_{\bt}$ for $\tilg_{\bullet}(\bt)$.
Similarly, let us consider the pull-backed section 
\[
\tilalpha_{i} := F_{RS}^{\ast}\alpha_{i} : [0,1]^{n} \to \bigsqcup_{\bt \in [0,1]^{n}}\Pi_{\tilg_{\bt}}(N).
\]
By definition we have $\tilg_{\bt} = g_{\bs}$ for $\bs = F_{RS}(\bt)$ and $\tilalpha_{i}(\bt) = \check{\mu}_{i}^{\bs}$.
We now define a continuous section
\[
\beta_{i} : [0,1]^{n} \to  \bigsqcup_{\bt \in [0,1]^{n}} \calH^{+}_{\tilg_{\bt}}(N)
\]
as follows.
For $\bt = (t_{1}, \ldots, t_{n})$, if $0 \leq t_{i} \leq 1/2$ holds, we define $\beta_{i}(\bt) := \tilalpha_{i}(\bt)$, and if $1/2 \leq t_{i} \leq 1$ holds, we define
\[
\beta_{i}(\bt) := (1-t_{i})\tilalpha_{i}(\bt) + t_{i} f_{i}^{\ast}\check{\mu}_{i}^{\reflect{\bs}{i}}
\]
for $\bs = F_{RS}(\bt)$.
We here check that, if $\bt$ satisfies that $1/2 \leq t_{i} \leq 1$, then $f_{i}^{\ast}\check{\mu}_{i}^{\reflect{\bs}{i}} \in \calH^{+}_{\tilg_{\bt}}(N)$ holds.
If $1/2 \leq t_{i} \leq 1$ holds, we have $\reflect{\bs}{i} \in F_{i}^{0}$, and therefore it follows that $f_{i}^{\ast}g_{\reflect{\bs}{i}} = g_{\bs} = \tilg_{\bt}$.
This gives the map $f_{i}^{\ast} : \calH^{+}_{g_{\reflect{\bs}{i}}}(N) \to \calH^{+}_{\tilg_{\bt}}(N)$, thus we get $f_{i}^{\ast}\check{\mu}_{i}^{\reflect{\bs}{i}} \in \calH^{+}_{\tilg_{\bt}}(N)$.
We define a section $\si : ([0,1]^{n}, \del[0,1]^{n}) \to \bigsqcup_{\bt \in [0,1]^{n}} (\Pi_{g_{\bt}}(N), \circPi_{g_{\bt}}(N))$ by
\[
\si(\bt) := \beta_{1}(\bt) + \cdots + \beta_{n}(\bt).
\]
Obviously $\si$ factors through the inclusion 
\[
\bigsqcup_{\bt \in [0,1]^{n}} (\calH^{+}_{g_{\bt}}(N),\calH^{+}_{g_{\bt}}(N) \setminus \{0\})
\inc
\bigsqcup_{\bt \in [0,1]^{n}} (\Pi_{g_{\bt}}(N), \circPi_{g_{\bt}}(N)).
\]

For this $\si$, we shall show that the mapping degree of the map
\[
\Psi \circ \si : ([0,1]^{n}, \del[0,1]^{n}) \to (\R^{n}, \R^{n} \setminus \{0\}).
\]
is $1$, up to sign.
Here $\Psi$ is the map given before \cref{lem for homotopy equivalence of Phi}.
If we assume this calculation of the mapping degree, the rest story is similar to Subsection~3.2 of \cite{Konno3}, given as follows.
Let us consider a cell structure of $[0,1]^{n}$ admitting a unique $n$-cell.
We equip $T^{n}$ with the cell structure obtained from the given cell structure of $[0,1]^{n}$ by identifying $F_{i}^{0}$ with $F_{i}^{1}$ for each $i$.
We now use D.~Ruberman's combination of wall-crossing and gluing arguments~\cite{MR1671187, MR1734421, MR1874146}.
(It is summarized as Proposition 4.1 in \cite{Konno2}.)
Take a generic element $\mu_{M} \in \Pi(M)$ and define $\tilde{\si} := \mu_{M} \# \si : [0,1]^{n} \to \bigsqcup_{\bt \in [0,1]^{n}} \Pi_{\pi(\mu_{M}) \# g_{\bt}}(X)$, where $\pi : \Pi(M) \to \Met(M)$ is the projection.
Ruberman's result implies that the counted number of the parameterized moduli space with respect to $\tilde{\si}$ coincides with the product of $\vp \cdot \calW(N)$ and the \SW invariant of $(M, \fraks_0)$ under a small perturbation, and  it is non-zero over $\Z/2$ because of \cref{cor reduction between finite dim mapping degree} and our assumption on $(M, \fraks_0)$.
Since $\si$ satisfies that $f_{i}^{\ast} \si(\bt) = \si(\reflect{\bt}{i})$ for each $i$ and $\bt \in F_{i}^{0}$, $\tilde{\si}$ can be regarded as an inductive section $\tilde{\si} : T^{n} \to \Pi(E)$ for the bundle $E$, and the moduli space parameterized on $[0,1]^{n}$ is isomorphic to the moduli space $\M_{\tilde{\si}}$ parameterized on $T^{n}$.
Here we used the vanishing of the moduli space on $\del[0,1]^{n}$.
Thus we obtain the non-triviality of the counted number of $\M_{\tilde{\si}}$ over $\Z/2$, which is nothing other than
$\frakm(s_{\vp_{e}^{\ast}\tilde{\si}}, 1; D_{e}^{n}) = 1$ in $\Z/2$.
Here $e$ is the unique $n$-cell of $T^{n}$ with respect to the given cell structure and $\vp_{e}$ is the characteristic map of $e$.
Hence $\SWcoh(E)$ is the generator of $H^{n}(T^{n};\Z/2)$.

Our remaining task is to show that the mapping degree of $\Psi \circ \si$ is $\pm1$.
To do it, for $i, j \in \{1,\ldots,n\}$ and $\bt = (t_{1}, \ldots, t_{n})$, we need to investigate $\left< \beta_{j}(\bt) \cup \gamma_{i}, [N]\right>$.
We divide our consideration into the four cases: Case~1-1, Case~1-2, Case~2-1, and Case~2-2.
They correspond to whether $t_{j} \leq 1/2$ or not and whether $i = j$ or not.
For $\bs = F_{RS}(\bt)$,
we write $\bs = (s_{1}, \ldots, s_{n})$, $\reflect{\bs}{i} = (\bar{s}_{1}, \ldots, \bar{s}_{n})$ and set $\vec{\mu}_{j}^{s_j} = (0, \ldots, \mu_{j}^{s_j}, \ldots, 0)$, $\vec{\mu}_{i}^{\bar{s}_j} = (0, \ldots, \mu_{j}^{\bar{s}_j}, \ldots, 0)$, where $\mu_{j}^{s_j}$ and $\mu_{j}^{\bar{s}_j}$ are in the $j$-th components.

Case~1-1:
Assume that $0 \leq t_{j} \leq 1/2$ and $i=j$.
We first have
\begin{align}
\begin{aligned}
\lb \beta_{i}(\bt) \cup \gamma_{i}, [N] \rb
=& \int_{\Si_{1,i}} \check{\mu}_{i}^{\bs} + \int_{\Si_{2,i}} \check{\mu}_{i}^{\bs}\\
=& \int_{\Si_{1,i}} \theta_{i,T_{i}}(\vec{\mu}_{i}^{s_{i}}) + \int_{\Si_{2,i}} \theta_{i,T_{i}}(\vec{\mu}_{i}^{s_{i}})\\
&+ \int_{\Si_{1,i}} Q^{X}_{i,T_{i}}D\theta_{i,T_{i}}(\vec{\mu}_{i}^{s_{i}})
+ \int_{\Si_{2,i}} Q^{X}_{i,T_{i}}D\theta_{i,T_{i}}(\vec{\mu}_{i}^{s_{i}})
\label{case11 eq1}
\end{aligned}
\end{align}
For the first two terms of the right-hand side of \eqref{case11 eq1},
note that
\begin{align}
\int_{\Si_{1,i}} \theta_{i,T_{i}}(\vec{\mu}_{i}^{s_{i}}) + \int_{\Si_{2,i}} \theta_{i,T_{i}}(\vec{\mu}_{i}^{s_{i}})
= \int_{\Si_{1,i}} \mu_{i}^{s_{i}} + \int_{\Si_{2,i}} \mu_{i}^{s_{i}}.
\label{case11 eq2}
\end{align}
We next consider the last two terms of the right-hand side of \eqref{case11 eq1}.
Let us take a metric $g'$ on $N$ such that a neighborhood of $\Si_{k,i}$ is isometric to $D^{2} \times \Si_{k,i}$ equipped with the product metric for $k=1,2$.
Here $\Si_{k,i}$ is regarded as $\{0\} \times \Si_{k,i}$.
Let $u$ be a coordinate of $D^{2}$.
Near a neighborhood of $\Si_{1,i}$, we have $Q^{X}_{i,T_{i}}D\theta_{i,T_{i}}(\vec{\mu}_{i}^{s_{i}}) = \mu_{i}^{s_{i}} - \check{\mu}_{i}^{\bs}$ and the right-hand side is a closed form, and so is the left-hand side.
It therefore follows that
\[
\int_{\{u\} \times \Si_{k,i}} Q^{X}_{i,T_{i}}D\theta_{i,T_{i}}(\vec{\mu}_{i}^{s_{i}})
= \int_{\Si_{k,i}} Q^{X}_{i,T_{i}}D\theta_{i,T_{i}}(\vec{\mu}_{i}^{s_{i}})
\]
for any $u \in D^{2}$.
Thus we have
\begin{align}
\begin{aligned}
\left| \int_{\Si_{k,i}} Q^{X}_{i,T_{i}}D\theta_{i,T_{i}}(\vec{\mu}_{i}^{s_{i}}) \right|^{2}
&= \left| \int_{\{u\} \times \Si_{k,i}} Q^{X}_{i,T_{i}}D\theta_{i,T_{i}}(\vec{\mu}_{i}^{s_{i}}) \right|^{2}\\
&\leq
\Area(\Si_{k,i}) \cdot \int_{\{u\} \times \Si_{k,i}}  \left|Q^{X}_{i,T_{i}}D\theta_{i,T_{i}}(\vec{\mu}_{i}^{s_{i}}) \right|^{2}.
\label{case11 eq3}
\end{aligned}
\end{align}
On the other hand, Fubini's theorem implies that
\begin{align*}
\int_{D^{2}} \int_{\{u\} \times \Si_{k,i}}  \left|Q^{X}_{i,T_{i}}D\theta_{i,T_{i}}(\vec{\mu}_{i}^{s_{i}}) \right|^{2}
=& \int_{D^{2} \times \Si_{k,i}}  \left|Q^{X}_{i,T_{i}}D\theta_{i,T_{i}}(\vec{\mu}_{i}^{s_{i}}) \right|^{2}\\
\leq& \|Q^{X}_{i,T_{i}}D\theta_{i,T_{i}}(\vec{\mu}_{i}^{s_{i}})\|^{2}_{L^{2}(N,g')}\\
\leq& C_{4}\|Q^{X}_{i,T_{i}}D\theta_{i,T_{i}}(\vec{\mu}_{i}^{s_{i}})\|^{2}_{L^{2}(N,g_{\bs})}.
\end{align*}
Here $C_{\bullet}$ denotes a constant which is independent of $T, T_{1}, \ldots, T_{n}$.
From this inequality, the inequality \eqref{case11 eq3}, and \cref{estimate for QDvp}, we have
\begin{align*}
\left| \int_{\Si_{k,i}} Q^{X}_{i,T_{i}}D\theta_{i,T_{i}}(\vec{\mu}_{i}^{s_{i}}) \right|^{2}
\leq& \frac{\Area(\Si_{k,i})}{\Area(D^{2})} \int_{D^{2}} \int_{\{u\} \times \Si_{k,i}}  \left|Q^{X}_{i,T_{i}}D\theta_{i,T_{i}}(\vec{\mu}_{i}^{s_{i}}) \right|^{2}\\
\leq& C_{5} \|Q^{X}_{i,T_{i}}D\theta_{i,T_{i}}(\vec{\mu}_{i}^{s_{i}})\|^{2}_{L^{2}(N,g_{\bt})}
\leq \frac{C_{6}}{T_{i}^{2}}.
\end{align*}
This inequality and the equalities \eqref{case11 eq1}, \eqref{case11 eq2} imply that 
\begin{align}
\left|\lb \beta_{i}(\bt) \cup \gamma_{i}, [N] \rb - \left(\int_{\Si_{1,i}} \mu_{i}^{s_{i}} + \int_{\Si_{2,i}} \mu_{i}^{s_{i}}\right) \right|
\leq \frac{C_{7}}{T_{i}}.
\label{case11 eq4}
\end{align}

Case~1-2:
Assume that $0 \leq t_{j} \leq 1/2$ and $i \neq j$.
Then we have 
\begin{align*}
\int_{\Si_{1,i}} \theta_{j,T_{j}}(\vec{\mu}_{j}^{s_{j}}) + \int_{\Si_{2,i}} \theta_{j,T_{j}}(\vec{\mu}_{j}^{s_{j}})
= 0.
\end{align*}
In addition, near a neighborhood of $\Si_{k,i}$ $(k=1,2)$, we get $Q^{X}_{j,T_{j}}D\theta_{j,T_{j}}(\vec{\mu}_{j}^{s_{j}}) = - \check{\mu}_{j}^{\bs}$, which is a closed form.
An argument which is similar to Case~1-1 therefore works, and thus we obtain
\begin{align}
\left|\lb \beta_{j}(\bt) \cup \gamma_{i}, [N] \rb \right|
\leq \frac{C_{8}}{T_{j}}.
\label{case12 eq1}
\end{align}

Case~2-1:
Assume that $1/2 < t_{j} \leq 1$ and $i = j$.
We first have
\begin{align}
\begin{aligned}
\lb \beta_{i}(\bt) \cup \gamma_{i}, [N] \rb
= (1-t_{i}) \left(\int_{\Si_{1,i}} \check{\mu}_{i}^{\bs} + \int_{\Si_{2,i}} \check{\mu}_{i}^{\bs}\right)
+  t_{i} \left(\int_{\Si_{1,i}} f_{i}^{\ast}\check{\mu}_{i}^{\reflect{\bs}{i}} + \int_{\Si_{2,i}} f_{i}^{\ast}\check{\mu}_{i}^{\reflect{\bs}{i}}\right) 
\label{case21 eq1}
\end{aligned}
\end{align}
For the first term of the right-hand side of this equation, we have shown \eqref{case11 eq4} in Case~1-1, and therefore consider the second term.
We have
\begin{align}
\begin{aligned}
\int_{\Si_{1,i}} f_{i}^{\ast}\check{\mu}_{i}^{\reflect{\bs}{i}} + \int_{\Si_{2,i}} f_{i}^{\ast}\check{\mu}_{i}^{\reflect{\bs}{i}}
=& \int_{\Si_{1,i}} f_{i}^{\ast}\theta_{i,T_{i}}(\vec{\mu}_{i}^{\bar{s}_i}) + \int_{\Si_{2,i}} f_{i}^{\ast}\theta_{i,T_{i}}(\vec{\mu}_{i}^{\bar{s}_i})\\
&+ \int_{\Si_{1,i}} Q^{X}_{i,T_{i}}D\theta_{i,T_{i}}(\vec{\mu}_{i}^{\bar{s}_i})
+ \int_{\Si_{2,i}} Q^{X}_{i,T_{i}}D\theta_{i,T_{i}}(\vec{\mu}_{i}^{\bar{s}_i})
\label{case21 eq2}
\end{aligned}
\end{align}
and
\begin{align}
\int_{\Si_{1,i}} f_{i}^{\ast}\theta_{i,T_{i}}(\vec{\mu}_{i}^{\bar{s}_i}) + \int_{\Si_{2,i}} f_{i}^{\ast}\theta_{i,T_{i}}(\vec{\mu}_{i}^{\bar{s}_i})
= \int_{\Si_{1,i}} f_{i}^{\ast} \mu_{i}^{\bar{s}_i} +  \int_{\Si_{2,i}} f_{i}^{\ast} \mu_{i}^{\bar{s}_i}.
\label{case21 eq3}
\end{align}
Let us consider the last two terms of the right-hand side of \eqref{case21 eq2}.
For $k=1,2$, we have 
$Q^{X}_{i,T_{i}}D\theta_{i,T_{i}}(\vec{\mu}_{i}^{\bar{s}_{i}}) = \mu_{i}^{\bar{s}_{i}} - \check{\mu}_{i}^{\bar{\bs}}$ near $f_{i}(\Si_{k,i})$, and hence $f_{i}^{\ast}Q^{X}_{i,T_{i}}D\theta_{i,T_{i}}(\vec{\mu}_{i}^{\bar{s}_{i}})$ is a closed from near $\Si_{k,i}$.
Therefore, as in Case~1-1, we have
\begin{align}
\begin{aligned}
\left| \int_{\Si_{k,i}} f_{i}^{\ast}Q^{X}_{i,T_{i}}D\theta_{i,T_{i}}(\vec{\mu}_{i}^{\bar{s}_{i}}) \right|^{2}
\leq& C_{9} \int_{D^{2} \times \Si_{k,i}} \left| f_{i}^{\ast}Q^{X}_{i,T_{i}}D\theta_{i,T_{i}}(\vec{\mu}_{i}^{\bar{s}_{i}}) \right|^{2}\\
\leq& C_{9} \int_{N_{i} \setminus D^{4}} \left| f_{i}^{\ast}Q^{X}_{i,T_{i}}D\theta_{i,T_{i}}(\vec{\mu}_{i}^{\bar{s}_{i}}) \right|^{2}\\
\leq& C_{9} \|f_{i}^{\ast}\|_{\op}^{2} \int_{N_{i} \setminus D^{4}} \left| Q^{X}_{i,T_{i}}D\theta_{i,T_{i}}(\vec{\mu}_{i}^{\bar{s}_{i}}) \right|^{2}\\
\leq& C_{10} \|Q^{X}_{i,T_{i}}D\theta_{i,T_{i}}(\vec{\mu}_{i}^{\bar{s}_{i}})\|^{2}_{L^{2}(N,g')}\\
\leq& C_{11} \|Q^{X}_{i,T_{i}}D\theta_{i,T_{i}}(\vec{\mu}_{i}^{\bar{s}_{i}})\|^{2}_{L^{2}(N,\bar{g}_{\bs})}
\leq \frac{C_{12}}{T_{i}^{2}}.
\label{case21 eq4}
\end{aligned}
\end{align}
Here we used the decomposition $N = (N_{i} \setminus D^{4}) \sqcup (\#_{j \neq i} N_{j}) \setminus D^{4}$ taken so that the all cylinders of length $2T$ are contained in $(\#_{j \neq i} N_{j}) \setminus D^{4}$, and $\|f_{i}^{\ast}\|_{\op}$ is the operator norm for the operator acting on $L^{2}(N_{i} \setminus D^{4})$.
Since $N_{i} \setminus D^{4}$ contains no cylindrical part, $\|f_{i}^{\ast}\|_{\op}$ is independent of $T, T_{1}, \ldots, T_{n}$, and so it can be absorbed into $C_{10}$.
From the equalities \eqref{case21 eq1}, \eqref{case21 eq2}, \eqref{case21 eq3} and the inequalities \eqref{case11 eq4}, \eqref{case21 eq4}, it follows that
\begin{align}
\begin{aligned}
&\left| \lb \beta_{i}(\bt) \cup \gamma_{i}, [N] \rb
- (1-t_{i})\left(\int_{\Si_{1,i}} \mu_{i}^{s_{i}} + \int_{\Si_{2,i}} \mu_{i}^{s_{i}}\right)
- t_{i} \left(\int_{\Si_{1,i}} f_{i}^{\ast} \mu_{i}^{\bar{s}_i} +  \int_{\Si_{2,i}} f_{i}^{\ast} \mu_{i}^{\bar{s}_i} \right)
\right|\\
&\leq \frac{C_{7}}{T_{i}} + \frac{\sqrt{C_{12}}}{T_{i}} =  \frac{C_{13}}{T_{i}}.
\label{case21 eq5}
\end{aligned}
\end{align}

Case~2-2:
Assume that $1/2 < t_{j} \leq 1$ and $i \neq j$.
We first have
\begin{align}
\begin{aligned}
\lb \beta_{i}(\bt) \cup \gamma_{i}, [N] \rb
= (1-t_{j}) \left(\int_{\Si_{1,i}} \check{\mu}_{j}^{\bs} + \int_{\Si_{2,i}} \check{\mu}_{j}^{\bs}\right)
+  t_{j} \left(\int_{\Si_{1,i}} f_{j}^{\ast}\check{\mu}_{j}^{\reflect{\bs}{j}} + \int_{\Si_{2,i}} f_{j}^{\ast}\check{\mu}_{j}^{\reflect{\bs}{j}}\right).
\label{case22 eq1}
\end{aligned}
\end{align}
For the first term of the right-hand side of this equation, as Case~1-2, we obtain
\begin{align}
\left|\int_{\Si_{1,i}} \check{\mu}_{j}^{\bs} + \int_{\Si_{2,i}} \check{\mu}_{j}^{\bs}\right|
\leq \frac{C_{14}}{T_{j}}.
\label{case22 eq2}
\end{align}
For the second term, note that
\[
\int_{\Si_{k,i}} f_{j}^{\ast}\check{\mu}_{j}^{\reflect{\bs}{j}}
= \int_{\Si_{k,i}} \check{\mu}_{j}^{\reflect{\bs}{j}}
\]
holds for $k=1,2$.
Therefore, as Case~1-2 again, we obtain
\begin{align}
\left|\int_{\Si_{1,i}} f_{j}^{\ast}\check{\mu}_{j}^{\reflect{\bs}{j}} + \int_{\Si_{2,i}} f_{j}^{\ast}\check{\mu}_{j}^{\reflect{\bs}{j}}\right|
\leq \frac{C_{15}}{T_{j}}.
\label{case22 eq3}
\end{align}
From the equality \eqref{case22 eq1} and the inequalities \eqref{case22 eq2}, \eqref{case22 eq3}, it follows that
\begin{align}
|\lb \beta_{j}(\bt) \cup \gamma_{i}, [N] \rb|
\leq \frac{C_{16}}{T_{j}}.
\label{case22 eq3}
\end{align}

We have now completed investigating all four cases.
For $i \in \{1, \ldots, n\}$ and $\bt \in F_{i}^{0}$, using \eqref{case11 eq4} in Case~1-1, \eqref{case12 eq1} in Case~1-2, and \eqref{case22 eq3} in Case~2-2, we have
\begin{align}
\begin{aligned}
&\left|\lb\si(\bt) \cup \gamma_{i}, [N]\rb
- \left(\int_{\Si_{1,i}} \mu_{i}^{0} + \int_{\Si_{2,i}} \mu_{i}^{0}\right) 
\right|\\
\leq& \left|\lb \beta_{i}(\bt) \cup \gamma_{i}, [N]\rb
- \left(\int_{\Si_{1,i}} \mu_{i}^{0} + \int_{\Si_{2,i}} \mu_{i}^{0}\right) 
\right|
+ \sum_{j \neq i} |\lb \beta_{j}(\bt) \cup \gamma_{i}, [N]\rb|\\
\leq& \frac{C_{7}}{T_{i}} + \sum_{\substack{ j \neq i \\ 0 \leq t_{j} \leq 1/2 }} \frac{C_{8}}{T_{j}} + \sum_{\substack{ j \neq i \\ 1/2 < t_{j} \leq 1 }}\frac{C_{16}}{T_{j}}
\leq \frac{C_{17}}{\min\{T_{j} \mid 1 \leq j \leq n\}}.
\label{main estimate for Fi0}
\end{aligned}
\end{align}
Similarly, for $\bt \in F_{i}^{1}$, using \eqref{case21 eq5} in Case~2-1, and \eqref{case12 eq1} in Case~1-2, \eqref{case22 eq3} in Case~2-2, we have
\begin{align}
\begin{aligned}
&\left|\lb\si(\bt) \cup \gamma_{i}, [N]\rb
- \left(\int_{\Si_{1,i}} f_{i}^{\ast} \mu_{i}^{0} +  \int_{\Si_{2,i}} f_{i}^{\ast} \mu_{i}^{0} \right)
\right|\\
\leq& \frac{C_{13}}{T_{i}} + \sum_{\substack{ j \neq i \\ 0 \leq t_{j} \leq 1/2 }} \frac{C_{8}}{T_{j}} + \sum_{\substack{ j \neq i \\ 1/2 < t_{j} \leq 1 }}\frac{C_{16}}{T_{j}}
\leq \frac{C_{18}}{\min\{T_{j} \mid 1 \leq j \leq n\}}.
\label{main estimate for Fi1}
\end{aligned}
\end{align}
We can regard $\gamma_{i}$ as a cohomology class on $\hat{N}_{i}$.
Since $f_{i}^{\ast}\gamma_{i} = -\gamma_{i}$ holds, we have
\begin{align}
\int_{\Si_{1,i}} f_{i}^{\ast}\mu_{i}^{0} + \int_{\Si_{2,i}} f_{i}^{\ast}\mu_{i}^{0}
= \int_{\hat{N}_{i}} f_{i}^{\ast}[\mu_{i}^{0}] \cup \gamma_{i}
= -\int_{\hat{N}_{i}} [\mu_{i}^{0}] \cup \gamma_{i}
= - \left( \int_{\Si_{1,i}} \mu_{i}^{0} + \int_{\Si_{2,i}} \mu_{i}^{0} \right)
\label{eq for fmu and mu}
\end{align}
Since $\Ker{D_{i}^{\hat{h}_{i}(0)}} \cap \Omega^{2}(N)$ is isomorphic to $H^{+}(\hat{N}_{i})$ via Hodge theory, $[\mu_{i}^{0}]$ is of positive self-intersection, and hence the quantity in \eqref{eq for fmu and mu} is non-zero.
Therefore, if we take $T, T_{1}, \ldots, T_{n}$ to be sufficiently large, the inequalities \eqref{main estimate for Fi0}, \eqref{main estimate for Fi1} and the equality \eqref{eq for fmu and mu} imply that 
$\lb\si(\bt) \cup \gamma_{i}, [N]\rb$ for $\bt \in F_{i}^{0}$ and that for $\bt \in F_{i}^{1}$ have different signs.
This means that $\Psi \circ \si(F_{i}^{0})$ and $\Psi \circ \si(F_{i}^{1})$ are contained in distinct connected componect of $\R^{n} \setminus (\R \times \cdots \times \{0\} \times \cdots \times \R)$, where $\{0\}$ is the $i$-th component.
This implies that the mapping degree of $\Psi \circ \si|_{\del[0,1]^{n}} : \del[0,1]^{n} \to \R^{n} \setminus\{0\}$ is $1$ up to sign, and this coincides with that of $\Psi \circ \si : ([0,1]^{n}, \del[0,1]^{n}) \to (\R^{n}, \R^{n} \setminus \{0\})$.
This completes the proof of \cref{theo on calculation of SW for connected sum with spin}.
\end{proof}

\begin{rem}
Since the diffeomorphisms $f_{i}$ used in \cref{theo on calculation of SW for connected sum with spin} reverse a given homology orientation $\calO$ of $X$, we cannot give a non-trivial element of $H^{\ast}(B\Diff(X, \fraks, \calO);\Z)$ using these diffeomorphisms.
At this stage the author does not know how to give a cohomology class of higher degree over $\Z$ using $\SWcoh$.
On the other hand, thanks to D.~Ruberman's example in \cite{MR1671187}, we can give a non-trivial element of $H^{1}(B\Diff(X, \frakP, \calO);\Z)$ for a suitable $(X, \frakP)$ using $\Dcoh$, described in \cref{subsection Other calculations}.
\end{rem}

\subsection{Behavior under the composition 1}
\label{subsection Cohomologically trivial bundles}

To give subtle examples of calculations for $\SWcoh$, we combine an argument of \cite{Konno3} by the author and one of D.~Ruberman~\cite{MR1671187}.
Ruberman has defined invariants of diffeomorphisms on  an oriented closed $4$-manifold $X$ in \cite{MR1671187} using $1$-parameter families of $SO(3)$-ASD equations and \SW equations.
In \cite{MR1671187} he also has given examples of diffeomorphisms on some $4$-manifold which are topologically isotopic to the identity map, but not smoothly isotopic to it.
To prove the former property, he considered the composition of two diffeomorphisms having mutually converse actions on the homology groups, and to prove the latter property,  he used the invariant based on $SO(3)$-Yang--Mills ASD equations.
His family corresponding to a given diffeomorphism can be regarded as a family of ASD equations parameterized on $S^{1}$ via the mapping torus construction for the diffeomorphism.
In this \lcnamecref{subsection Cohomologically trivial bundles} we consider a \SW version and higher-dimensional parameter version of Ruberman's argument.
We give  a bundle of a $4$-manifold $X$ obtained as the mapping torus of a tuple of commutative diffeomorphisms belonging to $\Diff(X, \fraks)$ for (the isomorphism class of) a $\spc$ structure $\fraks$ on $X$ such that each diffeomorphism is topologically isotopic to the identity, but the bundle is non-trivial as $\Diff(X, \fraks)$-bundle.
In this \lcnamecref{subsection Cohomologically trivial bundles} we consider an example of a bundle with non-spin fiber, and we shall do one with spin fiber in \cref{subsection Combination with Ruberman's argument: spin case} using an argument of \cref{subsection Calculation of SWcoh for connected sum with n(S2S^2)}.

Fix $n>0$ and let $M_{0}$ be an oriented closed smooth $4$-manifold with $b^{+}(M_{0}) \geq 2$ and $\fraks_{0}$ be the isomorpshim class of a $\spc$ structure on $M_{0}$.
Assume that the \SW invariant of $(M, \fraks_{0})$ is an odd number and that $M_{0} \# \CP^{2}$ is diffeomorphic to $M_{1} \# \CP^{2}$, where $M_{1} = b^{+}(M)\CP^{2} \# b^{-}(M)(-\CP^{2})$.
(Typically we can take a simply connected elliptic surface and the complex structure of it as $(M_{0}, \fraks_{0})$.)
Let us denote $H_{i}$ the projective line embedded in the right connected summand of $M_{i} \# \CP^{2}$.
As noted in Ruberman~\cite{MR1671187}, we can find a diffeomorphism $\vp : M_{0} \# \CP^{2} \to M_{1} \# \CP^{2}$ so that $\vp(H_{1})$ is homologous to $H_{2}$.
The induced map $\vp^{\ast} : H^{2}(M_{1}\#\CP^{2}) \to H^{2}(M_{0}\#\CP^{2})$ acts as identity on $H^{2}(\CP^{2})$.
We can therefore define the isomorphism class $\fraks_{1}$ of a $\spc$ structure on $M_{1}$ by  $\fraks_{1} := (\vp^{-1})^{\ast}\fraks_{0}$.
Extending by the identity, we get a diffeomorphism
$\vp : M_{0} \# \CP^{2}\#2(-\CP^{2}) \to M_{1} \# \CP^{2}\#2(-\CP^{2})$.
Let $\frakt_{0}$ be the isomorphism class of a spin$^{c}$ structure on $\CP^{2} \# 2(-\CP^{2}) = \CP^{2} \# (-\CP^{2}_{1}) \# (-\CP^{2}_{2})$ such that the $H^{2}(\CP^{2})$-component of $c_{1}(\frakt)$ gives a generator of $H^{2}(\CP^{2})$ and $j$-th component does of $H^{2}(-\CP^{2}_{j})$ for each $j=1,2$.
Set 
\[
X := M_{0} \# n(\CP^{2}\#2(-\CP^{2}))
= M_{0} \# N,
\]
where $N := \#_{i=1}^{n}N_{i}$ and $N_{i} := \CP^{2}\#2(-\CP^{2})$.
Let $\vp_{i} : M_{0} \# N_{i} \to M_{1} \# N_{i}$ be the copy of $\vp$, and we write $\vp_{i} : X \to M_{1}\# N$ also for the extension by the identity.
Let $\frakt$ be the isomorphism class of a spin$^{c}$ structure on $N$ defined by $\frakt = \#_{i=1}^{n} t_{i}$, where $\frakt_{i}$ is the copy of $\frakt_{0}$.
We define the isomorphism class $\fraks$ of a spin$^{c}$ structure on $X$  by
\begin{align}
\fraks := \fraks_{0} \# \frakt.
\label{eq spinc str defined by connected sum}
\end{align}
Then one can easily see that $d(\fraks)=-n$.
Let $f_{1,0}, \ldots, f_{n,0} \in \Diff(X, \fraks)$ be the diffeomorphisms given in Theorem 3.2 in \cite{Konno3}.
(These diffeomorphisms are written as $f_{1}, \ldots, f_{n}$ there.)
Each $f_{i,0}$ is identity on $M_{0} \#(\#_{i' \neq i} N_{i'})$, and
hence they are mutually commutative.
We also note that $f_{i,0}$ is obtained as the copy of a common diffeomorphism on $M_{0} \# \CP^{2} \# 2(-\CP^{2})$ for any $i$.
Let $f_{1,1}', \ldots, f_{n,1}' \in \Diff^{+}(M_{1}\#N)$ be the diffeomorphisms defined by substituting $M_{1}$ for $M_{0}$ in Theorem 3.2 in \cite{Konno3}, and let $f_{1,1}, \ldots, f_{n,1} \in \Diff^{+}(X)$ be the diffeomorphisms defined by $f_{i,1} := \vp_{i}^{-1} \circ f_{i,1}' \circ \vp_{i}$.
Then $f_{1,1}, \ldots, f_{n,1}$ are mutually commutative and belong to $\Diff(X, \fraks)$.
We also note that, if $i \neq i'$, then $f_{i,0}$ and $f_{i', 1}$ are also commutative.
As Ruberman~\cite{MR1671187}, by the construction of $\vp$ and $f_{i,0}$, $f_{i, 1}$, we have that $f_{i,0}$ is homotopic to $f_{i,1}$ for each $i$.
The diffeomorphism $f_{i} := f_{i,0} \circ f_{i,1}^{-1} \in \Diff(X, \fraks)$ is therefore homotopic to the identity.
In fact the result due to F.~Quinn~\cite{MR868975} implies that, more strongly, $f_{i}$ is topologically isotopic to the identity.

We here recall an invariant of $n$-tuples of commutative diffeomorphisms on an oriented closed $4$-manifold defined in \cite{Konno3} by the author.
This is a generalization of Ruberman's invariant given in~\cite{MR1671187} emerging from $1$-parameter families of \SW equations.
This invariant of commutative diffeomorphisms relates to $\SWcoh$ as follows.
For mutually commutative diffeomorphisms $f_{1}, \ldots, f_{n} \in \Diff(X, \fraks)$, say ones defined above, one can associate a number 
\begin{align}
\SWinv(f_{1}, \ldots, f_{n};\fraks) \in \Z \text{ or } \Z/2.
\label{eq def inv of commutative diffeomorphisms}
\end{align}
This is defined by counting the moduli space of families of \SW equations parameterized on $[0,1]^{n}$, and it can be interpreted as a counted number of a parameterized moduli space on $T^{n}$.
(See Example 2.6 in \cite{Konno3}.)
This number is nothing but the counted number of $\M_{\si}$ for some inductive section $\si : T^{n} \to \Pi(E)$ for the bundle $X \to E \to T^{n}$ defined as the mapping torus of $f_{1}, \ldots, f_{n}$.
Thus we have
\begin{align}
\SWinv(f_{1}, \ldots, f_{n};\fraks) = \pm\left<\SWcoh(E), [T^{n}]\right>.
\label{eq between inv of tuple of diff and SWcoh}
\end{align}
(If we work over $\Z/2$, of course, the plus-minus sign in the right-hand side is omitted.)

We now come back to our specific diffeomorphisms.
Since $f_{i,1}$ and $f_{i,1}$ reverse a homology orientation, we work on $\Z/2$, though $\SWinv(f_1, \ldots, f_n; \fraks)$ is defined over $\Z$.
As Lemma 2.6 Ruberman~\cite{MR1671187}, we immediately see that
\begin{align}
\begin{aligned}
\SWinv(f_1, \ldots, f_n; \fraks)
&= \sum_{0 \leq j_{1}, \ldots, j_{n} \leq 1}\SWinv(f_{1,j_{1}}^{\epsilon_{1}}, \ldots, f_{n,j_{n}}^{\epsilon_{n}}; \fraks)\\
&= \sum_{0 \leq j_{1}, \ldots, j_{n} \leq 1} \SWinv(f_{1,j_{1}}, \ldots, f_{n,j_{n}}; \fraks)
\label{eq for SW for comm diffeos}
\end{aligned}
\end{align}
over $\Z/2$, where $\epsilon_{i} \in \{1, -1\}$ is defined by $j_{i} = 1 \Leftrightarrow \epsilon_{i} = -1$, and $f_{i,0}^{1} := f_{i,0}$.
Let $j_{1}, \ldots, j_{n} \in \{0,1\}$ and $j_{1}', \ldots, j_{n}' \in \{0,1\}$.
Then, if we have $\#\Set{i | j_{i}=0} = \#\Set{i | j_{i}'=0}$, it follows that
\begin{align}
\SWinv(f_{1,j_{1}}, \ldots, f_{n,j_{n}}; \fraks)
= \SWinv(f_{1,j_{1}'}, \ldots, f_{n,j_{n}'}; \fraks)
\label{eq for SW for comm diffeos2}
\end{align}
because of symmetry.
Let us assume that $n$ can be written as $n = 2^{N}$ for some $N \geq 0$.
Then the binomial coefficient $\binom{n}{k}$ is even for any $k \in \{1, \ldots, n-1\}$, and it therefore follows from the equalities \eqref{eq for SW for comm diffeos} and \eqref{eq for SW for comm diffeos2} that
\[
\SWinv(f_1, \ldots, f_n; \fraks)
= \SWinv(f_{1,0}, \ldots, f_{n,0}; \fraks) + \SWinv(f_{1,1}, \ldots, f_{n,1}; \fraks)
\]
over $\Z/2$.
In addition, Theorem 3.2 in \cite{Konno3} implies that 
\begin{align}
\SWinv(f_{1,j}, \ldots, f_{n,j}; \fraks) 
= \SWinv(M_{j}, \fraks_{j}) = j+1
\label{eq calculation from Konno3}
\end{align}
over $\Z/2$ for each $j=0,1$, and thus we have $\SWinv(f_1, \ldots, f_n; \fraks)
= 1$.
We can therefore deduce that $\SWcoh(E) \neq 0$ in $H^{n}(T^{n};\Z/2)$ from \eqref{eq between inv of tuple of diff and SWcoh}.
Thus we have:

\begin{theo}
\label{theo cohomologically trivial bundle}
Let $N \geq 0$ and set $n=2^{N}$.
Let $(X, \fraks)$ and $X \to E \to T^{n}$ be as above.
Then $\SWcoh(E) \neq 0$ holds , in particular $E$ is a non-trivial $\Diff(X, \fraks)$-bundle.
\end{theo}

\begin{rem}
\label{first rem topologically trivial but smoothly non trivial}
We note that, in the case of $n=1$ above, since $f_{1}$ is topologically isomorphic to the identity, $X \to E \to S^{1}$ is trivial as a $\Homeo_{0}(X)$-bundle, and hence also as a $\Homeo(X, \fraks)$-bundle.
Here $\Homeo_{0}(X)$ is the identity component of the whole group of orientation preserving homeomorphisms on $X$.
\end{rem}

\begin{rem}
Although a result of F.~Quinn~\cite{MR868975} gives topological isotopies between $f_{1}, \ldots, f_{n}$ and the identity, the supports of the topological isotopies are not disjoint.
The topological isotopies are therefore not mutually commutative in general, and we cannot conclude that $E$ is trivial as a $\Homeo(X, \fraks)$-bundle unless $n=1$ from \cite{MR868975}.
The author does not know whether $E$ is trivial as $\Homeo(X, \fraks)$-bundle when $n>1$.
\end{rem}

\begin{ex}
\label{ex comp between fiberwise connected sum and connected sum concrete}
If we substitute $K3$ for $M_{0}$ in the above argument, we have
\begin{align}
X \cong (n+3)\CP^{2} \# (2n+19)(-\CP^{2}).
\label{eq diff betwteen X and CP -CP}
\end{align}
In this case we can take the (isomorphism class of a) $\spc$ structure coming from the complex structure of $K3$ as $\fraks_{0}$.
Let $\fraks$ be the (isomorphism class of a) $\spc$ structure on $X$ defined by \eqref{eq spinc str defined by connected sum}.
From the diffeomorphism \eqref{eq diff betwteen X and CP -CP} and \cref{theo cohomologically trivial bundle}, if $n$ can be written as $n = 2^{N}$ for some $N \geq 0$, we get a bundle
\[
(n+3)\CP^{2} \# (2n+19)(-\CP^{2}) \to E \to T^{n}
\]
with structure group $\Diff(X, \fraks)$ satisfying $\SWcoh(E) \neq 0$.
This is therefore non-trivial as a $\Diff(X, \fraks)$-bundle, and trivial as a $\Homeo(X, \fraks)$-bundle if $n=1$, explained in \cref{first rem topologically trivial but smoothly non trivial}.
\end{ex}

We here note that, to show only the non-triviality of our characteristic classes, one can directly use author's result \cite{Konno3}.
As explained in \eqref{eq calculation from Konno3},
Theorem 3.2 in \cite{Konno3} gives diffeomorphisms whose invariant explained in \eqref{eq def inv of commutative diffeomorphisms} is non-trivial over $\Z/2$, and thus we obtain:

\begin{theo}
\label{non-vanishing theorem non-spin}
Let $M$ be an oriented closed smooth $4$-manifold with $b^+(M) \geq 2$ and $\fraks_0$ be the isomorphism class of a spin$^{c}$ structure on $M$ such that the \SW invariant of $(M, \fraks_0)$ is an odd number.
Let $n$ be a positive integer.
Define $N = n\CP^2 \# m(-\CP^2) = \#_{i=1}^n \CP^2_{i} \# (\#_{j=1}^m (-\CP^2_{j}))$ for some $m \geq 2n$ and $\frakt$ is the isomorphism class of a spin$^{c}$ structure on $N$ such that the $i$-th component of $c_{1}(\frakt)$ gives a generator of $H^{2}(\CP^{2}_{i})$ and $j$-th component does of $H^{2}(-\CP^{2}_{j})$.
Then, for the pair defined by
\[
(X, \fraks) := (M \# N, \fraks_0\#\frakt),
\]
we have
\[
\SWcoh(X,\fraks) \neq 0 \text{ in } H^{n}(B\Diff(X,\fraks);\Z/2).
\]
\end{theo}

\begin{proof}
Let $X \to E \to T^{n}$ be the mapping torus of the diffeomorphisms $f_{1}, \ldots, f_{n}$ given in Theorem 3.2 in \cite{Konno3}.
Then \eqref{eq between inv of tuple of diff and SWcoh} and Theorem 3.2 in \cite{Konno3} imply that $\left<\SWcoh(E), [T^{n}]\right>= \SWinv(f_{1}, \ldots, f_{n};\fraks) \neq 0$ in $\Z/2$.
The assertion in the \lcnamecref{non-vanishing theorem non-spin} therefore follows from functoriality (\cref{part of theo of well-definedness of the invariant functoriality}).
\end{proof}

\subsection{Behavior under the composition 2}
\label{subsection Combination with Ruberman's argument: spin case}

We here discuss a spin analogue of \cref{subsection Cohomologically trivial bundles}:
we consider $S^{2} \times S^{2}$ in this \lcnamecref{subsection Combination with Ruberman's argument: spin case} instead of $\CP^{2}\#2(-\CP^{2})$ in \cref{subsection Cohomologically trivial bundles}.
Fix $n>0$ and let $M_{0}$ be an oriented closed spin smooth $4$-manifold such that $\sign(M_{0})=0$ (hence $M_{0}$ has even $b_{2}$) and $b^{+}(M_{0}) \geq 2$, and let $\fraks_{0}$ be the isomorpshim class of a $\spc$ structure on $M_{0}$.
Assume that the \SW invariant of $(M, \fraks_{0})$ is an odd number and that $M_{0} \# S^{2} \times S^{2}$ is diffeomorphic to $M_{1} \# S^{2} \times S^{2}$, where $M_{1} = \frac{b_{2}(M_{0})}{2}(S^{2} \times S^{2})$.
Set
\begin{align*}
X : = M_{0} \# n(S^{2} \times S^{2}).
\end{align*}
Let $\fraks$ be the isomorphism class of a $\spc$ structure on $X$ defined by $\fraks := \fraks_{0} \# \frakt$, where $\frakt$ is the isomorphism class of a $\spc$ structure coming from a spin structure on $n(S^{2} \times S^{2})$.
Recall that $f_{1,j}, \ldots, f_{n,j}$ in \cref{subsection Cohomologically trivial bundles} are given as the extension of the $n$-tuple of diffeomorphisms on $n(\CP^{2}\#2(-\CP^{2}))$ given in \cite{Konno3} by the identity on $M_{i}$.
If we use the diffeomorphisms $f_{1}, \ldots, f_{n}$ on $n(S^{2}\times S^{2})$ in \cref{subsection Calculation of SWcoh for connected sum with n(S2S^2)} instead of this $n$-tuple of diffeomorphisms on $n(\CP^{2}\#2(-\CP^{2}))$, we get commutative diffeomorphisms  $f_{1,j}, \ldots, f_{n,j} \in \Diff(X, \fraks)$ for $j=0,1$.
In this \lcnamecref{subsection Combination with Ruberman's argument: spin case} we use the notation $f_{1}, \ldots, f_{n}$ for the diffeomorphisms  defined as $f_{i,0} \circ f_{i,1}^{-1}$.
Each $f_{i}$ is topologically isotopic to the identity as in \cref{subsection Cohomologically trivial bundles}.
Let $X \to E \to T^{n}$ be the mapping torus of $f_{1}, \ldots, f_{n}$.
Instead of Theorem 3.2 in \cite{Konno3} used in \cref{subsection Cohomologically trivial bundles}, we can use calculations given in \cref{subsection Calculation of SWcoh for connected sum with n(S2S^2)}.
More precisely, in the proof of  \cref{theo on calculation of SW for connected sum with spin}, we have shown that
\[
\SWinv(f_{1,j}, \ldots, f_{n,j};\fraks) = j+1
\]
over $\Z/2$.
We can therefore deduce the following \lcnamecref{theo non vanishing spin subtle} by  the same argument of \cref{subsection Cohomologically trivial bundles}:

\begin{theo}
\label{theo non vanishing spin subtle}
Let $N \geq 0$ and set $n=2^{N}$.
Let $(X, \fraks)$ and $X \to E \to T^{n}$ be as above.
Then $\SWcoh(E) \neq 0$ holds , in particular $E$ is a non-trivial $\Diff(X, \fraks)$-bundle.
\end{theo}

\begin{rem}
By the same reason described in \cref{first rem topologically trivial but smoothly non trivial},  in the case of $n=1$, the bundle  $X \to E \to S^{1}$ is trivial as a $\Homeo(X, \fraks)$-bundle.
\end{rem}

\begin{ex}
\label{ex final is connected sum of SS}
We here use the $4$-manifold $X(p,q)$ in \cref{ex final is connected sum of SS before}.
We take $p,q$ so that $X(p,q) \# S^{2} \times S^{2} \cong m(S^{2}\times S^{2})$ for some $m>0$.
In the above argument,
let us substitute $X(p,q)$ for $M_{0}$, and as $\fraks_{0}$ take the (isomorphism class of a) $\spc$ structure coming from the symplectic structure of $X(p,q)$.
From the dissolving of $X(p,q) \# S^{2} \times S^{2}$ and \cref{theo non vanishing spin subtle},
if $n$ can be written as $n = 2^{N}$ for some $N \geq 0$, we get a bundle
\[
k(S^{2}\times S^{2}) \to E \to T^{n}
\]
with structure group $\Diff(X, \fraks)$ satisfying $\SWcoh(E) \neq 0$, where $k=m+n-1$.
This is therefore non-trivial as a $\Diff(X, \fraks)$-bundle, and trivial as a $\Homeo(X, \fraks)$-bundle if $n=1$.
\end{ex}

\begin{rem}
As remarked in \cref{rem Hanke Kotschick Wehrheim and Wall}, 
 it seems difficult to determine $k$'s in \cref{ex final is connected sum of SS}.
\end{rem}

\subsection{Ruberman's calculation and $\Dcoh$}
\label{subsection Other calculations}

We mention a calculation of $\Dcoh$ obtained from Ruberman's one given in \cite{MR1671187}.
In Theorem 3.1 in \cite{MR1671187}, he has constructed a diffeomorphism preserving the isomorphism class of an $SO(3)$-bundle on some $4$-manifold $X$ such that his invariant of the diffeomorphism based on a $1$-parameter family of ASD equations does not vanish.
This invariant coincides with the counted number of the parametrized moduli space of ASD equations for the family $X \to E \to S^{1}$ defined as the mapping torus, as in \cref{subsection Cohomologically trivial bundles,subsection Combination with Ruberman's argument: spin case}.
Thus we have:

\begin{theo}
\label{cal of Dcoh via Ruberman}
Let $M$ be an oriented closed smooth $4$-manifold with $b^+(M) \geq 2$ and $\frakP_{0}$ be the isomorphism class of an $SO(3)$-bundle on $M$ such that the formal dimension for $\frakP_{0}$ is zero and $w_{2}(\frakP_{0}) \neq 0$, and that the Donaldson invariant of $(M, \frakP_{0})$ does not vanish (over $\Z$).
Set 
\[
X = M \# \CP^{2} \# 2(-\CP^{2}).
\]
Let $L \to \CP^{2} \# 2(-\CP^{2})$ be the complex line bundle such that the $H^{2}(\CP^{2})$-component of $c_{1}(L)$ is a generator of $H^{2}(\CP^{2})$, and similarly the $H^{2}(-\CP^{2}_{j})$-component of $c_{1}(L)$  is a generator of $H^{2}(-\CP^{2}_{j})$ for $j=1,2$.
We define $\frakP$ as the isomorphism class of the $SO(3)$-bundle obtained by gluing (a representative of) $\frakP_{0}$ with $L \oplus \underline{\R}$, where $\underline{\R}$ is the trivial real line bundle.
Then, for a fixed a homology orientation $\calO$ on $X$, we have
\begin{align*}
\Dcoh(X, \frakP, \calO) \neq 0\ {\text in }\ H^{1}(B\Diff(X, \frakP, \calO);\Z).
\end{align*}
\end{theo}

\begin{proof}
Theorem 3.1 in Ruberman~\cite{MR1671187} asserts that there is a diffeomorphism on $X$ such that it preserves $\frakP$ and $\calO$ and that his invariant of the diffeomorphism coincides with the $-4$ times the Donaldson invariant of $(M, \frakP_{0})$.
The assertion in the \lcnamecref{cal of Dcoh via Ruberman} therefore follows as in the proof of \cref{non-vanishing theorem non-spin}.
\end{proof}

\begin{rem}
At this stage the author does not have a non-trivial example of $\Dcoh$ for more higher degree $n>1$.
In the argument of \cref{section Non triviality}, the difference on computability for $\SWcoh$ and for $\Dcoh$ arises from the difference between the structure of wall in both theories:
the wall in Donaldson theory is more complicated rather than that of \SW theory.
To generalize \cref{cal of Dcoh via Ruberman} for $n>0$, one has to study the structure of the wall in Donaldson theory for $b^{+}>0$.
\end{rem}

\section{Concluding remarks}
\label{section Concluding remarks}

Finally, we note some related work which we have not mention and also note further potential developments of our characteristic classes $\Dcoh$ and $\SWcoh$.

\begin{rem}
P.~Kronheimer~\cite{Kronheimer1} has defined an invariant for families of \SW equations.
More precisely, he has considered a family parameterized on a singular chain in the space of perturbations such that the parameterized moduli space vanishes on the boundary.
Namely, his invariant is a version of {\it relative} invariants for families.
In symplectic geometry, several people have considered relative Gromov--Witten invariants using Kronheimer's work.
See O.~Bu\c se~\cite{MR2218350} and T.~Nishinou~\cite{MR1967225}.
In our context, we can define a similar relative version of our characteristic classes as follows.
Let $X \to E \to B$ a bundle of an oriented closed smooth $4$-manifold $X$ on a CW complex $B$.
Let us choose either the ASD setting or the SW setting, and suppose that the formal dimension is $-n$ for $n>0$.
Fix an inductive section $\si : B^{(n)} \to \Pi(E)|_{B^{(n)}}$.
Let $B' \subset B$ be a subspace of $B$ such that the parameterized moduli space with respect to $\si$ vanishes on $B'$.
Set $\tau := \si|_{B' \cap B^{(n)}}$.
Then we can define $\Acoh(E, B', \tau) \in H^{n}(B, B')$ as $\Acoh(E)$.
The relative cohomology class $\Acoh(E, B', \tau)$ depends only on $E$, $B'$, and $\tau$, i.e., for another inductive section $\si'$ satisfying $\si'|_{B' \cap B^{(n)}} = \tau$, we can get the same cohomology class.
\end{rem}

\begin{rem}
\label{rem on cohomotopy version}

M.~Szymik~\cite{MR2652709} has considered a family version of the Bauer--Furuta invariant~\cite{MR2025298}.
As described in the introduction, one of the big differences between his setting and ours is the structure group of families.
The author expects that a Bauer--Furuta-type refinement can be also considered in our setting: families with structure group $\Diff(X,\fraks)$ or $\Diff(X, \fraks, \calO)$.
To establish such a refinement in full generality, we need some stacks, which will be discussed in a subsequent paper.
(See also \cref{rem stack}.)
\end{rem}

\begin{rem}
A possibility of a generalization of our theory is to construct a version of Floer theory for $\Dcoh$ and $\SWcoh$.
If we try to do it, we are faced with a serious analytic problem:
if we consider a family of $3$-manifolds, the differential of the Floer chain complexes varies, and in addition, critical points also do in the \SW situation.
One might hope that we shall overcome such a problem by introducing techniques developed in the context of symplectic geometry.
See, for example, K.~Fukaya~\cite{MR1953352} and M.~Abouzaid~\cite{Abouzaid,MR3656481}.
\end{rem}

\end{document}